\ifx \labelsONmargin\undefined

\ifx\RequirePackage\undefined
\expandafter\ifx\csname amsart.sty\endcsname\relax
\documentstyle[12pt,amssymb,amscd]{amsart}
\fi

\def\atSign{@@}

\def\mathbb{\Bbb}
\def\mathfrak{\frak}
\def\mathbf{\bold}
\ifx\boldsymbol\undefined	
\def\boldsymbol#1{{\bold #1}}
\fi

\def\mathbit{\boldsymbol}

\makeatletter
\newenvironment{proof}{%
	\@ifnextchar[{%
		\expandafter\let\expandafter\end@proof
		\csname endpf*\endcsname
		\my@proof
	}{\let\end@proof\endpf\pf}%
}{\end@proof}
\def\my@proof[#1]{\@nameuse{pf*}{#1}}
\makeatother

\def\xrightarrow[#1]#2{@>{#2}>{#1}>}
\def\xleftarrow[#1]#2{@<{#2}<{#1}<}

\def\providecommand#1{\def#1}
\def\emph#1{{\em #1}}
\def\textbf#1{{\bf #1}}

\def\mathring{\overset{\,\,{}_\circ}}
\else
\ifx\usepackage\RequirePackage
\else
\documentclass[12pt]{amsart}
\usepackage{amssymb,amscd}
\fi
\ifx \mathring\undefined		
\DeclareMathAccent{\mathring}{\mathalpha}{operators}{"17}
\fi

\long\def\FAKEendPROOF{\endtrivlist}
\ifx\endproof\FAKEendPROOF
\def\endproof{\qed\endtrivlist}
\fi

\ifx\mathbit\undefined	
\DeclareMathAlphabet{\mathbit}{OML}{cmm}{b}{it}
\fi

\def\atSign{@}

\advance\oddsidemargin -0.1\textwidth
\evensidemargin\oddsidemargin
\def\Sb#1\endSb{_{\substack{#1}}}
\def\Sp#1\endSp{^{\substack{#1}}}

\fi
\makeatletter
\@ifundefined{DeclareFontShape} 
{\@ifundefined{selectfont}
	{
		\message{We need AmS-LaTeX, this version of LaTeX does not 
			support it}\@@end
	}{
		\def\mathcal{\cal}
		
		\def\pcyr{%
			\def\default@family{UWCyr}%
			\let\oldSl@\sl
			\def\sl{\def\default@shape{it}\oldSl@}%
			\cyracc
			\language\Russian\family{UWCyr}\selectfont
		}
}}{
	\DeclareFontEncoding{OT2}{}{\noaccents@}
	\DeclareFontSubstitution{OT2}{cmr}{m}{n}
	\DeclareFontFamily{OT2}{cmr}{\hyphenchar\font45 }
	\DeclareFontShape{OT2}{cmr}{m}{n}{%
		<5><6><7><8><9><10>gen*wncyr %
		<10.95><12><14.4><17.28><20.74><24.88> wncyr10 %
	}{}
	\DeclareFontShape{OT2}{cmr}{m}{it}{%
		<5><6><7><8><9><10> gen * wncyi%
		<10.95><12><14.4><17.28><20.74><24.88> wncyi10%
	}{}
	\DeclareFontShape{OT2}{cmr}{bx}{n}{%
		<5><6><7><8><9><10> gen * wncyb%
		<10.95><12><14.4><17.28><20.74><24.88> wncyb10%
	}{}
	\DeclareFontShape{OT2}{cmr}{m}{sl}{%
		<-> ssub * cmr/m/it%
	}{}
	\DeclareFontShape{OT2}{cmr}{m}{sc}{%
		<5><6><7><8><9><10>%
		<10.95><12><14.4><17.28><20.74><24.88> wncysc10%
	}{}
	\DeclareFontFamily{OT2}{cmss}{\hyphenchar\font45 }
	\DeclareFontShape{OT2}{cmss}{m}{n}{%
		<8><9><10> gen * wncyss%
		<10.95><12><14.4><17.28><20.74><24.88> wncyss10%
	}{}
	
	\def\cyrencodingdefault{OT2}
	\def\pcyr{%
		\cyracc
		\let\encodingdefault\cyrencodingdefault
		\language\Russian\fontencoding{OT2}\selectfont
	}
}   

\@ifundefined{theorembodyfont} 
{
	\def\theorembodyfont#1{\relax}
	\makeatletter
	\let\@@th@plain\th@plain
	\def\th@plain{ \@@th@plain \slshape }
	\makeatother
	\let\normalshape\relax
}{}   

\ifx\cprime\undefined
\def\cprime{$'$}
\fi
\makeatletter
\def\@sect@my#1#2#3#4#5#6[#7]#8{%
	\ifnum #2>\c@secnumdepth
	\let\@svsec\@empty
	\else
	\refstepcounter{#1}%
	\edef\@svsec{\ifnum#2<\@m
		\@ifundefined{#1name}{}{\csname #1name\endcsname\ }\fi
		\noexpand\rom{\csname the#1\endcsname.}\enspace}\fi
	\@tempskipa #5\relax
	\ifdim \@tempskipa>\z@ 
	\begingroup #6\relax
	\@hangfrom{\hskip #3\relax\@svsec}{\interlinepenalty\@M #8\par}%
	\endgroup
	\if@article\else\csname #1mark\endcsname{%
		\ifnum \c@secnumdepth >#2\relax\csname the#1\endcsname. \fi#7}\fi
	\ifnum#2>\@m \else
	\let\@tempf\\ \def\\{\protect\\}\addcontentsline{toc}{#1}%
	{\ifnum #2>\c@secnumdepth \else
		\protect\numberline{%
			\ifnum#2<\@m
			\@ifundefined{#1name}{}{\csname #1name\endcsname\ }\fi
			\csname the#1\endcsname.}\fi
		#8}\let\\\@tempf
	\fi
	\else
	\def\@svsechd{#6\hskip #3\@svsec
		\@ifnotempty{#8}{\ignorespaces#8\unskip
			\ifnum\spacefactor<1001.\fi}%
		\ifnum#2>\@m \else
		\let\@tempf\\ \def\\{\protect\\}\addcontentsline{toc}{#1}%
		{\ifnum #2>\c@secnumdepth \else
			\protect\numberline{%
				\ifnum#2<\@m
				\@ifundefined{#1name}{}{\csname #1name\endcsname\ }\fi
				\csname the#1\endcsname.}\fi
			#8}\let\\\@tempf\fi}%
	\fi
	\@xsect{#5}}
\ifx\RequirePackage\undefined
\let\@sect\@sect@my             
\fi
\ifx\RequirePackage\undefined	
\def\th@remark@my{\theorempreskipamount6\p@\@plus6\p@
	
	8 warnings
	
	\theorempostskipamount\theorempreskipamount
	\def\theorem@headerfont{\it}\normalshape}
\let\th@remark\th@remark@my
\else				
\let\o@@remark\th@remark

\ifx\theorempostskipamount\undefined		
\else						
\def\th@remark{\o@@remark
	\ifdim\theorempostskipamount < 2pt\relax
	\theorempostskipamount\theorempreskipamount
	\multiply\theorempostskipamount\tw@
	\divide\theorempostskipamount\thr@@
	\fi
}
\fi
\fi
\let\myLabel\@gobble
\def\labelsONmargin{\@mparswitchfalse\def\myLabel##1{\@bsphack\marginpar
		{\normalshape\tiny\rm Label ##1}\@esphack}}
\ifx \url\undefined
\def\url#1{{\tt #1}}%
\fi
\def\cyracc{\def\u##1{
		\if \i##1\char"1A%
		\else \if I##1\char"12%
		\else \accent"24 ##1\fi\fi }%
	\def\"##1{\if e##1{\char"1B}%
		\else \if E##1{\char"13}%
		\else \accent"7F ##1\fi\fi }%
	\def\9##1{\if##1z\char"19 
		\else\if##1Z\char"11 
		\else\if##1E\char"03 
		\else\if##1e\char"0B 
		\else\if##1u\char"18 
		\else\if##1U\char"10 
		\else\if##1A\char"17 
		\else\if##1a\char"1F 
		\else\if##1p\char"7E 
		\else\if##1P\char"5E 
		\else\if##1Q\char"5F 
		\else\if##1q\char"7F 
		\else\if##1i\char"1A 
		\else\if##1I\char"12 
		\else\if##1N\char"7D 
		\fi
		\fi
		\fi
		\fi
		\fi
		\fi
		\fi
		\fi
		\fi
		\fi
		\fi
		\fi
		\fi
		\fi
		\fi
	}%
	\def\cydot{{\kern0pt}}}%
\def\cydot{$\cdot$}

\ifx\Russian\undefined
\def\Russian{0\relax
	\message{Don't know the hyphenation rules for Russian^^J
		Please do INITeX with `input  russhyph' in the 
		command line}%
	\gdef\Russian{0\relax}%
}
\fi
\ifx\RequirePackage\undefined			
\def\@putname#1#2#3#4{\def\@@ref{#3}\let\old@bf\bf
	\def\bf##1{\old@bf\if?\noexpand##1?{#4}\else##1\fi}%
	#1{#2}%
	\let\bf\old@bf}
\else						
\def\@putname#1#2#3#4{\def\@@ref{#3}\let\old@bf\bf	
	\let\old@reset@font\reset@font			
	\def\bf##1{\old@bf\if?\noexpand##1?{#4}\else##1\fi}%
	\def\reset@font##1##2{\old@reset@font##1\if?\noexpand##2?{#4}\else##2\fi}#1{#2}%
	\let\bf\old@bf\let\reset@font\old@reset@font}
\fi
\let\my@ref=\ref
\def\ref#1{\@putname\my@ref{#1}{#1}{\tiny\rm\@@ref}}
\let\my@pageref=\pageref
\def\pageref#1{\@putname\my@pageref{#1}{#1}{\tiny\rm\@@ref}}
\let\my@cite=\cite
\def\cite#1{\@putname\my@cite{#1}{\@citeb}{\tiny\rm\@@ref}}
\makeatother
\fi
\relax 
\theoremstyle{plain} 
\theorembodyfont{\sl}

\numberwithin{equation}{section}
\theorembodyfont{\rm}
\theoremstyle{definition}

\newtheorem{definition}{Definition}[section]
\newtheorem{conjecture}[definition]{Conjecture}
\newtheorem{example}[definition]{Example}

\theoremstyle{remark}

\newtheorem{remark}[definition]{Remark} 

\theoremstyle{plain} 
\theorembodyfont{\sl}

\newtheorem{theorem}[definition]{Theorem}
\newtheorem{lemma}[definition]{Lemma}
\newtheorem{corollary}[definition]{Corollary}
\newtheorem{proposition}[definition]{Proposition}

\newcommand{\cF}{\mathcal{F}}
\newcommand{\JJ}{\mathcal{J}}
\renewcommand{\SS}{\mathcal{S}}

\renewcommand{\gg}{\mathfrak{g}}
\newcommand{\pp}{\mathfrak{p}}
\newcommand{\qq}{\mathfrak{q}}

\renewcommand{\aa}{\mathfrak{a}}
\renewcommand{\ll}{\mathfrak{l}}
\newcommand{\hh}{\mathfrak{h}}
\newcommand{\bb}{\mathfrak{b}}
\newcommand{\kk}{\mathfrak{k}}

\newcommand{\ft}{\mathfrak{t}}
\newcommand{\fg}{\mathfrak{g}}
\renewcommand{\tt}{\operatorname{t}}
\renewcommand{\dim}{\operatorname{dim}}
\newcommand{\sdim}{\operatorname{sdim}}
\newcommand{\CC}{\mathbb{C}}
\newcommand{\ZZ}{\mathbb{Z}}
\newcommand{\RR}{\mathbb{R}}

\newcommand{\Aut}{\operatorname{Aut}}

\newcommand{\soc}{\operatorname{soc}}

\newcommand{\ad}{\operatorname{ad}}
\newcommand{\Hom}{\operatorname{Hom}}
\newcommand{\Ext}{\operatorname{Ext}}
\newcommand{\End}{\operatorname{End}}
\renewcommand{\Im}{\operatorname{Im}}
\newcommand{\Ker}{\operatorname{Ker}}
\newcommand{\Ind}{\operatorname{Ind}}

\newcommand{\Res}{\operatorname{Res}}

\newcommand{\Mod}{\operatorname{mod}}
\newcommand{\Span}{\operatorname{span}}
\newcommand{\Gr}{\mathscr{K}}

\newcommand{\fgl}{\mathfrak{gl}}
\newcommand{\fsl}{\mathfrak{sl}}
\newcommand{\osp}{\mathfrak{osp}}
\newcommand{\fp}{\mathfrak{p}}
\newcommand{\fq}{\mathfrak{q}}
\newcommand{\fh}{\mathfrak{h}}

\newcommand{\vareps}{\varepsilon}

\newcommand{\Id}{\operatorname{Id}}

\newcommand{\bV}{\mathbf{V}}
\newcommand{\bM}{\mathbf{M}}
\newcommand{\bK}{\mathbf{K}}
\newcommand{\bJ}{\mathbf{J}}
\newcommand{\bT}{\mathbf{T}}
\newcommand{\bW}{\mathbf{W}}
\newcommand{\bP}{\mathbf{P}}
\newcommand{\bQ}{\mathbf{Q}}
\newcommand{\bI}{\mathbf{I}}

\newcommand{\str}{\text{str}}

\usepackage{color}
\newcommand{\sch}{\operatorname{sch}}
\newcommand{\DS}{DS}
\newcommand{\ds}{ds}
\newcommand{\Fin}{\operatorname{{\mathcal F}in}}

\newcommand{\supp}{\operatorname{supp}}
\newcommand{\atyp}{\operatorname{atyp}}
\newcommand{\defect}{\operatorname{def}}
\newcommand{\rank}{\operatorname{rk}}
\newcommand{\cS}{\mathcal{S}}

\newcommand{\smult}{\text{smult}}

\usepackage[all,cmtip]{xy}
\usepackage{tikz}

\usepackage[scr=rsfs]{mathalpha}

\usetikzlibrary{shapes.misc}
\usetikzlibrary{decorations.pathreplacing}
\usetikzlibrary{decorations.pathmorphing}

\tikzset{snake it/.style={decorate, decoration=snake}}

\tikzset{cross/.style={cross out, draw=black, minimum size=2*(#1-\pgflinewidth), inner sep=0pt, outer sep=0pt},
	cross/.default={1pt}}

\makeatletter
\def\l@subsection{\@tocline{2}{0pt}{2.5pc}{5pc}{}}
\makeatother

\begin{document}
	\relax 
	
	\title[The Duflo--Serganova functor]{The Duflo--Serganova functor, vingt ans apr\`es}
	
	\author{ Maria Gorelik, Crystal Hoyt, Vera Serganova, Alexander Sherman }
	
	\date{ \today }
	
	\address{Dept. of Mathematics, The Weizmann Institute of Science, Rehovot 76100, Israel}
	\email{maria.gorelik@weizmann.ac.il}

	\address{Dept. of Mathematics, Bar-Ilan University, Ramat Gan 52900, Israel}
	
	\email{math.crystal\atSign{}gmail.com}
	
	\address{ Dept. of Mathematics, University of California, Berkeley, CA 94720, USA}
	
	\email{serganov\atSign{}math.berkeley.edu}
	
	\address{ Dept. of Mathematics, Ben Gurion University, Beer-Sheva 8410501,
		Israel}
	
	\email{xandersherm\atSign{}gmail.com}

	\maketitle
	
	\centerline{\emph{To Michel Duflo, with admiration.}}

	\begin{abstract}  We review old and new results concerning
		the  $\DS$ functor and associated varieties for Lie superalgebras. These notions were introduced in the unpublished manuscript \cite{DS} by Michel Duflo and the third author. This paper includes the results and proofs of the original manuscript, as well as a survey of more recent results. \end{abstract}

	\tableofcontents
	
	
	\section{Introduction}\label{intro}

	The $\DS$ functor was introduced by Michel Duflo and the third author approximately 20 years ago, but the original manuscript 
	\cite{DS} was never published. Since then much progress has been made in the study of the $\DS$ functor. This paper includes the results of the original manuscript, as well as a survey of more recent results obtained by different authors.

	The $\DS$ functor has a large and growing list of applications throughout the literature. 
	It was used  in \cite{S2}
	to prove the Kac--Wakimoto conjecture (see Section~\ref{app sch}); in  \cite{IRS} to describe the  supercharacter ring for $\pp(n)$ (see Section~\ref{sec ds groth}); in \cite{HPS} to study important $\mathfrak{sl}(\infty)$-modules (see Section~\ref{sec DS sl}); in \cite{ES2} to give a formula for the superdimension of $\pp(n)$-modules; in \cite{HsW} to give a new proof of the superdimension formula for $GL(m|n)$-modules; in \cite{Hs} to obtain reductive envelopes of certain supergroups; in \cite{BKN2} to compute complexity of certain modules over $\mathfrak{gl}(m|n)$; in \cite{CH} to classify the indecomposable summands of tensor powers of the standard representation of $OSP(m|2n)$; and in \cite{EHS} to construct  universal tensor categories.
	The $\DS$ functor has been applied  to study Deligne categories in numerous papers  (see e.g., \cite{CH,EHS,ES1}).

	The associated variety of a module over a Lie superalgebra $\gg=\gg_0\oplus\gg_1$ is a subvariety of the cone $ X\subset\gg_{1} $ of self-commuting odd elements. The cone $ X $ was studied  in \cite{Gr1,Gr2,Gr3}, where  geometric properties of $ X $ were used to obtain important results about the cohomology of Lie superalgebras.

	Now if $x\in X$ and $M$ is a $\gg$-module, then $x^2(M)=0$ and hence we can take the cohomology $M_x=\Ker x_M/\Im x_M$.
	The assignment $M\mapsto M_x$  defines the Duflo--Serganova functor  $\DS_x:\Mod(\gg)\to\Mod(\gg_x)$, where $\gg_x=\Ker\,\ad x/\Im\,\ad x$ is a Lie algebra. It is easy to see that this functor is symmetric monoidal. This obvious but remarkable fact does not have an analogue  in the theory of Harish-Chandra modules or in the theory of restricted Lie algebras. 
	
	For the basic classical Lie superalgebras, $\DS_x(L)$ has been computed for every simple
	finite-dimensional module $L$. These computations show that $\DS_x(L)$
	is semisimple and ``pure"  in the following sense: for every simple $\gg_x$-module $L'$
	one has $[\DS_x(L):L']\cdot [\DS_x(L):\Pi L']=0$. It would be interesting to find a conceptual
	proof of these facts, see Section~\ref{action of DS} for  details.
	
	The associated variety $X_M$ for a $\gg$-module $M$ is the {closure} in $X$ of the subset consisting of all elements $x\in X$ for which $M_x$ is nonzero.
	The associated variety for a module over a Lie superalgebra can be seen as an analogue of the associated varieties for Harish-Chandra modules, if we think about a Lie
	superalgebra $\gg=\gg_0\oplus\gg_1$ as a symmetric pair. Associated varieties for Harish-Chandra modules have many interesting
	applications in the classical representation theory (see, for example,
	\cite{V,KO,NOT}). 
	While the associated variety in the theory of Harish-Chandra modules is
	trivial if a module is finite-dimensional, finite-dimensional modules over
	Lie superalgebras have interesting associated varieties. Some applications  of these associated varieties are given in Sections~\ref{app sch} and \ref{app cohom}.

	On the other hand, the associated variety for a module over a Lie superalgebra  is also an analogue of the rank variety for restricted Lie algebras in positive characteristic,
	see \cite{FP}. For example, in many cases these associated varieties for Lie superalgebras detect projectivity in the category of finite-dimensional $\gg$-modules.
	This is proven in Section~\ref{sec proj} of the present paper; the original proof in the preprint  \cite{DS} had a mistake.
	
	In the category of finite-dimensional $\gg$-modules, associated varieties are closely related to blocks and central characters, see Theorem~\ref{th2} and
	Theorem~\ref{th3}. In the original preprint \cite{DS}, Theorem~\ref{th3} was proven for $\mathfrak{gl}(m|n)$, now it is known for all basic classical superalgebras (\cite{S2, M}).
	It also seems that associated varieties can be used to study category $\mathcal O$ for Kac--Moody superalgebras. Some results in this direction are
	obtained in \cite{CS} and \cite{GS}.
	
	Finally, let us mention that in contrast with restricted Lie algebras, \cite{FP}, the cohomological support varieties defined and
	studied in \cite{BKN1} and \cite{BKN2} are
	quite different from  the varieties  studied in this paper. This may indicate existence of a third definition which interpolates these two constructions.
	
	The Duflo-Serganova functor has also appeared under different guises in theoretical physics.  Odd operators $Q$ in a supersymmetric field theory that satisfy $Q^2=0$ are examples of BRST operators, and allow one to employ the so-called BRST formalism, which includes taking cohomology in $Q$.  Such a situation arises in several places in the literature.  In \cite{W}, it was used as a key part in topological and holomorphic twists of  supersymmetric field theories.  Twisting a supersymmetric field theory gives rise to simpler field theories that can either be topological (giving a TQFT), holomorphic, or something in between, depending on properties of the chosen $Q$.  A mathematically rigorous approach to the twisting of supersymmetric field theories is explained in \cite{Cos}. In \cite{CCMV}, the authors apply $\DS$ functors (referred to as cohomological reductions in their paper) to the algebra of smooth functions as well as certain vector bundles on the target spaces of sigma models, which are homogeneous superspaces.  Like in the mathematical setting, this reduction produces simpler theories which then in turn can give information about the original theory.

	
	\subsection{Acknowledgements} We would like to thank Kevin Coulembier, Inna Entova-Aizenbud, Thorsten Heidersdorf, Vladimir Hinich, Victor Kac, Victor Ostrik, Ivan Penkov, Julia Pevtsova,  Shifra Reif, and  Ilya Zakharevich for helpful comments and suggestions.  Needless to say, this paper would not have been possible without the original contribution of Michel Duflo.  In addition we thank the referees for very thorough reviews of an earlier version of this article.
	
	\subsection{Notation} \label{rem: classical LSA}
	Throughout this paper (and in particular in Sections 4-9 and 12), we will  primarily focus on the following important list of Lie superalgebras, which by slight abuse of terminology we will refer to as \emph{classical Lie superalgebras}:
	\begin{equation}\label{list1}
		\mathfrak{sl}(m|n),\ m\neq n,\  \mathfrak{gl}\left(m|n\right), 
		\ \mathfrak{osp}\left(m|2n\right),\  D\left(2|1;a\right),\ F(4),
		\  G(3),\  \pp(n), \ \qq(n).\end{equation}
	Note that each superalgebra appearing in this list has a ``cousin'' that is a classical Lie superalgebra in the sense of \cite{K1}. 
	Additionally, by \emph{basic classical Lie superalgebra}, we will mean a superalgebra from the above list, excluding $\pp(n)$ and $\qq(n)$.
	We will sometimes 
	refer to the superalgebras $D(2|1;a)$, $G(3)$ and $F(4)$ as ``exceptional''
	and to other superalgebras from our list as ``non-exceptional''.
	
	\subsubsection{List of notation} We present below a table of the commonly used notation in the article:
	\begin{itemize}
		\item $\Mod(\gg)$ the category of $\gg$-modules.
		\item $\Fin(\gg)$ the category of finite-dimensional $\gg$-modules.
		\item $\cF(\gg)$ the category of finite-dimensional $\gg$-modules semisimple over $\gg_0$.
		\item $p$ a parity function on weights.
		\item $\Mod_{\chi}^r(\gg)$ the category of $\gg$-modules with generalized central character $\chi$ of order $r$.
		\item $\Mod_{\chi}(\gg)$ category of $\gg$-modules admitting generalized central character $\chi$.
		\item $\Res_{\kk}^{\gg}$ the restriction functor from $\gg$-modules to $\kk$-modules.
		\item $\Ind_{\kk_0}^{\gg}$ the induction functor.
		\item $X$ the self-commuting cone of $\gg$.
		\item $X_k$ the set of rank $k$ elements in $X$.
		\item $\mathcal{O}_X$ the structure sheaf on $X$.
		\item $\partial$ the differential on $\mathcal{O}_X\otimes M$ in Section~\ref{app cohom}.
		\item $X_M$ the associated variety of $M$.
		\item $\DS_x$ the Duflo--Serganova functor determined by $x\in X$.
		\item $DS^r$ the Duflo--Serganova functor determined by a rank $r$ element.
		\item $\eta_x:\mathcal{Z}(\gg)\to\mathcal{Z}(\gg_x)$ the induced map on center.
		\item $\sigma_x$ the involution of $\gg_x$ for classical Lie superalgebras (\ref{involutions section}). 
		\item  $\Gr(\mathcal{C})$ the Grothendieck group of a full abelian subcategory $\mathcal{C}$ of $\Mod(\gg)$.
		\item $M_{gr}$ the image of a module $M$ in $\Gr(\CC)$.
		\item $\Gr_-(\mathcal{C})$ the reduced Grothendieck group
		(quotient by $M_{gr}=-(\Pi M)_{gr}$).
		\item $[M]$ the image of a module $M$ in $\Gr_{-}(\CC)$.
		\item $\Gr_+(\mathcal{C})$ the character group (quotient by $M_{gr}=(\Pi M)_{gr}$).   
		\item $\sch M$ the supercharacter of $M$.
		\item $ds_x$ the map induced by the functor $DS_x$ on reduced Grothendieck groups.
		\item $ds^r$ the map $ds_x$ for a rank $r$ element $x$.
		\item $G_0$ the simply connected, connected Lie group corresponding to $\gg_0$.
		\item $\Delta$ the roots of $\gg$ with respect to a Cartan subalgebra of $\gg_0$.
		\item $W$ the Weyl group.
		\item $\rho$ the Weyl vector.
		\item $L(\lambda)=L_{\gg}(\lambda)$ is the irreducible $\gg$-module of highest weight $\lambda$ with respect to a chosen Borel subalgebra.
		\item $\atyp\chi,\atyp\lambda$ the degree of atypicality of $\chi,\lambda$ resp.
		\item $[M:L]$ the multiplicity of a simple module $L$ in a finite-length module $M$.
		\item $[M:L]_{non}$ the non-graded multiplicity of a simple module $L$ in $M$, meaning the number of times both $L$ and $\Pi L$ appear.
		\item $R$ the super Weyl denominator.
		\item $k(\lambda)$ a virtual Kac module.
	\end{itemize}


	\section{Definitions and basic properties }\label{sec:definitions}
	
	Our ground field is $\CC$, and by $\overline{Y}$ we denote the Zariski closure of a subset $Y$ of an affine space.
	By $\Pi$ we denote the change of parity functor in the category of vector superspaces.
	
	Throughout this paper we assume that the Lie superalgebra $\gg=\gg_{0}\oplus\gg_{1} $ is finite dimensional. Let $ G_{0} $ 
	denote a simply-connected connected algebraic group with Lie algebra $ \gg_{0} $.  We will write $\Fin(\gg)$ for the category of finite-dimensional $\gg$-modules, and $\cF(\gg)$ for the full subcategory of $\Fin(\gg)$ consisting of modules which are semisimple over $\gg_0$.  The  category $\cF(\gg)$ will be the main object of study after Section~\ref{sec:definitions}.
	
	\subsection{The associated variety $X_M$}
	
	Let
	\begin{equation}
		X=\left\{x\in\gg_{1} \mid \left[x,x\right]=0\right\}.
		\notag\end{equation}
	It is clear that $ X $ is a $ G_{0} $-invariant Zariski closed cone in $ \gg_{1} $.

	Let $M$ be a
	$ \gg $-module. For every $ x\in X $, the corresponding
	element $x_M\in\operatorname{End}_{\CC}(M)$ satisfies
	$x_M^2=0$. Set 
	$$ 
	M_{x}:=\operatorname{Ker} x_M/xM 
	$$ 
	and define
	\begin{equation}
		X_M:=\left\{x\in X \mid M_{x}\not=0\right\}.
		\notag\end{equation}
	We call $ X_M $ the \emph{associated variety} of $ M $.
	
	\begin{lemma} \label{lm1}\myLabel{lm1}\relax  If $ M $ is a finite-dimensional $( \gg, G_0) $-module, then $ X_M $ is Zariski
		closed $ G_{0} $-invariant subvariety.
	\end{lemma}
	
	\begin{proof} For a finite-dimensional $M$,
		$$X\setminus X_M=\{x\in X\,|\,\operatorname{rank}\,x_M=\dim M_0=\dim M_1\}.$$
		Hence $X_M$ is Zariski closed.
		Now  $ M $ is a $ G_{0} $-module. For each
		$ g\in G_{0} $ and $ x\in M $ one has
		\begin{equation}
			M_{\operatorname{Ad}_{g}\left(x\right)}=gM_{x},
			\notag\end{equation}
		which implies the lemma.\end{proof}
	
	\subsection{The Lie superalgebra $\gg_x$}
	For $x\in X$, we
	define
	$$
	\gg_x:=\gg^x / [x,\gg],
	$$
	where $\gg^x:=\Ker \ad_x$ and $[x,\gg]:=\Im\,\ad_x$.

	The next lemma follows from the definitions.
	
	\begin{lemma}\label{gx Mx}
		Let $\gg$ be a finite-dimensional Lie superalgebra and $x\in X$.
		\begin{enumerate}
			\item Then $[x,\gg]$ is an ideal in $\gg^x$, and hence $\gg_x$ has the natural structure of a Lie superalgebra.
			\item If $M$ is a $\gg$-module, then $M_x$ is a $\gg_x$-module.
		\end{enumerate}
	\end{lemma}

	Now  we observe that for each $x\in X$, the correspondence $M\mapsto M_x$ is functorial. Let $\Mod(\gg)$ (respectively, $\Mod(\gg_x)$) denote the category of all $\gg$-modules (respectively, $\gg_x$-modules).

	\begin{definition} 
		The \emph{Duflo--Serganova functor}  $\DS_x:\Mod(\gg)\to\Mod(\gg_x)$ is defined by $\DS_x(M):=M_x$.  \end{definition}

	The functor $\DS_x$ has many nice properties. The following lemma shows that $\DS_x$ is a symmetric monoidal tensor functor.
	
	\begin{lemma}\label{tensor} Let $\gg$ be a finite-dimensional Lie superalgebra, let $x\in X$, and let $M,N$ be $\gg_x$-modules. \begin{enumerate}
			\item  We have a canonical isomorphism  $(M\otimes N)_x\simeq M_x\otimes N_x$ of $\gg_x$-modules.
			\item  For any $\gg$-module $M$ we have a canonical isomorphism $(M^*)_x\to (M_x)^*$  of $\gg_x$-modules.  
		\end{enumerate}
		Hence, $\DS_x:\Mod(\gg)\to\Mod(\gg_x)$  is a symmetric monoidal tensor functor.  
	\end{lemma}
	\begin{proof} 
		For (1), we have the natural morphism of $\gg_x$-modules
		$M_x\otimes N_x\to (M\otimes N)_x$. We have to check that this is an
		isomorphism. This follows from the fact that
		over the $(0|1)$-dimensional superalgebra  $\CC x$, we have
		$M=M_x\oplus F$, $N=N_x\oplus F'$ for some free $\CC x$-modules $F$ and $F'$, and we have
		$$(M\otimes N)=M_x\otimes N_x\oplus (F\otimes N\oplus M\otimes F'),$$
		where $F\otimes N\oplus M\otimes F'$ is a free $\CC x$-module.
		
		For (2), we have a natural map $(M^*)_x\to(M_x)^*$ given by $\varphi\mapsto\varphi|_{\Ker x}$, using the fact that $\varphi(\operatorname{im}x)=0$.  In the other direction: given $\varphi:M_x\to\CC$, write $\tilde\varphi$ for the lift of $\varphi$ to $\Ker x$ and choose a splitting $M=\Ker x\oplus V$.  Then $\phi=\tilde\varphi\oplus0$ is annihilated by $x$, and this defines a morphism $(M_x)^*\to (M^*)_x$ inverse to our previous map.    
	\end{proof}

	The next lemma shows that the  functor $\DS_x$ preserves the superdimension   of a finite-dimensional module $M$, where the superdimension of $ M=M_0\oplus M_1 $ is given by  $\sdim M :=\dim M_0 -\dim M_1$.
	
	\begin{lemma}\label{lem sdim}
		For any finite-dimensional $ \gg $-module $ M $ and $ x\in X $, $\sdim  M= \sdim  M_{x} $.
	\end{lemma}
	
	\begin{proof}
		Let $ \Pi\left(N\right) $ stand for the superspace
		isomorphic to $ N $ with switched parity. Since $ M/\operatorname{Ker} x_M $ is isomorphic to $ \Pi\left(xM\right) $,
		we have
		$$
		\sdim M=\sdim(\operatorname{Ker} x_M)-\sdim\left(x M\right)=\sdim \left(\operatorname{Ker} x_M /xM\right)=\sdim M_{x}.
		$$
	\end{proof}
	
	In fact, Lemma~\ref{lem sdim} has a natural generalization, as we will see in the next section.

	\subsection{Reduced Grothendieck groups and $ds_x$}\label{rem ds induce}
	Let $\mathcal{C}$ be a full abelian subcategory of $\Mod(\gg)$ such that:
	\[
	(*) \ \ \ \ \Pi M\text{ is an object of }\mathcal{C}\text{ whenever }M\text{ is.}
	\]
	We define the Grothendieck group $\Gr(\mathcal{C})$ in the usual way as the quotient of the free $\mathbb{Z}$-module with basis $M_{gr}$, for each object $M$ in $\mathcal{C}$, with relations $M_{gr}=M_{gr}'+M_{gr}''$ for every short exact sequence $0\to M'\to M\to M''\to 0$ in $\mathcal{C}$.
	
	The \emph{reduced Grothendieck group}  $\Gr_-(\mathcal{C})$ of the category  $\mathcal{C}$ is the quotient $\Gr(\mathcal{C})$ by the relation  $M_{gr}=-(\Pi M)_{gr}$ for all objects $M\in\mathcal{C}$.  We  define the {\em character group} $\Gr_+(\mathcal{C})$ to be the quotient of $\Gr(\mathcal{C})$ by $M_{gr}=(\Pi M)_{gr}$. Write $(-)_{\mathbb{Q}}$ for the extension of scalars from $\mathbb{Z}$ to $\mathbb{Q}$.  Then by the Chinese Remainder Theorem we have
	\begin{equation}\label{GR decomp}
		\Gr(\mathcal{C})_{\mathbb{Q}}\cong\Gr_-(\mathcal{C})_{\mathbb{Q}}\times\Gr_+(\mathcal{C})_{\mathbb{Q}}.
	\end{equation}
	If $\mathcal{C}$ is closed under $\otimes$, then its tensor structure provides $\Gr(\mathcal{C})$ with a ring structure such that $\Gr_-(\mathcal{C})$ and $\Gr_+(\mathcal{C})$ are quotient rings, and (\ref{GR decomp}) becomes an isomorphism of rings.
	
	\begin{remark}
		Since we work over the integers, observe that if $M$ is any module in $\mathcal{C}$ with $M\cong \Pi M$, its image in $\Gr_{-}(\mathcal{C})$ will be 2-torsion (although it need not be 0).
	\end{remark}
	
	\begin{lemma}\label{reduced}[V. Hinich] If
		$$0\to N\xrightarrow{\psi} M\xrightarrow{\varphi} L\to 0$$
		is an exact sequence of $\gg$-modules, then there exists an exact sequence
		$$0\to E\to N_x\to M_x\to L_x\to\Pi E\to 0 $$
		for some $\gg_x$-module $E$.
	\end{lemma}
	\begin{proof} Set $E$ be the kernel of the induced map $\psi:N_x\to M_x$ and $E'$ be the quotient $L_x/\varphi (M_x)$. Then we have the exact sequence
		$$0\to E\to N_x\to M_x\to L_x\to E'\to 0.$$
		The odd map $\psi^{-1}x\varphi^{-1}:L_x\to N_x$ induces an isomorphism $E'\to\Pi E$.
	\end{proof}
	
	Lemma~\ref{tensor} and Lemma~\ref{reduced} imply the following.
	
	\begin{corollary}
		The functor $\DS_x$ is a middle exact tensor functor and satisfies $\DS_x(\Pi M)=\Pi \DS_x(M)$.
	\end{corollary}

	\begin{corollary}  Let $\mathcal{C}$, (resp. $\mathcal{C}_x$) be full abelian subcategories of $\Mod(\gg)$ (resp. $\Mod(\gg_x)$) satisfying (*).  Suppose that $DS_x(M)$ lies in $\mathcal{C}_x$ whenever $\mathcal{C}$ lies in $\mathcal{C}$.  Then the functor  $\DS_x:\mathcal{C}\to\mathcal{C}_x$ induces a homomorphism on the corresponding reduced Grothendieck groups 
		$$
		\ds_x:\Gr_{-}(\mathcal{C})\to \Gr_{-}(\mathcal{C}_x).
		$$
		
	\end{corollary} 
	
	We now focus in particular on the case when $\mathcal{C}=\Fin(\gg)$.  Then Lemma~\ref{reduced} in particular implies that the following diagram commutes:
	\begin{equation}\label{eqdsx}
		\xymatrix{ \Gr_-(\Fin(\gg))\ar[rd]_{\ds_x} \ar[r] & \Gr_-(\Fin(\gg^x))  \\
			&        \Gr_-(\Fin(\gg_x)) \ar[u]}
	\end{equation}
	where the horizontal arrow is induced by the restriction functor $\Res^{\gg}_{\gg^x}$, and the up arrow is induced by $\Res^{\gg_x}_{\gg^x}$, where  $\gg^x \to \gg_x$ is the canonical surjection.

	\begin{remark}
		The map $\ds_x:\Gr_{-}(\Fin (\gg))\to \Gr_{-}(\Fin (\gg_x))$ is a ring homomorphism compatible with duality. This follows from the fact that $\Fin (\gg)$ and $\Fin(\gg_x)$ are tensor categories, since 
		$\DS_x$ is a tensor functor that preserves the duality.
	\end{remark}
	
	\begin{remark}
		The existence of a homomorphism between reduced Grothendieck rings was first observed when $\gg=\mathfrak{gl}(m|n)$ in \cite{HsW}.  In \cite{HR}, this homomorphism was introduced in the language of supercharacters for the category of finite-dimensional modules of finite-dimensional Kac-Moody superalgebras, and its kernel and image were described.
	\end{remark}

	\subsection{$ds_x$ as restriction} For a module $M$, we write $[M]$ for its image in the reduced Grothendieck ring. 
	
	\begin{lemma}\label{ds restriction}
		Suppose that we have a splitting $\gg_x\subseteq\gg^x$ so that $\gg^x=\gg_x\ltimes [x,\gg]$.  Then for a finite-dimensional $\gg$-module $M$  we have
		\[
		ds_x[M]=[\Res_{\gg_x}^{\gg}M].
		\]
	\end{lemma}
	\begin{proof}
		This follows immediately by applying the restriction $\Gr_{-}(\Fin(\gg^x))\to \Gr_{-}(\Fin(\gg_x))$ to the equality $[DS_xM]=[\Res_{\gg^x}^{\gg}M]$ coming from Lemma~\ref{reduced}.
	\end{proof}
	
	\begin{lemma}\label{ds composition}
		Let $x,y\in X$ such that $[x,y]=0$, and suppose that we have splittings
		\[
		\gg^{y}=\gg_y\ltimes [y,\gg],\ \ \ \gg^{x+y}=\gg_{x+y}\ltimes[x+y,\gg],
		\]
		Furthermore suppose that under these splittings, $x\in\gg_{y}$ and 
		\[
		(\gg_{y})^{x}=\gg_{x+y}\ltimes[x,\gg_{y}].
		\]
		Then we have
		\[
		ds_{x+y}=ds_{x}\circ ds_{y}:\Gr_{-}(\gg)\to \Gr_{-}(\gg_{x+y})
		\]
	\end{lemma}
	
	\begin{proof}
		This follows immediately from Lemma~\ref{ds restriction} and the corresponding statement for restriction.
	\end{proof}

	\begin{lemma}\label{lem:group ker}
		Suppose that $x,y\in X$ and that there exists $g\in G_0$ such that $gx=y$.  Then we have a commutative diagram
		\[
		\xymatrix{\Gr_{-}(\Fin(\gg))\ar[r]^{ds_x}\ar[dr]_{ds_y} & \Gr_{-}(\Fin(\gg_x))\ar[d]^{\sim} \\ & \Gr_{-}(\Fin(\gg_y))}
		\]
		where the downward arrow is an isomorphism and is induced by the action of $g$.  In particular:
		\[
		\ker \left(ds_{x}|_{\Gr_{-}(\Fin(\gg))}\right)=\ker\left(ds_{y}|_{\Gr_{-}(\Fin(\gg))}\right)
		\]
	\end{lemma}
	
	\begin{proof}
		We have a commutative diagram
		\[
		\xymatrix{\Fin(\gg)\ar[r]^{DS_x}\ar[dr]_{DS_y} & \Fin(\gg_x)\ar[d]^{\sim} \\ & \Fin(\gg_y)}
		\]
		where the downward arrow is induced by the action of $g$, and is an equivalence.  Passing to the reduced Grothendieck ring completes the argument.
	\end{proof}
	
	\subsection{Supermultiplicity}
	Let $\gg$ be  a finite-dimensional Lie superalgebra, and let $\mathfrak{a}$ be any subalgebra of $\gg$.
	We will view
	$$\mathfrak{a}_x:=\mathfrak{a}^x/([\gg,x]\cap \mathfrak{a})$$
	as a subalgebra of $\gg_x$. 
	
	In addition to preserving superdimension, the $\DS$ functor  also preserves the supermultiplicity of $\gg^x$-modules, when this notion is well-defined.  We continue to work just with finite-dimensional modules.  
	The \emph{multiplicity}  of a simple module $L$ in a finite-length module $M$, denoted  $[M:L]$, is the number of factors in the Jordan--H\"older series of $M$ which are isomorphic to $N$.  If $M$ is a finite-dimensional module and $L$ is a finite-dimensional simple module, then we can define the \emph{supermultiplicity} of $L$ in $M$ to be:
	
	\begin{equation}\label{smult}
		\smult(M;L):=\left\{\begin{array}{ll}
			\ [M:L]-[M:\Pi L] & \text{ if }L\not\cong\Pi L \\
			\ [M:L] \text{ mod } 2 & \text{ if }L\cong\Pi L.\end{array}\right.
	\end{equation}
	
	The following lemma is immediate.
	\begin{lemma}
		Let $L$ be simple, finite-dimensional $\gg$-module.  Then $\smult(-;L)$ defines a homomorphism 
		\[
		\Gr_-(\Fin(\gg))\to \left\{\begin{array}{ll}\mathbb{Z} & \text{ if }L\not\cong\Pi L \\  \mathbb{Z}_{2} & \text{ if }L\cong\Pi L\end{array}\right.
		\]
	\end{lemma}
	
	The following proposition is from \cite{G3}.  
	
	\begin{proposition}\label{super_mult_preserved_ds}
		Let $M$ be in $\Fin(\gg)$, and let $L$ be a simple finite-dimensional $\gg^x$-module.  Then  one has
		\[
		\smult(\Res^{\gg}_{\gg^x} M;L)=
		\smult(\DS_x(M);L),
		\]
		where $\DS_x(M)$ is viewed as a $\gg^x$-module.
	\end{proposition}
	
	\begin{proof}
		This follows immediately from (\ref{eqdsx}).
	\end{proof}
	
	\begin{remark}
		In many cases $\gg_x$ can be viewed as a subalgebra of $\gg$ in a way that $\gg^x = \gg_x\ltimes [x,\gg]$, and in these cases, the above formula also
		holds for each simple $\gg_x$-module $L$. 
		In particular, the claim holds if $\gg$ is a classical Lie superalgebra (see Proposition~\ref{lm202}).  \end{remark}

	\begin{proposition}\label{propka}
		We have 
		the following commutative diagram
		$$
		\xymatrix{ \Gr_{-}(\Fin(\gg)) \ar[d]^{\ds_x} \ar[r] &  \Gr_{-}(\Fin(\mathfrak{a}^x))  \\
			\Gr_{-}(\Fin(\gg_x))  \ar[r] & \Gr_{-}(\Fin(\mathfrak{a}_x))\ar[u]^{\text{res}^{\mathfrak{a}_x}_{\mathfrak{a}^x}}}  
		$$
		where the horizontal arrows are induced by  the corresponding restriction functors
		and $\text{res}^{\mathfrak{a}_x}_{\mathfrak{a}^x}$ is induced by the functor  $\text{Res}^{\mathfrak{a}_x}_{\mathfrak{a}^x}$  for the canonical surjection $\aa^x\to\aa_x$. 
	\end{proposition}
	\begin{proof}
		The restriction functors give the commutative diagram
		$$\xymatrix{ \Fin(\gg^x)  \ar[r] & \Fin(\mathfrak{a}^x) \\
			\Fin(\gg_x)  \ar[r]  \ar[u] & \Fin(\mathfrak{a}_x)\ar[u] \\
		}$$
		which allows us to rewrite the original diagram in the form
		$$ 
		\xymatrix{ \Gr_{-}(\Fin(\gg))\ar[d]^{\ds_x} \ar[r] & \Gr_{-}(\Fin(\mathfrak{a}^x))  \\
			\Gr_{-}(\Fin(\gg_x)) \ar[r] & \Gr_{-}(\Fin(\gg^x))\ar[u]  }
		$$
		where all  arrows except $\ds_x$ are induced by the  restriction functors. By~(\ref{eqdsx}), the above diagram is commutative.
	\end{proof}

	\begin{example}\label{exafin} 
		Suppose $\gg$ is a classical Lie superalgebra in the sense of \cite{K1}, and let $\mathfrak{a}:=\hh$
		be a Cartan subalgebra of $\gg_{0}$.  Restriction induces a map
		$\Gr_{-}(\Fin(\gg))\to \Gr_{-}(\Fin(\hh))$  which we write as $[N]\mapsto \sch N$, where $\sch N$ denotes the supercharacter of $N$  (see  (\ref{def sch})). 
		If $\hh_x$ is  a Cartan subalgebra of $(\gg_x)_{0}$, then the composed map 
		$ds_x:\Gr_{-}(\Fin(\gg))\to \Gr_{-}(\Fin(\hh_x))$ is given by
		$[N]\mapsto \sch \DS_x(N)$. If we fix an embedding $\hh_x\to \hh$, then Proposition~\ref{propka} gives the Hoyt--Reif formula \cite{HR}
		\begin{equation}\label{HR formula}
			\sch \DS_x(N)=(\sch N)|_{\hh_x}.
		\end{equation}
	\end{example}
	
	\subsection{Properties of associated varieties} Here we list a few basic properties of associated varieties for a finite-dimensional Lie superalgebra $\gg$.  Let $\mathcal{U}(\gg)$ denote the universal enveloping algebra of $\gg$.
	
	We have the following.
	
	\begin{lemma}\label{lm2}\myLabel{lm2}\relax 
		Let $\gg$ be a finite-dimensional Lie superalgebra.
		\begin{enumerate}
			\item
			If $ M=\mathcal U\left(\gg\right)\otimes_{\mathcal U\left(\gg_{0}\right)}M' $ for some $ \gg_{0} $-module $ M' $, then $ X_M=\left\{0\right\} $;
			\item
			If $ M={\CC} $ is trivial, then $ X_M=X $;
			\item
			For any $ \gg $-modules $ M $ and $ N $, one has $ X_{M\oplus N}=X_M\cup X_{N} $;
			\item
			For any $ \gg $-modules $ M $ and $ N $, one has $ X_{M\otimes N}=X_M\cap X_{N} $;
			\item
			For any $ \gg $-module $ M $, $ X_{M^{*}}=X_M $;
		\end{enumerate}
	\end{lemma}
	
	\begin{proof} (2) and (3)  follow directly from the definition, while (4) and (5) follow from Lemma~\ref{tensor}.

		To prove (1), let $ x\in X $ and $ x\not=0 $. Let $ \left\{v_{j}\right\}_{j\in J} $ be a basis of  $ M' $  and $ x_{1},\dots ,x_{m} $ be a basis of $ \gg_{1} $
		such that $ x=x_{1} $. Then by the PBW Theorem for Lie superalgebras, the elements $ x_{i_{1}}x_{i_{2}}\dots x_{i_{k}}\otimes v_{j} $ for all
		$ 1\leq i_{1}<i_{2}<\dots <i_{k}\leq m $, $ j\in J $ form a basis of $ M $. The action of
		
		$ x=x_{1} $ in this basis is
		easy to write, and it is clear that $ \operatorname{Ker} x=xM $ is spanned by the vectors
		$ x_{1}x_{i_{2}}\dots x_{i_{k}}\otimes v_{j} $.
	\end{proof}
	
	The following lemma is the premise of Section~\ref{sec proj}, where the relationship between projectivity of module and its associated variety will be studied more in depth.
	\begin{lemma}\label{proj_lemma}
		Suppose that $\gg_{0}$ is reductive (i.e., $\gg$ is quasireductive).  If $M$ is projective in $\cF(\gg)$ then we have $X_M=\{0\}$.
	\end{lemma}
	\begin{proof}
		This follows from (1) of Lemma~\ref{lm2} using that $M$ will be a direct summand of $\Ind_{\gg_0}^{\gg}\Res_{\gg_0}^{\gg}M$.
	\end{proof}
	
	\begin{remark}
		There is a natural action of $G_0\times\mathbb{G}_m$ on the associated variety $X$ of $\gg$, where the one-dimensional torus $\mathbb{G}_m$ acts by scaling.  
		
		For $\lambda\in\mathbb{G}_m$, it is easy to check we have an equality of functors $\DS_{\lambda x}=\DS_{x}$.  For $g\in G_0$, the functors $\DS_{gx}$ and $\DS_x$ are isomorphic, in a suitable sense, when we restrict to finite-dimensional modules.
	\end{remark}


	\section{The universal enveloping algebra and central characters}\
	
	In this section, $\gg$ denotes a finite-dimensional Lie superalgebra.
	Let  $Z(\gg)$ (respectively, $Z(\gg_x)$) denote the center  of the universal enveloping algebra $\mathcal{U}(\gg)$ (respectively, $\mathcal{U}(\gg_x)$). 
	
	For each central character $\chi: Z(\gg)\to\mathbb{C}$, we denote by $\Mod_\chi^r(\gg)$ 
	the full subcategory of $\Mod(\gg)$ consisting of the modules 
	that are annihilated by $(z-\chi(z))^r$ for every $z\in Z(\gg)$.  We set
	\[
	\Mod_\chi(\gg)=\bigcup\limits_{r=1}^\infty\Mod_\chi^r(\gg),
	\] 
	and we say that a $\gg$-module $M$ admits central character $\chi$ if $M$ lies in $\Mod_\chi(\gg)$.
	By Dixmier's generalization of Schur's Lemma (see \cite{Dix}), each simple module lies in $\Mod_\chi^1(\gg)$ for a suitable central character $\chi$.

	Symmetrization gives an isomorphism $\mathcal{U}(\gg)\simeq S(\gg)$ as $\ad_\gg$-modules. Then since $M\mapsto M_x$ is a tensor functor we have
	$$\mathcal{U}(\gg)_x\simeq \mathcal{U}(\gg_x).$$
	
	Observe that $\ad_x(\mathcal{U}(\gg))$ is an ideal in $\mathcal{U}(\gg)^{\ad_x}$ and consider the canonical homomorphism of algebras  $\pi:\mathcal{U}(\gg)^{\ad_x}\to \mathcal{U}(\gg_x)$.
	Then we have a homomorphism
	$$Z(\gg)\hookrightarrow \mathcal{U}(\gg)^{\ad_x}\xrightarrow{\pi}\mathcal{U}(\gg_x)$$
	and since $Z(\gg_x)=\mathcal{U}(\gg_x)^{ad \gg_x}$ we have a well defined homomorphism
	$$\eta_x: Z(\gg)\to Z(\gg_x).$$
	The dual map of central characters 
	\begin{equation}\label{def eta^*}
		\eta_x^{*}\colon \operatorname{Hom} \left(Z\left(\gg_{x}\right),{\CC}\right) \to \operatorname{Hom} \left(Z(\gg),{\CC}\right)
	\end{equation}
	is very important due to the following statement.

	\begin{proposition}\label{maincch}
		Take  $M\in \Mod_\chi^r(\gg)$.
		\begin{enumerate}
			\item
			If $\eta_x$ is surjective, then $\DS_x(M)$ lies in $\Mod_{(\eta_x^*)^{-1}\chi}^r(\gg_x)$.
			
			\item
			For each simple subquotient $L'$ of $\DS_x(M)$ there exists
			$\chi'\in (\eta_x^*)^{-1}(\chi)$ such that  $L'\in  \Mod_{\chi'}^1(\gg_x)$.
			In particular, $\DS_x(M)=0$ if $\chi\not\in \, Im(\eta_x^*)$.
		\end{enumerate}
	\end{proposition}
	\begin{proof}
		View $M$ as a $\mathcal{Z}(\gg)$-module. Note that $\DS_x(M)$ 
		can be viewed as a subquotient
		of this module. 
		Take $z\in\operatorname{Ker}\chi$ and set  $z_x:=\eta_x(z)$.
		Since $M\in\Mod_\chi^r(\gg)$ one has 
		$z^r M=0$, so
		$z_x^r \DS_x(M)=0$. This gives (i). For (ii)
		take $\chi':\mathcal{Z}(\gg_x)\to\mathbb{C}$ such that  
		$\operatorname{Ker}\chi' L'=0$. Then $(z_x-\chi'(z_x))L'=0$, so
		$\chi'(z_x)=0$. Hence
		$\eta_x(\operatorname{Ker} \chi)\subset \operatorname{Ker}\chi'$ which implies
		$\eta_x(\chi')=\chi$ as required.
	\end{proof}
	
	\begin{corollary} If $M$ admits a central character $\chi$ and $M_x$ has a subquotient admitting a central character $\zeta$,
		then $\chi=\eta^*(\zeta)$.
	\end{corollary}

	Using Proposition~\ref{super_mult_preserved_ds}, we obtain the following interesting corollary.
	
	\begin{corollary}
		Assume that $\gg_x$ can be embedded into $\gg$ (i.e.,~$\gg^x=\gg_x\ltimes [x,\gg]$). 
		If $M\in \Mod_\chi(\gg)$ and $L'\in \Mod_{\chi'}(\gg_x)$ is a simple $\gg_x$-module  such that $[\Res^{\gg}_{\gg_x} M:L']<\infty$ and $[\Res^{\gg}_{\gg_x} M:\Pi(L')]<\infty$, then 
		\[
		[\Res^{\gg}_{\gg_x} M:L']=[\Res^{\gg}_{\gg_x} M:\Pi(L')]\ \text{ if }\ \chi'\not\in (\eta_x^*)^{-1}(\chi).
		\]
	\end{corollary} 
	
	\subsection{The involutions $\sigma_x$ for classical Lie superalgebras}\label{involutions section}
	The maps $\eta_x$ for classical Lie superalgebras were described in~\cite{G3} using the results of \cite{Ser1}, \cite{K2}, \cite{Ser2}.  There is a nice uniform description of the image, which requires us to introduce an involution $\sigma_x$ on $\gg_x$.  We note that for $\pp(n)$ the center is always trivial, however we will introduce an involution $\sigma_x$ for later use.
	
	\begin{itemize}
		\item For $\gg=\gg\ll(m|n),\mathfrak{osp}(2m+1|2n),\qq(n),\pp(n)$ or $G(3)$, we declare the involution $\sigma_x$ on $\gg_x$ to be trivial, i.e.,~$\sigma_x=\text{Id}$.
		\item For $\gg=D(2|1;a)$ and $x\not=0$, one has $\gg_x=\mathbb{C}z$, and we set $\sigma_x=-\text{Id}$. 
		\item For $\gg=F(4)$ and $x\not=0$, 
		one has $\gg_x\cong \mathfrak{sl}_3$, and $\sigma_x$ is induced by the  involution of the Dynkin diagram of $\mathfrak{sl}_3$. 
		\item For $\gg=\mathfrak{osp}(2m|2n)$  one has $\gg_x=\mathfrak{osp}(2(m-s)|2(n-s))$; we set $\sigma_x=\text{Id}$ if $m-s=0$, and if $m-s>0$, $\sigma_x$ is induced by the involution of one of the Dynkin diagrams of $\gg_x$.
	\end{itemize}

	\begin{remark}\label{inv_rem} Let us give another description of the involution $\sigma_x$. Consider
		an embedding $x\in\mathfrak{sl}(1|1)\subset\gg$ as in the proof of Theorem~\ref{th7}. Let $K$ denote the normalizer of $\mathfrak{sl}(1|1)$ inside the adjoint supergroup of $\gg$.  Then $K$ has a normal subgroup with the Lie superalgebra $\gg_x$.  
		The image of the natural homomorphism 
		$K\to \operatorname{Aut}\gg_x$ is disconnected
		if $\gg=\mathfrak{osp}(2m|2n)$, $m-s>0$, $D(2,1;a)$
		or $F_4$. In these cases, the image is a semidirect
		product of the adjoint group of $\gg_x$ and $\mathbb{Z}_2$. The involution $\sigma_x$ is a generator of $\mathbb{Z}_2$.
		
	\end{remark}

	\begin{proposition}\label{image eta}
		For $\gg$ a classical Lie superalgebra, $x\in X$, and involution $\sigma_x$ on $\gg_x$ as above, we have
		\[
		\eta_x(\mathcal{Z}(\gg))=\mathcal{Z}(\gg_x)^{\sigma_x}.
		\]
	\end{proposition}
	
	\begin{remark}
		Although we don't give a proof of Proposition~\ref{image eta}, it can be used to give another proof of Theorem~\ref{th210}, using Proposition~\ref{thetabasic} below.
	\end{remark}
	
	In the following lemma $\gg$ is general, but with the obvious view toward the cases of interest above.

	\begin{proposition}\label{thetabasic}
		Assume that 
		$\Im\eta_x=\mathcal{Z}(\gg_x)^{\sigma_x}$, where $\sigma_x$ is an involutive automorphism of
		$\gg_x$. For any $\chi\in\Im\eta_x^*$ we have
		
		\begin{enumerate}
			\item
			the set  $(\eta_x^*)^{-1}(\chi)$ is of the form $\{\chi',\sigma_x(\chi')\}$;
			\item if $\sigma_x(\chi')\not=\chi'$, then
			$\DS_x(\Mod_{\chi}^1(\gg))\subseteq\Mod_{\chi'}^1(\gg_x)\oplus \Mod_{\sigma_x(\chi')}^1(\gg_x)$;
			\item
			if  $\sigma_x(\chi')=\chi'$, then  $\DS_x(\Mod_\chi^1(\gg))\subseteq\Mod_{\chi'}^2(\gg_x)$.
		\end{enumerate}
	\end{proposition}
	\begin{proof}
		We set 
		$$\mathfrak{m}:=\Ker\chi,\ \ \ I:=\mathcal{Z}(\gg_x)\eta_x(\mathfrak{m}),\ \ 
		A:=\mathcal{Z}(\gg_x)/I$$
		and denote by $\psi$ the canonical map $\mathcal{Z}(\gg_x)\to A$.
		The algebra $A$ inherits the action of $\sigma_x$
		and $A^{\sigma_x}=\mathcal{Z}(\gg_x)^{\sigma_x}/\eta_x(\mathfrak{m})\cong\mathbb{C}$, so
		$A=\mathbb{C}\oplus A_-$, where
		$A_-:=\{a\in A|\ \sigma_x(a)=-a\}$.
		
		The central characters in $(\eta_x^*)^{-1}(\chi)$ correspond
		to the maximal ideals of $A$: for each $\chi''\in (\eta_x^*)^{-1}(\chi)$
		the ideal $\Ker\chi''$ is a maximal ideal of $\mathcal{Z}(\gg_x)$; this ideal
		contains $I$ and $\psi(\Ker\chi'')$ is a maximal ideal in $A$.
		For each $N\in\Mod_\chi^1 (\gg)$ one has $\mathfrak{m} N=0$, so $I\DS_x(N)=0$.
		Hence $\DS_x(N)$ has a structure of an $A$-module.
		
		If $A_-=\mathbb{C}a$ with  $a^2=1$, then 
		$A$ has two maximal ideals $\mathbb{C}(1\pm a)$ (one has
		$1-a=\sigma(1+a)$). Taking $\mathfrak{m}':=\psi^{-1}(\mathbb{C}(1+a))$ we get
		$(\eta_x^*)^{-1}(\chi)=\{\chi',\sigma_x(\chi')\}$, where
		$\Ker\chi'=\mathfrak{m}'$. 
		One has 
		$\DS_x(N)=N'_+\oplus N'_-$,
		where $N'_{\pm}=\{v\in \DS_x(N)|\ (a\pm 1)v=0\}$. Therefore $\mathfrak{m}' N'_{+}=0$
		and $\sigma(\mathfrak{m}') N'_{-}=0$, so $\DS_x(N)$ 
		lies in $\Mod_{\chi'}^1(\gg_x) + \Mod_{\sigma_x(\chi')}^1 (\gg_x)$.

		Consider the case when  $A_-\not=\mathbb{C}a$ with  $a^2=1$. For any $a_1,a_2\in A_-$ one has $a_1a_2\in A^{\sigma_x}=\mathbb{C}$. 
		If  $a_1a_2=1$ for some $a_1,a_2\in A_-$, then
		for each $a\in A_-$  one has $a_2a\in\mathbb{C}$, so $a=a_1a_2a\in\mathbb{C}a_1$
		that is $A_-\not=\mathbb{C}a_1$, a contradiction.
		Therefore $a_1a_2=0$ for all $a_1,a_2\in A_-$, so $(A_-)^2=0$ and
		$A_-$ is the unique maximal ideal in $A$. Then $(\eta_x^*)^{-1}(\chi)=\chi'$
		where $\Ker\chi'=\psi^{-1}(A_-)$. Since $(A_-)^2 \DS_x(N)=0$
		we have $\DS_x(\Mod_chi^1(\gg))\subset\Mod_{\chi'}^2(\gg_x)$.
	\end{proof}


	\section{Description of $\gg_x$ for classical Lie superalgebras $\gg$}
	
	In this section, we describe   $\gg_x$ and realize $\gg_x$ as a subalgebra of $\gg$ for classical Lie superalgebras.

	\subsection{Iso-sets and defect}\label{sec isoset}
	Now we assume that $\gg_0$ is a reductive Lie algebra and $\gg_1$ is
	a semisimple $\gg_0$-module. Such Lie superalgebras are called
	quasireductive.

	For a quasireductive Lie superalgebra $\gg$, we may choose a Cartan subalgebra $\ft\subseteq\gg_{0}$ and obtain roots $\Delta\subseteq\ft^*\setminus\{0\}$ by considering its adjoint action on $\gg$.  We have subsets $\Delta_{i}\subseteq\Delta$ for $i=0,1$ consisting of roots which are either even or odd.  In particular we have a root decomposition
	\begin{equation}
		\gg=\hh\oplus \bigoplus_{\alpha\in\Delta}\gg_{\alpha}.
		\notag\end{equation}
	where $\hh$ denotes the centralizer of $\ft$ in $\gg$.  
	We write each $a\in\gg_i$ (for $i=0,1$) in the form
	$$a=\sum_{\alpha\in\supp(a)} a_{\alpha},\ \ \text{ where }
	\ a_{\alpha}\in\gg_{\alpha}\setminus\{0\},\ \ \supp(a)\subset \Delta_i\cup\{0\}.$$
	
	We say that $A\subset \Delta_1$ is an {\em iso-set} if the elements of $A$ are linearly independent and if for each $\alpha,\beta\in \Delta_1\cap (A\cup (-A))$ one has $\alpha+\beta\not\in\Delta_0$.
	We call the maximal cardinality of an iso-set the {\it defect} of $\gg$.
	We let $\cS$ denote the set of iso-sets in $\Delta_1$.  The Weyl group $ W $ of $ \gg_{0} $ acts on $ \cS $ in the obvious
	way. Put $ \cS_{k}=\left\{{A}\in \cS \mid |{A}|=k\right\} $, with $ \cS_{0}=\left\{\varnothing\right\} $.

	\subsection{Basic classical Lie superalgebras}\label{KM_description}
	
	Suppose $\gg$ is a basic classical Lie superalgebra (see ~\ref{rem: classical LSA}). If $\gg\neq\fgl(m|n)$, then $\gg$ is a simple  Kac--Moody superalgebra (see \cite{K1,H}). The Lie superalgebra $\mathfrak{gl}(m|n)$ has as an ideal $\mathfrak{sl}(m|n)$, and when $m\neq n$, $\mathfrak{sl}(m|n)$ is simple and $ \mathfrak{gl}\left(m|n\right)\cong\mathfrak{sl}\left(m|n\right)\oplus{\CC}$. The Lie superalgebra $\mathfrak{gl}(n|n)$ has a unique  simple subquotient $\mathfrak{psl}(n|n):=\fsl(n|n)/\Span\{\Id\}$.
	
	We fix a Cartan subalgebra $ \hh \subset \gg $. Then $\hh$ coincides with a Cartan subalgebra $\ft$ of $ \gg_{0} $, and each root space $ \gg_{\alpha} $ is one dimensional. In this case, the parity of $ \alpha\in\Delta $ is by definition the parity of the root space $ \gg_{\alpha} $.

	A finite-dimensional Kac--Moody superalgebra  has a nondegenerate symmetric invariant bilinear form $ \left(\cdot,\cdot\right) $. This form is not necessarily positive definite, and some roots can be isotropic.  For a non-isotropic root $ \beta $, we denote by $ \beta^{\vee} $ the element of $ \ft $ such that
	$ \alpha\left(\beta^{\vee}\right)=\frac{2\left(\alpha,\beta\right)}{\left(\beta,\beta\right)} $.  For an isotropic root $\beta$, set $\beta^\vee\in\ft$ to be the element of $\ft$ corresponding to $\beta$ under the isomorphism $\ft\to\ft^*$ induced by the form.
	
	\begin{remark}
		The notion of defect was originally introduced in \cite{KW} for Kac--Moody superalgebras.  Finite-dimensional Kac--Moody superalgebras are quasireductive, and in this case, the notion of iso-set corresponds to the well-known notion of isotropic set: a set of mutually orthogonal
		linearly independent isotropic roots in $ \Delta_{1} $.  One can see that the defect of $ \gg $ in these cases is equal to the dimension of maximal isotropic subspace  in  $\Span_{\RR}\Delta$.
		
		For finite-dimensional Kac-Moody Lie superalgebras the defect has the following geometric interpretation: it is given by the dimension of the geometric quotient of $\gg_{1}$ by $G_0$.  In fact in these cases, $S(\gg_{1})^{G_0}$ is a polynomial algebra, and the number of generators is given by the defect.  For $\gg=\qq(n),\pp(n)$, $S(\gg_{1})^{G_0}$ is again a polynomial algebra, except the number of generators for $\qq(n)$ is $n$ while the number of generators of $\pp(n)$ is $\lfloor\frac{n}{2}\rfloor$.
		
		In \cite{BKN1}, a cohomological definition of support varieties was given using the relative $\Ext$ functor.  There, they define the defect of a $\gg$ to be the dimension of $\Ext_{\mathcal{F}(\gg)}^{\bullet}(\mathbb{C},\mathbb{C})\cong S(\gg_{1})^{G_0}$.
	\end{remark}

	\subsection{The Lie superalgebras $\pp(n)$ and $\qq(n)$}\label{sec pq}
	
	\subsubsection{$\pp(n)$} The periplectic Lie superalgebra $\pp(n)$ and the queer Lie superalgebra $\qq(n)$ are quasireductive Lie superalgebras and can be realized as subalgebras of $\mathfrak{gl}(n|n)$. 
	
	With respect to a suitable basis, the periplectic Lie superalgebra $\gg=\pp(n)$ consists of block matrices of the form
	$$
	\left(\begin{array}{c|r}
		A & B\\
		\hline
		C & -A^t
	\end{array}
	\right),
	$$
	where $B$ is symmetric, $C$ is skew-symmetric,  and $\ft:=\hh_0$ is the diagonal Cartan subalgebra of $\gg_0 \cong \mathfrak{gl}(n)$.
	Then $\gg$ has a $\ZZ$-grading 
	$ \gg=\gg^{1}\oplus\gg^0\oplus\gg^{-1} $ such that $ \gg^0=\gg_0 $, 
	$ \gg_{1}=\gg^1\oplus\gg^{-1}$, and corresponding sets of roots 
	$\Delta_0=\Delta(\gg^{0})=\{\varepsilon_i-\varepsilon_j\mid 1\leq i \neq j \leq n \}$,
	$\Delta(\gg^{-1})=\{-(\varepsilon_i+\varepsilon_j)\mid 1\leq i < j \leq n \}$, and
	$\Delta(\gg^{1})=\{\varepsilon_i+\varepsilon_j\mid 1\leq i \leq j \leq n \}$. 
	Imposing the additional condition that  $\operatorname{tr} A =0$ defines the Lie superalgebra $\mathfrak{\mathbf{s}p}(n)$, which is simple when $n \geq 3$; however, 
	$\mathfrak{\mathbf{s}p}(n)$ does not admit a nondegenerate (even or odd) invariant bilinear form. 
	
	\subsubsection{$\qq(n)$} With respect to an appropriate basis of $\mathbb{C}^{n|n}$, the queer Lie superalgebra $\gg=\qq(n)$ consists of block matrices of the form
	\begin{equation}\label{eq TAB}
		T_{A,B}:=\left(\begin{array}{c|r}
			A & B\\
			\hline
			B & A
		\end{array}
		\right),
	\end{equation}
	such that $\ft:=\hh_0$ is the diagonal Cartan subalgebra of $\gg_0 \cong \mathfrak{gl}(n)$.
	The set of roots for $\mathfrak{q}(n)$ is $\Delta=\{\pm(\varepsilon_i-\varepsilon_j) \mid 1\leq i < j\leq n \}$, and each root $\alpha\in\Delta$ is both even and odd since $\dim (\gg_\alpha)_{0}=\dim (\gg_\alpha)_{1}=1$. 
	Imposing the additional condition $\operatorname{tr} B = 0$ defines the Lie superalgebra $\mathfrak{sq}(n)$, and since $\operatorname{Id}\in \mathfrak{sq}(n)$ we can also define the  Lie superalgebra $\mathfrak{psq}(n) := \mathfrak{sq}(n)/ \langle\operatorname{Id} \rangle$, which is simple for $n\geq 3$.

	\subsection{Table of defects}

	The defect of a classical Lie superalgebra is given in the following table.~
	\renewcommand{\arraystretch}{1.4}
	$$
	\begin{array}{|c|c|}
		\hline 
		\gg & \textbf{Defect} \\ \hline \hline 
		\mathfrak{gl}\left(m|n\right) & \min\{m,n\}\\ \hline 
		\mathfrak{sl}\left(m|n\right),\ m\neq n &
		\min\{m,n\} \\ \hline 
		\mathfrak{osp}\left(m|2n\right) &  \min\{\lfloor m/2\rfloor,n\} \\ \hline 
		\pp(n) & n \\ \hline 
		\qq(n) & \lfloor n/2\rfloor\\ \hline 
		D\left(2|1;a\right) & 1\\\hline 
		F(4) & 1\\ \hline 
		G(3) & 1\\ \hline
	\end{array}
	$$
	\renewcommand{\arraystretch}{1}

	\subsection{Description and realization of $\gg_x$ in $\gg$}\label{g x in g}
	
	Suppose  $\gg$ is classical Lie superalgebra.
	Let $A=\left\{\beta_{1},\dots ,\beta_{k}\right\} \in \cS$ be non-empty, and take
	$x=x_{\beta_1}+\dots +x_{\beta_k}$ where each  $ x_{\beta_i}\in\gg_{\beta_{i}}$ is nonzero. If $\gg=\pp(n)$,  let $s$ denote the number of roots $\beta\in A$ of the form $2\varepsilon_j$.  Note that by Section~\ref{subsec orbits}, all elements of $X$ are $G_0$-conjugate to an element of this form.
	
	The following table describes $\gg_x$. 
	
	\renewcommand{\arraystretch}{1.4}
	$$
	\begin{array}{|c|c|}
		\hline 
		\gg & \gg_x \\ \hline \hline 
		\mathfrak{gl}\left(m|n\right) & \mathfrak{gl}\left(m-k|n-k\right)\\ \hline 
		\mathfrak{sl}\left(m|n\right),\ m\neq n &
		\mathfrak{sl}\left(m-k|n-k\right)\\ \hline 
		\mathfrak{osp}\left(m|2n\right) & \mathfrak{osp}\left(m-2k|2n-2k\right)\\ \hline
		\pp(n) & \pp(n-(2k-s))  \\ \hline 
		\qq(n) & \qq(n-2k)\\ \hline 
		D(2|1;a) & \CC\\\hline 
		F(4) & \mathfrak{sl}\left(3\right)\\ \hline 
		G(3) & \mathfrak{sl}\left(2\right)\\ \hline
	\end{array}
	$$
	\renewcommand{\arraystretch}{1}
	Note that in the last three rows the defect of $\gg$ is $1$, so $k=1$.

	The functor $\DS_x$ reduces the defect of $\gg$ by a non-negative integer which is called the {\em rank of $x$}, that is, $\rank x:= \defect \gg - \defect \gg_x$.

	\begin{definition}\label{rank_defn}
		Let $\gg$ be one of the Lie superalgebras listed in the above table, and let $x\in X$.  Then the rank of $x$ is as follows:
		\begin{itemize}
			\item if $x=0$, then $\rank x=0$;
			\item if $x\neq0$ and $\gg$ is an exceptional Lie superalgebra, then $\rank x=k=1$;
			\item if $\gg$ is not exceptional and $\gg\neq \pp(n)$, then $\rank x=k$;
			\item if $\gg=\pp(n)$, then $\rank x=2k-s$.
		\end{itemize}
	\end{definition}

	\begin{remark}
		Note that if $\gg\neq \pp(n)$ then $\rank x=k$, the size of $\mathcal S$.
	\end{remark}

	\begin{remark}
		For $\gg\ll(m|n)$, $\mathfrak{sl}(m|n)$ and $\pp(n)$, we observe that $\rank x$ is given by the rank of $x$ as a linear operator acting on the standard representation.  For $\mathfrak{osp}(m|2n)$ and $\qq(n)$, we have that $\rank x$ is half the rank of $x$ as linear operator in the standard representation.
	\end{remark}
	
	Let $\gg$ be a  classical Lie superalgebra with $\gg\neq \pp(n)$, and let $x\in X$.  We now explain how to realize $ \gg_{x} $ as a subalgebra of $\gg$, in such a way that $\gg^x = \gg_x\ltimes [x,\gg]$.  For $\pp(n)$ we will also have such an embedding, but the construction is more ad-hoc, so we state it separately. 
	
	Thus we assume $\gg\neq\pp(n)$, and as above, we let $ {A}=\left\{\beta_{1},\dots ,\beta_{k}\right\}\in\cS $ and $ x=x_{\beta_1}+\dots +x_{\beta_k} $ for some nonzero $ x_{\beta_i}\in(\gg_{\beta_{i}})_{1} $.  Let
	$ y=y_{\beta_1}+\dots +y_{\beta_k} $ for some nonzero $ y_{\beta_i}\in(\gg_{-\beta_{i}})_{1}$, and set $ h_{\beta_i}=\left[x_{\beta_i},y_{\beta_i}\right] $, $ h=\left[x,y\right] $. Clearly,
	$ h=h_{\beta_1}+\dots +h_{\beta_k} $, and $ h,x,y $ generate an $ \mathfrak{sl}\left(1|1\right) $-subalgebra in $ \gg $.  We choose the $y_{\beta_i}$ such that $h$ is generic, meaning that it satisfies
	\[
	\ker\ad_h=\ker\ad_{h_{\beta_1}}\cap\cdots\cap\ker h_{\beta_{k}}.
	\]
	We may decompose $\gg$ with respect to the adjoint action of $h$ giving
	$
	\gg=\oplus_{\mu}\gg^{\mu},
	$
	where
	$
	\gg^{\mu}=\left\{g\in\gg \mid \left[h,g\right]=\mu(h) g\right\}.
	$ 
	In particular, $\gg^0=\Ker \ad_h$, and this is in fact a decomposition of $\mathfrak{sl}(1|1)$-modules. 
	
	For each $\beta_i\in A$, set $ \hh_{\beta_i}=\left[(\gg_{\beta_i})_1,(\gg_{-\beta_i})_{1}\right] $.   Set
	\begin{equation}\label{A perp}
		A^\perp=\Ker\beta_1^\vee\cap\cdots\cap\Ker\beta_k^\vee\subset\hh^*.
	\end{equation}

	Define 
	$$ 
	\hh_{{A}}:=\hh_{\beta_{1}}\oplus\dots \oplus\hh_{\beta_{k}},
	\quad
	\hh_A^\perp:=\Ker\beta_1\cap\dots\cap\Ker\beta_k,
	\quad
	\Delta_x:={A}^{\perp}\cap(\Delta\backslash\left({A}\cup-{A}\right)).
	$$
	We have the following.   
	
	\begin{proposition} \label{lm202}\myLabel{lm202}\relax  Suppose $ \gg\neq\pp(n) $ is a classical Lie superalgebra, and let $A\in \cS$ with corresponding $x\in X$.   Then $\gg_{x}$ is isomorphic to the root subalgebra generated by $\{\gg_{\alpha}\}_{ \alpha\in\Delta_x}$ and a splitting of $\hh_x\cong \hh_{A}^{\perp}/\hh_{A}$ of  $\hh_{A}$,  and this splitting $\hh_x$ will  define a Cartan subalgebra of $\gg_x$.  If we identify $\gg_x$ with its image in $\gg$ we have  $\gg^x = \gg_x\ltimes [x,\gg]$.
	\end{proposition}
	
	\begin{proof} First, note that there is an isomorphism
		\begin{equation}
			\gg_{x}\cong(\gg^{0}\cap \gg^x)/(\gg^{0}\cap\left[x,\gg\right]).
			\notag\end{equation}
		We observe that
		\begin{equation}
			\gg^{0}\cap \gg^x=\hh_{{A}}^{\perp}\oplus \bigoplus\limits_{\alpha\in{A}^{\perp}\cap(\Delta\backslash-{A})} \gg_{\alpha}\text{,  }\hspace{1cm}  \gg^{0}\cap\left[x,\gg\right]=\hh_{{A}}\oplus\gg_{\beta_{1}}\oplus\dots \oplus\gg_{\beta_{k}}.
			\notag\end{equation}
		The above equalities follow from the representation theory of $\mathfrak{sl}(1|1)$.  Now it is clear that one can choose $ \hh_{x} $ in such a way that 
		$ \gg_{x}=\left(\hh_{x}\oplus\oplus_{\alpha\in\Delta_x} \gg_{\alpha} \right) $ is a subalgebra of $\gg$.
	\end{proof}
	
	Separately, we have:
	\begin{proposition}\label{embedding p}
		Let $\gg=\pp(n)$ and choose $x\in X$ of rank $r$ arising from a subset $A\in\SS$.  Then $A$ lies in the span of $\epsilon_{i_1},\dots,\epsilon_{i_r}$ for a unique set of indices $I=\{i_1,\dots,i_r\}\subseteq[n]:=\{1,\dots,n\}$. Write $\pp(n-r)$ for the root subalgebra corresponding to the weights $\epsilon_i$ for $i\in[n]\setminus I$.  Then we have a natural isomorphism $\gg_x\cong \pp(n-r)$, and $\gg^x=\pp(n-r)\ltimes[x,\gg]$.
	\end{proposition}
	\begin{proof}
		Straightforward check.
	\end{proof}
	
	\begin{remark}
		Propositions~\ref{lm202} and \ref{embedding p} have the following useful application:
		if $\ft$ acts diagonally on $N$ and 
		$\Omega(N)=\{\nu\in\ft^*|\ N_{\nu}\not=0\}$, then for $x$ as in the propositions one has
		\begin{equation}\label{Omega}
			\Omega(\DS_x(N))\subset (\Omega(\DS_x(N)))|_{\ft_x}.
		\end{equation}
		
	\end{remark}

	\begin{remark}
		By Theorem~\ref{th1} and Proposition~\ref{p orbits}, if $\gg$ is classical then all $G_0$-orbits on $X$ contain an element $x$ arising from an iso set $A\in\mathcal{S}$.  Thus Propositions~\ref{lm202} and \ref{embedding p} imply that our table in Section~\ref{g x in g} is accurate.
	\end{remark}
	

	\section{Geometry of $ X $ for classical Lie superalgebras}\label{sec geo}
	
	In this section, we study the $G_0$-orbits on $X$, and for $\gg\neq \pp(n)$, we prove an important bijection between the $G_0$-orbits on $X$ and the $W$-orbits on $\cS$. We also describe the orbits in the $\pp(n)$ case. Next, we study the  stabilizer and normalizer of $x$ in $G_0$.  Finally, we give a dimension formula for the $G_0$-orbits on $X$ for basic classical Lie superalgebras.
	
	\subsection{$G_0$-orbits on $X$}\label{subsec orbits}

	\begin{theorem}\label{th1}\myLabel{th1}\relax 
		Suppose $ \gg $ is a basic classical Lie superalgebra or $\gg=\qq(n)$. Then 
		there are finitely many $ G_{0} $-orbits on $ X $, and these orbits
		are in one-to-one correspondence with $ W $-orbits in $ \cS $
		via the map
		\begin{equation}\label{map Phi}
			\Phi:\cS \to X/G_{0}
		\end{equation} defined by
		$$ 
		{A}=\left\{\beta_{1},\dots ,\beta_{k}\right\}\mapsto G_{0}x , 
		$$
		where $ x=x_{\beta_1}+\dots +x_{\beta_k}$ for some nonzero $ x_{\beta_i}\in\gg_{\beta_{i}}$.
	\end{theorem}

	\begin{proof} 
		To see that $ \Phi\left({A}\right) $ does not depend on a choice of $ x_{\beta_i} $
		note that since $ \beta_{1},\dots ,\beta_{k} $ are linearly independent, for any other choice
		\begin{equation}
			x'=\Sigma x_{\beta_i}'=\Sigma c_{i}x_{\beta_i}
			\notag\end{equation}
		there is $ h\in\hh $ such that $ c_{i}=e^{\beta_{i}\left(h\right)} $ and therefore
		\begin{equation}
			x'=\exp \left(\operatorname{ad}_h\right)\left(x\right).
			\notag\end{equation}
		If $ {\text B}=w\left({A}\right) $ for some $ w\in W $, then clearly $ \Phi\left({\text B}\right) $ and $ \Phi\left({A}\right) $ belong to the same
		orbit. Therefore $ \Phi $ induces the map $ \overline{\Phi}:\cS/W \to X/G_{0} $. We check
		case by case that $ \overline{\Phi} $ is injective and surjective.
		
		If $ \gg $ is $ {\mathfrak s}{\mathfrak l}\left(m|n\right) $ or $ \gg{\mathfrak l}\left(n|n\right) $, $ \gg $ has a natural $ {\ZZ} $ grading
		$ \gg=\gg^{1}\oplus\gg^0\oplus\gg^{-1} $ such that $ \gg_{0}=\gg^0$, 
		$ \gg_{1}=\gg^1\oplus\gg^{-1}$. The orbits of $ W $ on $ \cS $ are
		enumerated by the pairs of numbers $ \left(p,q\right) $, where $ p=|{A}\cap\Delta\left(\gg^1\right)| $,
		$ q=|{A}\cap\Delta(\gg^{-1})| $. The orbits of $ G_{0} $ on $ X $ are enumerated by the same pairs of
		numbers $ \left(p,q\right) $ in the following way. If $ x=x^{+}+x^{-} $, where $ x^{\pm}\in\gg^{\left(\pm1\right)}$, then
		$ p=\operatorname{rank}\left(x^{+}\right) $, $ q=\operatorname{rank}\left(x^{-}\right) $. We can see by the construction of $ \overline{\Phi} $, that $ \overline{\Phi} $ maps
		$ \left(p,q\right) $-orbit on $ \cS $ to the $ \left(p,q\right) $-orbit on $ X $.
		
		Let $ \gg=\mathfrak{osp}\left(m|2n\right) $. If $ m=2l+1 $ or $ m=2l $ with $ l>n $, then the $ W $-orbits on $ \cS $ are
		in one-to-one correspondence with $ \left\{0,1,2,\dots ,\min \left(l,n\right)\right\} $. Namely, $ {A} $ and $ {\text B} $ are
		on the same orbit if they have the same number of elements. As it was shown
		in \cite{Gr2}, $ X $ can be identified with the set of all linear maps $ x:{\CC}^{m} \to
		{\CC}^{2n} $, such that $ \operatorname{Im} x $ is an isotropic subspace in $ {\CC}^{2n} $ and $ \operatorname{Im} x^{*} $ is an
		isotropic subspace in $ {\CC}^{m} $. Furthermore, $ x,y\in X $ belong to the same $ G_{0} $-orbit iff
		$ \operatorname{rank}\left(x\right)=\operatorname{rank}\left(y\right) $. One can see that rank $ \Phi\left({A}\right)=|{A}| $.
		
		Now let $ \gg=\mathfrak{osp}\left(2l|2n\right) $ where $ l\leq n $. If $ {A},{\text B}\in \cS $ and $ |{A}|=|{\text B}|<l $, then $ {A} $ and $ {\text B} $
		are on the same $ W $-orbit. In the same way if $ \operatorname{rank}\left(x\right)=\operatorname{rank}\left(y\right)<l $, then $ x $ and $ y $
		are on the same $ G_{0} $-orbit. However, the set of all $ x\in\gg_{1} $ of maximal rank
		splits into two orbits, since the Grassmannian of maximal isotropic subspaces
		in $ {\CC}^{2l} $ has two connected components. In the same way $ \cS_l $ splits into two
		$ W $-orbits. Hence in this case again $ \overline{\Phi} $ is a bijection.
		
		If $ \gg $ is one of exceptional Lie superalgebras $ D\left(2|1;a\right) $, $ G(3) $ or $ F(4)$, then
		the direct calculation shows that $ X $ has two $ G_{0} $-orbits: $ \left\{0\right\} $ and the orbit
		of a highest vector in $ \gg_{1} $. The set $ \cS $ also consists of two $ W $-orbits: $ \varnothing $ and
		the set of all isotropic roots in $ \Delta $.

		Finally, let $\gg=\qq(n)$. Then $\gg$ is isomorphic to the subalgebra of $\mathfrak{gl}(n|n)$ consisting of block matrices of the form $T_{A,B}$ in (\ref{eq TAB}) and 
		$X=\{T_{0,B}\mid B^2=0\}$. So $G_0$ is isomorphic to $GL(n)$ and acts by conjugation on $B$. Thus the $G_0$-orbits correspond to Jordan forms. 
		If for $r=0,1,\ldots,[\frac{n}{2}]$, we set $\cS_r:=\{\beta_{n-1-2i}\}_{i=0}^{r-1}$   and fix an element $x_r\in \cS_r$ with $\supp (x_r) = \cS_r$ ($x_0:=0$), then the elements $x_0,x_1,\ldots,x_{[\frac{n}{2}]}$ form a set of representatives for $G_0$-orbits in $X$.
	\end{proof}
	
	Theorem~\ref{th1} does not hold for $\pp(n)$; however, we have the following proposition, whose proof is an exercise in linear algebra which we omit.
	
	\begin{proposition}\label{p orbits}
		For $\gg=\pp(n)$, the $G_0$-orbits on $X$ are indexed by a pair or nonnegative integers $(r,s)$ such that $r+2s\leq n$.  An explicit representative of the orbit labeled by $(r,s)$ is given by
		\[
		x=x_{2\epsilon_1}+\dots+x_{2\epsilon_r}+x_{-\epsilon_{r+1}-\epsilon_{r+2}}+\dots+x_{-\epsilon_{r+2s-1}-\epsilon_{r+2s}},
		\] 
	\end{proposition} 
	
	\begin{remark} \label{rem393}\myLabel{rem393}\relax  Note that for a finite-dimensional Kac--Moody superalgebra $\gg$ the representation of $ G_{0} $ in $ \gg_{1} $
		is symplectic and multiplicity free (see \cite{Kn}). The cone $ X $ is the
		preimage of 0 under the moment map $ \gg_{1} \to \gg_{0}^{*} $.
	\end{remark}
	

		\subsection{The stabilizer and normalizer of $x$ in $G_0$}
		
		\begin{lemma}\label{stabilizer}
			Let $\gg$ be a basic classical Lie superalgebra.
			Let $x\in X$. The stabilizer $G_0^x$ of $x$ in $G_0$ is connected.
			Furthermore, $G_0^x$ is a semidirect product of a reductive group $K$ and the normal unipotents
			subgroup $U$ with Lie algebras $(\gg_x)_0$ and $[x,\gg_1]$, respectively.
		\end{lemma}
		\begin{proof}
			The second assertion follows from Proposition~\ref{lm202}. 
			To check the connectedness we use the explicit construction of orbits given in   the proof of Theorem~\ref{th1}.
			
			Let $\gg=\mathfrak{gl}(m|n)$. Then $G_0=GL(m)\times GL(n)$, consider the parabolic subgroups $P_1\subset GL(m)$ which stabilizes the flag
			$\Im x^+\subset\Ker x^-$ and $P_2\subset GL(n)$ which stabilizes the flag
			$\Im x^-\subset\Ker x^+$. The Levi subgroup of $K_1$ of $P_1$ is isomorphic to $GL(p)\times GL(q)\times GL(m-k)$ and the Levi subgroup of $K_2$ of $P_2$ is
			isomorphic to $GL(p)\times GL(q)\times GL(n-k)$. Let $K\simeq K_1\times K_2$ is the subgroup isomorphic to $GL(p)\times GL(q)\times GL(m-k)\times GL(n-k)$
			where $GL(p)$ and $GL(q)$ are embedded diagonally. Then $G_0^x=K\rtimes U$, where $U$ is the unipotent normal subgroup of $P_1\times P_2$.
			
			Let $\gg=\mathfrak{osp}(m|2n)$. Then $G_0=SO(m)\times SP(2n)$, consider the parabolic subgroups $P_1\subset SO(m)$ which stabilizes $\Im x^*$ 
			and $P_2\subset SP(n)$ which stabilizes
			$\Im x$. The Levi subgroup of $K_1$ of $P_1$ is isomorphic to $GL(k)\times SO(m-2k)$ and the Levi subgroup of $K_2$ of $P_2$ is
			isomorphic to $GL(k)\times SP(2n-2k)$. Let $K\simeq K_1\times K_2$ is the subgroup isomorphic to $GL(k)\times SO(m-2k)\times SP(2n-2k)$
			where $GL(k)$ is  embedded diagonally. Then $G_0^x=K\rtimes U$, where $U$ is the unipotent normal subgroup of $P_1\times P_2$.
			
			In all exceptional cases $x$ is a highest weight vector in $\gg_1$ and $G_0^x$ is a subgroup of codimension $1$ in the parabolic subgroup $P\subset G_0$,
			the latter is the stabilizer of $\CC x$ in the projectivization of $\gg_1$.
		\end{proof}

		\begin{remark} It follows from above proof that there exists a parabolic subgroup $P\subset G_0$ such that $G_0^x$ is a subgroup of $P$ and the maximal normal unipotent
			subgroup of $G_0^x$ is equal to that of $P$. 
		\end{remark}
		
		\begin{remark}
			For the $\qq(n)$ it remains true that $G_0^x$ is connected; this follows from a result of Springer and Steinberg,  (see Chpt. IV, Sec. 1.7 of \cite{SpSt}).  \end{remark}

		We write $N_0^x$ for the normalizer of $x$ in $G_0$.
		
		\begin{corollary}\label{normalizer} For $\gg$ basic classical and for $\gg=\qq(n)$,
			$N_0^x$ is connected.
		\end{corollary}
		
		\begin{proof}
			We have an exact sequence 
			\[
			1\to G_0^x\to N_0^x\xrightarrow{\alpha} \mathbb{G}_m\to 1,
			\]
			where  $\mathbb{G}_m$ is the one-dimensional torus and
			$\alpha(g)=(g\cdot x)/x$, where $g\cdot x$ is the action of $g$ on $x$.  By a case-by-case check, the map $\alpha$ is always surjective and split.  We may now use Lemma~\ref{stabilizer}.
		\end{proof}
		
		\begin{remark}
			For $\lambda\in\CC^\times$, we have an equality of functors $\DS_x=\DS_{\lambda x}$.  It follows that $N_0^x$ acts naturally on the functor $\DS_x$. We have shown that
			\[
			N_0^x=G_0^x\rtimes\mathbb{G}_m.
			\]
			For $\gg$ basic classical or $\qq(n)$, we have shown $G_0^x$ is connected, and thus its symmetries are encompassed in $(\gg_x)_0$.  It follows that the only additional symmetries we obtain in this fashion are from the extra action of $\mathbb{G}_m$.
		\end{remark}

		\begin{remark}
			We note that for the $\gg=\pp(n)$, Lemma~\ref{stabilizer} and Corollary~\ref{normalizer} are not true for all choices of $x$.  Issues arise due to the orthogonal group being disconnected and the lack of a splitting for $\alpha$ in general.
		\end{remark}

		\subsection{Dimension of the $ G_{0} $-orbits on $ X $}
		
		Throughout the rest of Section~\ref{sec geo}, we assume that $\gg$ is a basic classical Lie superalgebra.
		We recall the notation $ \Phi:\cS \to X/G_{0} $ introduced in 
		(\ref{map Phi}). 
		
		Using the explicit description of $ G_{0} $-orbits on $ X $ and the description of
		root systems, which can be found in \cite{K1}, one can check the following statements case by case. We omit this computation here. 
		
		\begin{lemma} \label{lm110}\myLabel{lm110}\relax  Let $ {A},{\text B}\in \cS. $
			\begin{enumerate}
				\item
				If $ \alpha\in\Delta $ is a linear combination of roots from $ {A} $, then $ \alpha\in{A}\cup-{A} $;
				\item
				If $ |{A}|\leq|{\text B}| $, then there exists $ w\in W $ such that $ w\left({A}\right)\subset{\text B}\cup-{\text B} $;
				\item
				$ \Phi\left({A}\right) $ lies in the closure of $ \Phi\left({\text B}\right) $ iff $ w\left({A}\right)\subset{\text B} $ for some $ w\in W $.
			\end{enumerate}
		\end{lemma}

		Recall  the definition of   $ {A}^{\perp} $ from (\ref{A perp}). In the basic classical case, $ {A}^{\perp} $ is the set of all weights orthogonal to $ {A} $ with respect to
		the standard form on $ \hh^{*} $.

		\begin{theorem} \label{th7}\myLabel{th7}\relax  Let $ {A}\in \cS $. Then $ \dim  \Phi\left({A}\right)=\frac{|\Delta_{1}\backslash{A}^{\perp}|}{2}+|{A}| $.
		\end{theorem}
		
		\begin{proof} Let $ {A}=\left\{\beta_{1},\dots ,\beta_{k}\right\} $, $ x=x_{\beta_1}+\dots +x_{\beta_k} $ for some choice of $ x_{\beta_i}\in\gg_{\beta_{i}} $,
			$ y=y_{\beta_1}+\dots +y_{\beta_k} $ for some $ y_{\beta_i}\in\gg_{-\beta_{i}} $. Let $ h=\left[x,y\right] $, $ h_{\beta_i}=\left[x_{\beta_i},y_{\beta_i}\right] $. Clearly,
			$ h=h_{\beta_1}+\dots +h_{\beta_k} $ and $ h,x,y $ generate an $ \mathfrak{sl}\left(1|1\right) $-subalgebra in $ \gg $. As a module over this subalgebra
			$ \gg $ has a decomposition
			\begin{equation}
				\gg=\oplus_{\mu}\gg^{\mu},
				\notag\end{equation}
			where
			\begin{equation}
				\gg^{\mu}=\left\{g\in\gg \mid \left[h,g\right]=\mu(h) g\right\}.
				\notag\end{equation}
			Note that
			\begin{equation}
				\dim  \left[\gg,x\right]=\sum_{\mu}\dim  \left[\gg^{\mu},x\right],
				\notag\end{equation}
			and from the description of irreducible $ \mathfrak{sl}\left(1|1\right) $-modules for $ \mu\not=0 $
			\begin{equation}
				\dim  \left[\gg^{\mu},x\right]=\frac{\dim  \gg^{\mu}}{2}.
				\notag\end{equation}
			On the other hand, for $ \mu\not=0$,  $\sdim  \gg^{\mu}=0 $ and therefore
			\begin{equation}
				\dim  \gg^{\mu}=2 \dim  \gg_{1}^{\mu}.
				\notag\end{equation}
			Observe that for a generic choice of $ x_{\beta_{i}}\in\gg_{\beta_{i}} $, $ \gg_{\alpha}\subset\gg^{0} $ iff $ \left(\alpha,\beta_{i}\right)=0 $ for
			all $ i\leq k $. Indeed, for generic
			choice of $ x_{\beta_i} $ the condition $ \alpha\left(h\right)=0 $ implies $ \alpha\left(h_{\beta_i}\right)=0 $ for all $ i $, and
			therefore $ \left(\alpha,\beta_{i}\right)=0 $ for all $ i $. Hence
			\begin{equation}
				\oplus_{\mu\not=0}\gg_{1}^{\mu}=\oplus_{\alpha\in\Delta_{1}\backslash{A}^{\perp}}\gg_{\alpha}
				\notag\end{equation}
			and
			\begin{equation}
				\sum_{\mu\not=0} \dim  \left[\gg^{\mu},x\right]=\sum_{\mu\not=0} \dim  \gg_{1}^{\mu}=|\Delta_{1}\backslash{A}^{\perp}|.
				\notag\end{equation}
			To calculate $ \dim  \left[\gg^{0},x\right] $ note that
			\begin{equation}
				\gg^{0}=\hh\oplus\oplus_{\alpha\in\Delta\cap{A}^{\perp}}\gg_{\alpha}.
				\notag\end{equation}
			We claim that
			\begin{equation}
				\left[\gg^{0},x\right]=\oplus_{i=1}^{k}{\CC}h_{\beta_i}\oplus\oplus_{i=1}^{k}\gg_{\beta_{i}},
				\notag\end{equation}
			hence $ \dim  \left[\gg^{0},x\right]=2k $. Indeed, if $ \left(\alpha,\beta_{i}\right)=0,\alpha\not=\pm\beta_{i} $ then $ \alpha\pm\beta_{i}\notin\Delta $. Therefore
			$ \left[x,\gg_{\alpha}\right]=0 $ for any $ \alpha\in\Delta\cap{A}^{\perp},\alpha\not=-\beta_{i} $. Furthermore, $ \left[x,\gg_{-\beta_{i}}\right]={\CC}h_{\beta_i} $ and $ \left[x,\hh\right]=\oplus_{i=1}^{k}\gg_{\beta_{i}} $.
			Thus, we obtain
			\begin{equation}
				\dim  \left[\gg,x\right]=|\Delta_{1}\backslash{A}^{\perp}|+2k.
				\label{equ99}\end{equation}\myLabel{equ99,}\relax 
			Now the statement will follow from the lemma.
			
			\begin{lemma} \label{lm71}\myLabel{lm71}\relax  $\sdim \left[\gg,x\right]=0$.
			\end{lemma}
			
			\begin{proof} Define an odd skew-symmetric form $\omega$ on $ \gg $ by
				\begin{equation}
					\omega\left(y,z\right):=\left(x,\left[y,z\right]\right).
					\notag\end{equation}
				Obviously the kernel of $ \omega $ coincides with the centralizer $ \gg^x $. Thus, $ \omega $ induces a
				non-degenerate odd skew-symmetric form on the quotient $ \gg/ \gg^x $. Hence $ \sdim\gg/ \gg^x=0$.
				But
				$ \left[\gg,x\right]\cong\Pi\left(\gg/\gg^x \right) $, which implies the lemma.\end{proof}
			
			Now Lemma~\ref{lm71} implies $ \dim \left[\gg_{0},x\right]=\frac{1}{2} \dim \left[\gg,x\right] $. Since $ \dim  G_{0}x=\dim \left[\gg_{0},x\right] $, 
			Theorem~\ref{th7} follows from~\eqref{equ99}.
		\end{proof}

		Theorem~\ref{th7} has the following corollaries.
		
		\begin{corollary} \label{cor73}\myLabel{cor73}\relax  If $ |{A}|=|{B}| $, then $ \dim  \Phi\left({A}\right)= \dim  \Phi\left({B}\right)$.
		\end{corollary}
		\begin{proof} Follows from Theorem~\ref{th1} and Lemma~\ref{lm110} (2).\end{proof}

		\begin{corollary} \label{cor71}\myLabel{cor71}\relax  Let $ d $ be the defect of $ \gg $. Then the irreducible
			components of $ X $ are in bijection with $ W $-orbits on $ \cS_{d}:=\left\{{A}\in \cS \mid |{A}|=d\right\}  $. If all odd roots of $ \gg $
			are isotropic, then the dimension of each component is equal to $ \frac{\dim 
				\gg_{1}}{2}=\frac{|\Delta_{1}|}{2} $.
		\end{corollary}
		
		\begin{proof} As follows from Theorem~\ref{th1} and Lemma~\ref{lm110} (3), each
			irreducible component is the closure of $ \Phi\left({A}\right) $ for a maximal $ {A}\in \cS $. By Lemma~\ref{lm110} $ \left(2\right) |{A}|=d $. Hence the first statement. Theorem~\ref{th7} immediately
			implies the statement about dimension.\end{proof}
		
		\begin{corollary} \label{cor72}\myLabel{cor72}\relax  If all odd roots of $ \gg $ are isotropic, then the
			codimension of $ \Phi\left({A}\right) $ in $ X $ equals $ \frac{|\Delta_{1}\cap{A}^{\perp}|}{2}-|{A}| $.
		\end{corollary}
		
		\begin{proof} The codimension  of $ \Phi\left({A}\right) $ in $ X $ equals $ \dim  X-\dim  \Phi\left({A}\right) $. Using Theorem~\ref{th7} and
			Corollary~\ref{cor71} we obtain
			\begin{equation}
				\operatorname{codim} \Phi\left({A}\right)=\frac{|\Delta_{1}|-|\Delta_{1}\backslash{A}^{\perp}|}{2}-|{A}|=\frac{|\Delta_{1}\cap{A}^{\perp}|}{2}-|{A}|.
				\notag\end{equation}
		\end{proof}

		Recall that 
		$ \gg_{x}=\gg^x /\left[x,\gg\right] $ and $M_x$ is a $\gg_x$-module, see Lemma~\ref{gx Mx}.
		
		\begin{lemma} \label{lm201}\myLabel{lm201}\relax  Let $ {\mathfrak m}^{\perp} $ denote the
			orthogonal complement to $ {\mathfrak m} $ with respect to the invariant form on $ \gg $. Then
			$ \left[x,\gg\right]^{\perp}=\gg^x $.
		\end{lemma}
		
		\begin{proof} Let $ u\in \gg^x $, $ v\in\left[x,\gg\right] $. Then $ v=\left[x,z\right] $ and
			\begin{equation}
				\left[u,\left[x,z\right]\right]=\left(-1\right)^{p\left(u\right)}\left[x,\left[u,z\right]\right]\in\left[x,\gg\right].
				\notag\end{equation}
			Now the statement follows from the identity
			\begin{equation}
				\left(u,\left[x,z\right]\right)=\left(\left[u,x\right],z\right).
				\notag\end{equation}
		\end{proof}


		\section{Central characters and atypicality for  classical $\gg\neq\pp(n)$}

		Throughout this section, $\gg$ denotes a basic classical Lie superalgebra or $\gg=\qq(n)$.  We define the notion of atypicality for a central character, and see how it is affected by the $\DS$ functor.  Furthermore, we describe the associated variety of an irreducible module in terms of its degree of atypicality. 
		
		\subsection{The Weyl group and Weyl vector } 
		Let us fix a Borel subalgebra $ {\bb}\subset\gg $ by choosing a decomposition
		$ \Delta=\Delta^{+}\cup\Delta^{-} $. Note that in general this choice is not unique but our consideration will
		not depend on it. Later we will use different Borel subalgebras in some
		proofs. Set
		\begin{equation}
			\rho=\frac{1}{2}\sum_{\alpha\in\Delta_{0}^{+}}\alpha-\frac{1}{2}\sum_{\alpha\in\Delta_{1}^{+}}\alpha.
			\notag\end{equation}
		
		Define the shifted action of $ W $ on $ \hh^{*} $ by
		\begin{equation}
			w\cdot\lambda:=w\left(\lambda+\rho\right)-\rho.
			\notag\end{equation}
		Note that for $\gg=\qq(n)$, $\rho=0$, so there is no shift in the $W$-action.

		For $\gg$ basic classical, recall that in Section~\ref{KM_description} we defined for each $\alpha\in\Delta_{1}$ a coroot $\alpha^\vee\in\ft=\hh_{0}$.  In this section, for $\gg=\qq(n)$ we will denote by $\alpha^\vee$ a non-zero element of $[(\gg_{\alpha})_{1},(\gg_{-\alpha})_1]$.  Notice that in the basic classical cases we have $w{\alpha}^{\vee}=({w\alpha})^{\vee}$ for every $w\in W$ (under the non-shifted action); for $\gg=\qq(n)$ we impose this condition on the set of $\alpha^\vee$.  We say that an iso-set $A\subset \Delta_1$ is  orthogonal to
		$\mu$ if $\mu(\alpha^{\vee})=0$ for each $\alpha\in A$, and we write $A\subset\mu^\perp$ and $\mu\in A^\perp$.  Note that this is compatible with our definition of $A^\perp$ in (\ref{A perp}).

		\subsection{Central characters}

		Recall that $ Z(\gg) $ denotes the center of the universal enveloping algebra $\mathcal U\left(\gg\right) $. For $\lambda\in\ft^*$ and chosen, fixed Borel subalgebra $\mathfrak{b}$, we denote by $ M(\lambda)$ the Verma module of highest $\lambda$ and by $L(\lambda)$ the
		unique irreducible quotient of $ M(\lambda)$.
		We say that $ \lambda\in\ft^{*} $ is {\em dominant\/}
		if $ L(\lambda)$ is finite-dimensional. 
		One can see that any $ z\in Z(\gg) $ acts as a scalar $ \chi_{\lambda}\left(z\right) $ on $ M(\lambda)$ and $ L(\lambda)$. Therefore $ \lambda\in\ft^{*} $
		defines a central character $ \chi_{\lambda}:Z(\gg) \to {\CC} $.  We emphasize that $\chi_{\lambda}$ depends also on the choice of Borel subalgebra.
		For a central character $\chi$, let
		\begin{equation}
			\ft^*_{\chi}=\left\{\mu\in\ft^{*} \mid \chi_{\mu}=\chi\right\}.
			\notag\end{equation}
		For every $\lambda\in\ft^*$ set  
		
		\begin{equation}
			\cS_{\lambda}:=\left\{{A}\in \cS \mid {A}\subset\left(\lambda+\rho\right)^{\perp}\right\}.
			\notag\end{equation}
		
		\begin{lemma} \label{lm6}\myLabel{lm6}\relax  Let $ \chi=\chi_{\lambda} $, $ {A}\in \cS_\lambda $ be maximal. Then
			\begin{equation}
				\ft^*_{\chi}=\bigcup_{w\in W} w\cdot(\lambda+\Span A).
				\notag\end{equation}
		\end{lemma}
		
		\begin{proof} This easily follows from the description of $ Z(\gg) $ formulated in
			\cite{K2} and proven in \cite{G1,Ser1}.\end{proof}
		
		We call $\lambda$ {\it regular} if there is a unique maximal $A\in \cS_\lambda$. For every $\chi$ there exists a regular $\lambda\in\ft_\chi^*$. 
		
		\subsection{Degree of atypicality}
		
		We define the degree of atypicality following \cite{KW}.  For a central character $ \chi $ set
		\begin{equation}
			\cS_{\chi}=\bigcup_{\lambda\in\ft_{\chi}^*}\cS_{\lambda}.
			\notag\end{equation}
		
		\begin{lemma} \label{lm5}\myLabel{lm5}\relax  There exists a number $ k $ such that $ \cS_{\chi}=\bigcup_{i\leq k}\cS_{i} $.
		\end{lemma}
		
		\begin{proof} It follows easily from Lemma~\ref{lm6} that $ \cS_{\chi} $ is $ W $-invariant.
			Furthermore, if $ {A}\in \cS_{\chi} $ and $ {A}' $ is obtained from $ {A} $ by multiplication of some
			roots in $ {A} $ by $ -1 $, then $ {A}'\in \cS_{\chi} $. Hence the statement follows from Lemma~\ref{lm110}
			(1) and (2), which also holds for $\qq(n)$.\end{proof}
		
		The number $ k $ is called the {\em degree of atypicality\/} of
		$ \chi $. It is clear
		that the degree of atypicality of $ \chi $ is not bigger than the defect of $ \gg $.  The degree of atypicality of a weight $ \lambda $ is by definition the degree of atypicality
		of $ \chi_{\lambda} $.  If $ k=0 $, then $ \chi $ is called {\em typical}.  We say a module is typical if it lies in $\bigoplus\limits_{\chi\text{ typical}}\Mod_\chi(\gg)$. 
		
		Let $ X_{k}=\Phi\left(\cS_{k}\right) $, $ \overline{X}_{k} $ denote the closure of $ X_{k} $. Lemma~\ref{lm110} $ \left(3\right) $ implies
		that
		\begin{equation}
			\overline{X}_{k}=\bigcup_{i=0}^{k}X_{i}.
			\notag\end{equation}

		\begin{theorem} \label{th2}\myLabel{th2}\relax  Let $ \gg $ be a basic classical Lie superalgebra or $\qq(n)$, and let $ M $ be a
			$ \gg $-module which admits central character $ \chi $, with degree of atypicality of $ \chi $
			equal to $ k $. Then we have $ X_M\subset\overline{X}_{k} $.
		\end{theorem}
		The proof of Theorem~\ref{th2} will be given in Section~\ref{proof th2}.
		
		\begin{theorem} \label{th3}\myLabel{th3}\relax  Let $ \gg $ be a basic classical Lie superalgebra.
			For any integral dominant $ \lambda\in\ft^{*} $ with degree of atypicality $ k $, $ X_{L(\lambda)}=\overline{X}_{k} $.
		\end{theorem}
		This theorem is proven in \cite{S2} for the Lie superalgebras $\mathfrak{osp}(m|2n)$ and $\mathfrak{gl}(m|n)$.
		For exceptional Lie superalgebras it is a consequence of results in \cite{M}.
		
		\begin{remark} Theorem~\ref{th3} is easy for typical $ \lambda $  since in this case $ L(\lambda)$ is projective.
		\end{remark}
		
		\begin{remark}
			Theorem~\ref{th3} fails for $\gg=\qq(n)$; indeed if we consider the irreducible $\qq(3)$-module $L=\mathfrak{psq}(3)$, we have $X_{L}=\{0\}$.  However $L$ has atypicality 1.
		\end{remark}
		
		\subsection{Proof of Theorem\protect ~\protect \ref{th2}}\label{proof th2}   We assume that $\gg$ is as in Theorem~\ref{th2}.  Recall that up to conjugacy, we may present $x\in X_k$ as $x=x_1+\dots+x_k$ where $x_i$ is a non-zero element of $(\gg_{\beta_i})_1$ for an odd root $\beta_i$.  Here  $A=\{\beta_1,\dots,\beta_k\}$ will be an iso-set.  We begin with a lemma.
		
		\begin{lemma}\label{lemchichi}
			For suitable choices of Borel subalgebras $\mathfrak{b}\subseteq\gg$ and $\mathfrak{b}_{x}\subseteq\gg_x$, for each $\lambda'\in\ft_x^*$ there exists $\lambda\in\ft^*$ such that
			\begin{itemize}
				\item
				$\lambda|_{\ft_x}=\lambda'$;
				\item $\atyp\lambda=\atyp\lambda'+k$;
				\item $[\DS_x(L_{\gg}(\lambda)):L_{\gg_x}(\lambda')]=1$.
			\end{itemize}
			In particular, $\eta^*(\chi_{\lambda'})=\chi_{\lambda}$ and thus $\atyp\chi_{\lambda'}=\atyp\eta^*(\chi_{\lambda'})-k$.
		\end{lemma}
		
		\begin{proof}
			
			We can always choose a Borel subalgebra $ \bb\subset\gg $
			such that $ \beta_{1},\dots ,\beta_{k} $ are simple roots. Note that in this case $(\rho,\beta_i)=0$ for all $i=1,\dots,k$.   Further recall from Proposition~\ref{lm202} that we may realize $\gg_x$ in $\gg$ such that $\hh_x$ will be a subalgebra of $\hh$, and $\ft_x:=(\hh_x)_0$ will lie in $\ker\beta_1\cap\cdots\cap\ker\beta_k$.  Moreover, $\gg_x$ will admit a natural Borel subalgebra $\bb_x\subseteq\bb$ containing $\hh_x$. 
			
			Let $\lambda\in\ft^*$ be a weight such that $\lambda|_{\ft_x}=\lambda'$ and $\lambda(\beta_i^\vee)=0$ for all $i$.  For $\qq(n)$ we strengthen our assumption on $\lambda$: we further require that $(\lambda,\epsilon_{i_1})=(\lambda,\epsilon_{i_2})=0$, where $\beta_i=\epsilon_{i_1}-\epsilon_{i_2}$, for all $i$. 
			
			Now to prove our statement with this choice of $\lambda$, we observe that a nonzero highest weight vector $v_{\lambda}\in L_{\gg}(\lambda)_{\lambda}$ satisfies $xv_{\lambda}=0$ and $v_{\lambda}\notin xL_{\gg}(\lambda)$.  The former statement is obvious because it is a highest weight vector.  For the latter statement, we show that $L_{\gg}(\lambda)_{\lambda-\beta_i}=0$ for all $i$, which clearly is sufficient.  In the basic classical case this follows from the fact that $(\lambda,\beta_i)=0$ and $\beta_i$ is simple.  In the $\qq(n)$ case, the statement follows from the representation theory of the $\qq(2)$-subalgebra associated to each simple root $\beta_i$, again using that the $\beta_i$ are simple.    
			
			Now to finish the proof of Lemma~\ref{lemchichi}, we observe that $\eta^*(\chi_{\lambda'})=\chi_{\lambda}$ by Proposition~\ref{maincch}, and that $\atyp \lambda'=\atyp\lambda-k$.
		\end{proof}

		Lemma~\ref{lemchichi} implies Theorem~\ref{th2} and the following. 
		
		\begin{theorem} \label{th209}\myLabel{th209}\relax  Let $M$ be a $\gg$-module that admits a central character with degree of
			atypicality $ k $, and $ x\in X_{k} $. Then $M_{x}$ is a typical module. In particular, if $M$ is semisimple over $\gg_{0}$ and $ M_{x} $ is finite dimensional, then $M_x$ semi-simple
			over $ \gg_{x} $. 
		\end{theorem}
		
		\begin{proof}  We only need to prove the last assertion. For this we use that $(\gg_x)_{0}$ is reductive, so it acts semisimply on $M_x$ if and only if its centre does.  But its centre lies in the even part of any Cartan subalgebra, whose action is induced by the action of the Cartan subalgebra of $\gg$ on $M$.  Thus the condition that $\gg_{0}$ acts semisimply on $M$, 
			along with the typicality of $M_x$, ensures the semisimplicity of $M_x$.
		\end{proof}
		
		Recall, from the notation of Section~\ref{g x in g}, the isomorphism $\hh_x\simeq\hh_A^\perp/\hh_A$, and set $\ft_x:=(\hh_x)_0$.  Then this isomorphism induces a canonical isomorphism of dual spaces $\ft_x^*\simeq A^\perp/\operatorname{span}A$. Consider the natural 
		projection $p_A: A^\perp\to \ft^*_x$. It follows immediately from Lemma~\ref{lm6} that $p_A(\lambda)=p_A(\nu)$ implies $\chi_\lambda=\chi_{\nu}$.
		Proposition~\ref{maincch} and Lemma~\ref{lemchichi} imply the following
		
		\begin{corollary}\label{corcchar} If $\lambda\in A^{\perp}$ then $\chi_\lambda=\eta^*(\chi_{p_A(\lambda)})$.
		\end{corollary}
		
		\subsection{The preimage of $\eta^*$}\label{subsec_preimage_pullback_map}
		Now we will compute the preimage $(\eta^*)^{-1}(\chi)$ for any $\chi$, showing in particular it is always finite of size one or two.  Our description will use the involutions $\sigma_x$ described in Section~\ref{involutions section}. 
		
		Define the following subgroup of the Weyl group $W$:
		\begin{equation}
			W_A=\left\{w\in W \mid w\left({A}\right)\subset{A}\cup-{A}\right\}.
			\notag\end{equation}
		It is clear that $A^\perp$ and $\hh_A$ are $W_A$-stable. Hence $W_A$ acts in the natural way on $\hh_x$ and $\hh_x^*$. By $W_{\gg_x}$ we denote the Weyl group of $\gg_x$ viewed as a subalgebra of $\gg$.
		Obviously, $W_{\gg_x}\subset W_A$.
		\begin{lemma}\label{preimage} Fix $A\in \cS_k$. Let $\chi=\chi_{\lambda}$ be a central character of atypicality degree $k$ where $\lambda\in A^\perp$ is some
			regular weight. Then
			$$|(\eta^*)^{-1}(\chi)|=\frac{|W_A\cdot p_A(\lambda)|}{|W_{\gg_x}\cdot p_A(\lambda)|}.$$
		\end{lemma}
		\begin{proof} Lemma~\ref{lm6} implies the following equality 
			$$(\hh^*_\chi)_{reg}\cap A^\perp=W_A\cdot \lambda+\Span A. $$
			The condition that $\eta^*(\chi_{p_A(\lambda)})=\eta^*(\chi_{p_A(\nu)})$ for $\lambda,\nu\in (\hh^*_\chi)_{reg}$ is equivalent to
			$p_A(\nu)\in W_{\gg_x} \cdot p_A(\lambda)$. Hence the statement.
		\end{proof}  
		\begin{lemma}\label{case}
			Consider the action homomorphism $a:W_A\to\operatorname{Aut}(\ft_x)$.  Let $k=\rank x$.  If $\sigma_x=\text{id}$ (i.e.,~$\gg=\mathfrak{gl}(m|n),\qq(n),\mathfrak{osp}(2m+1|2n),\,G(3)$, or $\mathfrak{osp}(2m|2n)$ with $k=m$)  then
			$a(W_A)=a(W_{\gg_x})$. If $\sigma_x\neq\text{Id}$ (i.e.,~$\gg=\mathfrak{osp}(2m|2n)$ with $k<m$, $D(2|1;a)$ or $F(4)$), then $[a(W_A):a(W_{\gg_x})]=2$, and we have
			\[
			a(W_A)=a(W_{\gg_x})\sqcup \sigma_xa(W_{\gg_x})
			\]
			where by abuse of notation we also write $\sigma_x$ for the involution induced on $\ft$ by $\sigma_x$.
		\end{lemma}
		\begin{proof} If  $\mathfrak{osp}(2m+1|2n)$ or $G(3)$, then $\gg_x=\mathfrak{osp}(2m+1-2k|2n-2k)$ or $\mathfrak{sl}_2$, respectively. In both cases the
			automorphism group of the root system $\Delta(\gg_x)$ coincides with the Weyl group $W_{\gg_x}$. Since $a(W_A)\subset\operatorname{Aut}\Delta(\gg_x)$, the statement
			follows. Similarly in the case $\gg=\mathfrak{osp}(2m|2n)$ and $k=m$, $\gg_x=\mathfrak{sp}(2(n-m))$ and $\operatorname{Aut}\Delta(\gg_x)=W_{\gg_x}$.
			
			For  $\gg=\mathfrak{gl}(m|n),\,\gg_x=\mathfrak{gl}(m-k|n-k)$, we get
			$$W_A=S_k\times S_{m-k}\times S_{n-k},\,\,W_{\gg_x}=S_{m-k}\times S_{n-k}$$ and $\ker a=S_k$.
			
			For $\gg=\qq(n)$, $\gg_x=\qq(n-2k)$ and we get
			\[
			W_A=S_2^{k}\times S_{n-2k}, \ \ W_{\gg_x}=S_{n-2k}
			\]
			and $\ker a=S_2^k$.
			
			If $\gg=\mathfrak{osp}(2m|2n)$, with $k<m$, $D(2|1;a)$ or $F(4)$, then $\gg_x=\mathfrak{osp}((2m-2k)|(2n-2k))$, $\mathfrak{o}(2)$ or $\mathfrak{sl}_3$.
			Then by direct computation $\operatorname{Aut}\Delta(\gg_x)/W_{\gg_x}\cong \langle\sigma_x\rangle$, where for $\gg=D(2|1;a)$ we set $\operatorname{Aut}\Delta(\gg_x)=\{\pm\text{Id}\}$.  Further, by direct computation $$a(W_A)=\operatorname{Aut}\Delta(\gg_x).$$
		\end{proof}
		\begin{theorem} \label{th210}\myLabel{th210}\relax  If $ \gg\not=\mathfrak{osp}\left(2l|2n\right) $, $F(4)$ or
			$ D(2|1;a) $, then $ \eta^{*} $ is injective.
			If $ \gg=\mathfrak{osp}\left(2l|2n\right),$ $F(4)$, or  $ D\left(2|1;a\right) $, then a preimage of $ \eta^{*} $
			has at most two elements.
		\end{theorem}
		
		\begin{proof} In the case when $\rank (x)=|A|=k$ equals the atypicality degree of $\chi$, we have $(\eta^{*})^{-1}(\chi)$
			has at most two elements by Lemma~\ref{preimage} and Lemma~\ref{case}.
			If $k$ is less than the atypicality degree of $\chi$, then consider the embedding $A\subset B$ with $|B|$ equal to the degree of atypicality. Let
			$z=x+y$ with $y=\sum_{\beta\in B\backslash A}x_\beta$. Then we have $(\gg_x)_y=\gg_z$. The composed map
			$$\eta^{*}_z:\operatorname{Hom} \left(Z\left(\gg_{z}\right),{\CC}\right) \xrightarrow{\eta_{z,x}^{*}}
			\operatorname{Hom} \left(Z(\gg_x),{\CC}\right)\xrightarrow{\eta^{*}} \operatorname{Hom} \left(Z(\gg),{\CC}\right).$$
			Since the statement holds for $\eta_z^*$ and for $\eta_{z,x}^{*}$, it holds for $\eta^*$. 
		\end{proof}

		
		\section{Superdimensions and supercharacters for basic classical Lie superalgebras}\label{app sch}

		In this section, $\gg$ denotes a basic classical Lie superalgebra.   We explore connections between the superdimension and supercharacter of a $\gg$-module $M$ and of the corresponding associated variety $ X_M $.

		\subsection{Superdimensions}
		Recall that  $\sdim M :=\dim M_0 -\dim M_1$, and that by Lemma~\ref{lem sdim}, $ \sdim M= $  $ \sdim M_x $.  So we have the following.

		\begin{lemma} \label{lm106}\myLabel{lm106}\relax  If $ X_M\not=X $, then  $\sdim M=0 $. In particular, if a
			finite-dimensional  module $ M $ admits a central character whose degree of
			atypicality is less than the defect of $ \gg $, then $\sdim M=0 $.
		\end{lemma}

		\begin{remark} In fact, Serganova proved a stronger statement, namely the Kac--Wakimoto conjecture: a simple finite-dimensional module over a basic classical Lie superalgebra has nonzero superdimension if and only if it has maximal degree of atypicality (i.e.,~equal to the defect of $\gg$). For a proof, see \cite{S2}. A version of the Kac--Wakimoto conjecture for periplectic Lie superalgebras $\pp(n)$ was proven in \cite{ES3}. Meanwhile, for $\qq(n)$ it is known that $\sdim L=0$ for all nontrivial  finite-dimensional simple modules~$L$, see  \cite{Che}. \end{remark}
		
		\subsection{Supercharacters}
		For a finite-dimensional $\gg$-module $M$ with weight decomposition $M=\oplus_{\mu\in\hh^{*}}M^{\mu}$, the supercharacter of $M$ is defined to be
		\begin{equation}\label{def sch}
			\sch M=\sum_{\mu\in\hh^{*}}\left(\sdim\ M^{\mu}\right)e^{\mu}.
		\end{equation}
		The supercharacter $\sch M$ is a $ W $-invariant analytic function on $\hh$, so  we will also write it as
		$\sch_{M}\left(h\right)$, 
		for $ h\in\hh $. Then $\sch_{M}\left(h\right)=\operatorname{str}(e^h)$, and the Taylor
		series for $ \sch_{M} $ at $ h=0 $ is 
		\begin{equation}
			\sch_{M}\left(h\right)=\sum_{i=0}^{\infty}p_{i}\left(h\right),
			\notag\end{equation}
		where $ p_{i}\left(h\right) $ is a homogeneous polynomial of degree $ i $ on $ \hh $. The {\it order} of $ \sch_{M} $
		at zero is by definition the minimal $ i $ such that $ p_{i}\not\equiv 0 $.
		
		\begin{theorem} \label{th9}\myLabel{th9}\relax  Assume that all odd roots of $ \gg $ are isotropic. Let $ M $ be a
			finite-dimensional $ \gg $-module, $ s $ be the
			codimension of $ X_M $ in $ X $. The order of $ \sch_{M} $ at zero is not less than $ s $.
			Moreover, the polynomial $ p_{s}\left(h\right) $ in Taylor series for $ \sch_{M} $ is determined
			uniquely up to proportionality.
		\end{theorem}
		
		\begin{proof} The proof is based on the following lemma, the proof of this
			lemma is similar to the proof of Lemma~\ref{lem sdim} (6). We leave it to the reader.
			
			\begin{lemma} \label{lm153}\myLabel{lm153}\relax  Let $ x\in X $, $ h\in\hh_{0} $ and $ \left[h,x\right]=0 $. Then $ \operatorname{Ker} x $ and $ xM $ are
				$ h $-invariant and str$ _{M}e^h= $str$ _{M_{x}}e^h$.
			\end{lemma}
			
			Now we proceed to the proof of the Theorem~\ref{th9}. If $ X_M$ contains an irreducible component of $X$, the statement is trivial since $s=0$. Otherwise
			there exists $ k $ smaller than the defect of $ \gg $ such that
			\begin{equation}
				X_M\subset\cup_{{A}\in \cS\text{, }|{A}|\leq k}\Phi\left({A}\right).
				\notag\end{equation}
			Let $ {A}=\left\{\beta_{1},\dots ,\beta_{k+1}\right\}\in \cS $, $ x=x_{\beta_1}+\dots +x_{\beta_{k+1}} $ for some nonzero $ x_{\beta_i}\in\gg_{\beta_{i}} $. Then $ M_{x}=\left\{0\right\} $.
			If $ h\in\hh $ satisfies $ \beta_{1}\left(h\right)=\dots =\beta_{k+1}\left(h\right)=0 $, then $ \left[h,x\right]=0 $. Hence by Lemma~\ref{lm153}
			str$ _{M}e^h= $str$ _{M_{x}}e^h=0 $. This we have proved the following property
			\begin{equation}
				\sch_{M}\left(\hh_{{A}}^{\perp}\right)=0\text{ for all }{A}\in \cS,\ |{A}|=k+1.
				\label{equ29}\end{equation}\myLabel{equ29,}\relax 
			Let $ p_{i} $ be the
			first nonzero polynomial in the Taylor series for $ \sch_{M} $ at zero. Then $ p_{i} $ also
			satisfies~\eqref{equ29}. Let $ {\text B}=\left\{\alpha_{1},\dots ,\alpha_{k}\right\}\in \cS $ and $ \bar{p}_{i} $ be
			the restriction of $ p_{i} $ to $ \hh_{{\text B}}^{\perp} $. If $ \bar{p}_{i}\not=0 $, then degree of $ \bar{p}_{i} $ is $ i $.
			Since $ p_{i}\left(\hh_{{\text B}\cup\alpha}^{\perp}\right)=0 $ for any $ \alpha\not=\pm\alpha_{i} $, $ \alpha\in{\text B}^{\perp} $, then $ \alpha $ divides $ \bar{p}_{i} $.
			That gives the estimate on $ i $. Indeed, $ i $ is not less than the number of all
			possible $ \alpha $, i.e.,~$ \frac{|\Delta_{1}\cap{\text B}^{\perp}|}{2}-|{\text B}| $. By Corollary~\ref{cor72}
			the latter number is the codimension $ s $ of $ X_M $ in $ X $. Hence $ i\geq s $.
			
			To prove the second assertion we need to show
			that if two homogeneous $ W $-invariant polynomials $ p $ and $ q $ of degree $ s $ satisfy
			\eqref{equ29}, then $ p=cq $ for some $ c\in{\CC} $. After restriction to $ \hh_{{\text B}}^{\perp} $
			\begin{equation}
				\bar{p}=a \Pi_{\alpha\in\left(\Delta_1^{+}\cap{\text B}^{\perp}\right)\backslash\pm{\text B}} \alpha\text{, }\bar{q}=b \Pi_{\alpha\in\left(\Delta_1^{+}\cap{\text B}^{\perp}\right)\backslash\pm{\text B}} \alpha
				\notag\end{equation}
			for some constants $ a $ and $ b $. Therefore there exists $ f=p-cq $ such that
			$ f\left(\hh_{{\text B}}^{\perp}\right)=0 $. Thus, $ f $ satisfies~\eqref{equ29} for $ k $ instead of $ k+1 $. Then $ \operatorname{deg}f>s $, which
			implies $ f=0 $.\end{proof}

		

		\section{Reduced Grothendieck rings and $\ds_x$} \label{sec ds groth}
		We retain the notation of Section~\ref{sec:definitions}. 
		In this section we discuss  the homomorphism $\ds_x: \Gr_-(\gg)\rightarrow \Gr_-(\gg_x)$, for classical Lie superalgebras  $\gg$,
		where $\Gr_-(\gg)$ and $\Gr_-(\gg_x)$ stand for  the reduced Grothendieck rings of $\cF(\gg)$ and $\cF(\gg_x)$, respectively.  The study of $\ds_x$ was initiated by Hoyt and Reif in \cite{HR} for the basic classical Lie superalgebras.  We consider
		$\gg_x$ as a subalgebra of $\gg$ using the 
		splitting $\gg^x=\gg_x\ltimes [x,\gg]$ as in Propositions~\ref{lm202} and \ref{embedding p}.
		
		\subsection{Subcategories of $\cF(\gg)$ and the $DS$ functor}\label{subcats and DS}
		Before discussing $ds_x$, we describe certain subcategories of $\cF(\fg)$ and their relation to the $DS$ functor.  Let $\fg$ be one of basic classical superalgebras and $\Lambda_{\cF(\gg)}$ denote the abelian subgroup of $\ft^*$ consisting of weights of $M\in {\cF(\fg)}$. For any subset $\Lambda\subset\Lambda_{\cF(\fg)}$  we
		denote by $\cF^{\Lambda}(\fg)$ the full subcategory of $\cF(\fg)$ consisting of modules with weights in $\Lambda$.
		
		Let $G$ be an algebraic supergroup with Lie superalgebra  $\fg$. Then $G$ is determined by the lattice $\Lambda_G\subset \Lambda_{\cF(\gg)}$ and
		the category $\cF(G)$  of finite-dimensional representations of $G$ is equivalent to $\cF^{\Lambda_G}(\gg)$.
		
		If $x\in\gg_1$ is a self-commuting element and $G^x$ is the centralizer of $x$ in $G$ then the Lie algebra of $G^x$ is the kernel of $ad_x$.
		We denote by $G_x$ the quotient of $G^x$ such that $\operatorname{Lie}G_x=\gg_x$.
		It is clear that $\DS_x$ induces the functor $\cF(G)\to\cF(G_x)$.
		
		We denote by $G$ the particular supergroup for every basic classical superalgebra:
		\begin{enumerate}
			\item If $\gg=\mathfrak{gl}(m|n)$ we set $G:=GL(m|n)$, $\Lambda_G:=\displaystyle\sum_{i=1}^m\mathbb Z\varepsilon_i+\displaystyle\sum_{j=1}^n\mathbb Z\delta_j$.
			\item If $\gg=\mathfrak{sl}(m|n)$ we set $G:=SL(m|n)$, $\Lambda_G:=\Lambda_{GL(m|n)}\cap \fh^*$.
			\item If $\gg=\mathfrak{osp}(2m|2n)$ or $\mathfrak{osp}(2m+1|2n)$  we set $G:=SOSP(2m|2n)$ (respectively, $SOSP(2m+1|2n)$),
			$\Lambda_G:=\sum_{i=1}^m\mathbb Z\varepsilon_i+\sum_{j=1}^j\mathbb Z\delta_j$.
			\item If $\gg=\fp(n)$ we set $G:=P(n)$, $\Lambda_G:=\sum_{i=1}^n\mathbb Z\varepsilon_i$.
			\item If $\gg=\fq(n)$ we set $G:=Q(n)$, $\Lambda_G:=\sum_{i=1}^n\mathbb Z\varepsilon_i$.
			\item If $\gg$ is an exceptional Lie algebra of type $G(3)$ and $F(4)$ then $G$ is the adjoint group with $\Lambda_G$ being the root lattice. One has
			$\Lambda_{\cF(\gg)}=\Lambda_G$ for $G(3)$ and $\Lambda_{\cF(\gg)}/\Lambda_G=\mathbb{Z}_2$ for $F(4)$, see~\cite{M}.
			\item For $\gg=D(2,1,a)$ we consider the algebraic group $G$ with $\Lambda_{G}:=\Lambda_{\cF(\gg)}$.
		\end{enumerate}
		Next we set $\Lambda_{\gg}:=\Lambda_G$ in all cases when $\gg=[\gg,\gg]$. In the remaining cases we set
		$$\begin{array}{l}\Lambda_{\mathfrak{gl}(m|n)}:=\Lambda_{GL(m|n)}+\mathbb{C}\str\ \ \text{ where }\str:=\vareps_1+\dots+\vareps_m+\delta_1+\dots+\delta_n,\\
			\Lambda_{\mathfrak{p}(m)}:=\Lambda_{P(m)}+\mathbb{C}\str,\ \ \text{ where }\str:=
			\vareps_1+\dots+\vareps_m,\\
			\Lambda_{\fq(m)}:=\Lambda_{Q(m)}+\mathbb{Z}\frac{\str}{2}, \ \ \text{ where }\str:=
			\vareps_1+\dots+\vareps_m.
		\end{array}
		$$

		\subsubsection{}
		Consider the case when $\gg$ is non-exceptional. Take $x\not=0$ as in Prop. 4.5.
		Let $\Lambda'_\gg=\Lambda_{\cF(\gg)}\setminus \Lambda_\gg$. Then we have a decomposition
		\[
		\cF(\gg)=\cF^{\Lambda_g}(\gg)\oplus \cF^{\Lambda'_{\gg}}(\gg).
		\]
		Every module $M\in \cF^{\Lambda'_{\gg}}(\gg)$ is projective and hence $\DS_x(M)=0$. Furthermore, $\DS_x$ induces the functor
		$\cF^{\Lambda_{G}}(\gg)\to \cF^{\Lambda_{G_x}}(\gg_x)$.
		
		In the cases when $[\gg,\gg]\neq \gg$ we can be more precise. Namely, when $\gg_x\neq0$, $\DS_x$ restricts to the functor:
		\[
		\cF^{\Lambda_{G}+c \str}(\gg)\to \cF^{\Lambda_{G_x}+c \str}(\gg_x),
		\]
		where $c\in\mathbb C$ for $\mathfrak{gl}(m|n)$ and $\fp(n)$, $c=0,\frac{1}{2}$ for $\fq(m)$.
		To see this for $\mathfrak{gl}(m|n)$ and $\fp(n)$ we just note that every $M\in \cF^{\Lambda_{G}+c \str}$ can be obtained from  $M_0\in \cF(G)$ by tensoring with
		one dimensional character in $\chi_c\in(\gg/[\gg,\gg])^*$ and therefore
		it is suffices to compute $\DS_x$ in the case $c=0$ and then use
		$\DS_x(M\otimes \chi_c)=\DS_x(M)\otimes \chi_c$.
		The case of $\fq(n)$ is straightforward.

		\subsubsection{Exceptional algebras}\label{remG3}
		Take $x\not=0$. All such $x$ are conjugate by the adjoint action of $G_0$.
		
		For $D(2,1,a)$, $G_x=\mathbb C^*$. Therefore $\DS_x:\cF (D(2,1,a)\to \cF (\mathbb C^*)$.
		By [Germoni] for
		$\gg=D(2,1,a)$ with
		$a\not\in\mathbb{Q}$ all atypical modules 
		in $\cF(\gg)$ have zero central character. Using the filtration of projective modules obtained by Germoni, one can show (see \cite{G4}) that $\DS_x(L)$
		is a trivial $\mathbb C^*$-module
		for any simple atypical $\gg$-module $L$ and therefore for any $\gg$-module. In other words the image of $\DS_x$ lies in the category of
		vector superspaces equipped with the trivial action of $\mathbb C^*$.

		Combining the description of dominant weights  (see~\cite{M})
		and~(\ref{Omega}) we obtain the following results for $G(3)$ (with $\gg_x=\fsl_2$) and $F(4)$   (with $\gg_x=\fsl_3$) :
		\[
		\DS_x(\cF(G(3)))\subset \cF(PSL(2))\ \ \ \ \DS_x(\cF(F(4)))\subset \cF(PSL(3))
		\]
		where  $\cF(G)$ denotes the category of finite-dimensional representations of the algebraic group $G$. In fact, from \cite{G3} and \cite{M} we obtain that
		\[
		\DS_x(\cF(F(4)))\subset \cF(G''_x)
		\]
		where $G''_x$ is a non-connected algebraic group fitting into a non-split exact sequence
		\begin{equation}\label{F4}
			1\to PSL(3)\to G''_x\to \mathbb Z_2\to 1,\end{equation}
		compare to Remark~\ref{inv_rem}.

		\subsection{Properties of $\ds_x$}
		We now begin the discussion of $ds_x$ with general remarks which are valid for classical Lie superalgebras $\gg$ satisfying $\fh=\ft$, i.e., classical $\gg\neq\qq(n)$. (The case of general $\gg$ is considered in~\cite{GSS}).  As already noted, by Section~\ref{g x in g}, we have splittings
		$\gg^x=\gg_x\ltimes [x,\gg]$.  Further we have $\ft=\ft_x\oplus\ft'$, where $\ft_x\subseteq\gg_x$ is a Cartan subalgebra.

		In these cases the ring $\Gr_-(\gg)$ is spanned by the images of the simple finite-dimensional modules. Since these modules are highest weight modules, the map
		$[N]\mapsto [\Res^{\gg}_{\fh} N]$ gives an embedding of
		$\Gr_-(\gg)\to \Gr_-(\fh)$ and we identify $\Gr_-(\gg)$ with this ring.
		The image  is called the 
		{\em ring of supercharacters}, since $[\Res^{\gg}_{\fh} N]=\sch N$.
		By Lemma~\ref{ds restriction}
		$$\sdim \DS_x(N)_{\mu'}=\sum_{\mu\in\ft^*: \mu|_{\ft_x}=\mu'} \sdim N_{\mu}.$$
		for each $N\in\cF(\gg)$ and $\mu'\in \ft_x^*$.
		Thus $\ds_x$  written for the supercharacter rings takes the form
		\begin{equation}\label{xaxa}
			\ds_x\bigl(\sum_{\nu\in\ft^*} m_{\nu}e^{\nu}\bigr)=\sum_{\nu\in\ft^*} m_{\nu}e^{\nu|_{\ft_x}}.\end{equation}
		coincides with the restriction of the map $f\mapsto f|_{\ft_x}$.
		Using the representation theory of $\fsl(1|1)$ it is easy to see that 
		if $x$ can be embedded into an $\fsl(1|1)$-triple $x,y,\alpha^{\vee}$
		with $\alpha^{\vee}\in\ft$, then for 
		\begin{equation}\label{xaxaxa}
			\ds_x\bigl(\sum_{\nu\in\ft^*} m_{\nu}e^{\nu}\bigr)=\sum_{\nu\in\ft^*: \nu(\alpha^{\vee})=0} m_{\nu}e^{\nu|_{\ft_x}}.\end{equation}
		These formulas appeared in~\cite{HR};
		for $\fq(n)$ a similar formula is given in~\cite{GSS}.
		
		\subsubsection{}
		Now consider the case with $\gg$ from the list (\ref{rem: classical LSA}).
		By  Lemma~\ref{ds restriction} one has $\ds_x=\ds_y$ if $\gg_x=\gg_y$ as embedded subalgebras.
		For $x$ of rank $r$ 
		we denote the map $\ds_x: \Gr_-(\gg)\to \Gr_-(\gg_x)$ by $\ds^r$.
		By Lemma~\ref{ds composition} one has 
		$\ds^i=\ds^1\circ\ds^1\circ\ldots\circ \ds^1$
		if $\DS^1(\DS^1\ldots(\gg))=\DS^i(\gg)$. 
		By Lemma~\ref{lem:group ker} we have
		$\ker \ds_x=\ker\ds_y$ if $x,y$ are conjugated by an inner automorphism.
		Note that $\rank x=\rank y$ implies the existence of $x'\in X$ 
		such that $\gg_{x'}=\gg_y$ and $x,x'$ are conjugated by an inner automorphism.
		Hence the ideal $\ker\ds^r\subset \Gr_-(\gg)$ does not depend on the choice of $x$ of rank $r$.

		\subsection{The ring $\Gr_-(\gg)$}
		For $\gg\not=\fp(n),\fq(n)$,
		Sergeev and Veselov
		interpreted the supercharacter ring as a ring of functions admitting certain supersymmetry conditions, see~\cite{SerV}.
		For example, the supercharacter ring  for the category $\cF^\ZZ(\mathfrak{gl}(m|n))$ can be realized as
		$$
		\JJ\cong\left\{ f\in\ZZ\left[x_{1}^{\pm1},\ldots,x_{m}^{\pm1},y_{1}^{\pm1},\ldots,y_{n}^{\pm1}\right]^{S_{m}\times S_{n}}\,\middle\vert\, f\vert_{x_1=y_1=t} \text{ is independent of } t \right\},\label{eq:GL K_G}
		$$
		and in this case, if $\rank x=1$, then  $\ds_{x}\left(f\right)=f\vert_{x_1=y_1}$ (see~\cite{HR}).
		The supercharacter ring for $\fp(n)$ was described in~\cite{IRS} using an inductive argument with the help of $ds^2$. Using a similar method Reif described the 
		ring $\Gr_+(\fq_n)$  in~\cite{R}. Note that 
		for $\fq(n)$ the supercharacter of a finite-dimensional nontrivial simple $\fq(n)$-module is always zero \cite{Che}.

		\subsection{The image of $\ds_x$}
		Let $\gg$ be one of the superalgebras from the list (\ref{list1}). The categories
		$\cF^{\Lambda_G}(\gg)\cong\cF(G)$ were introduced in Section~\ref{subcats and DS}. 
		Let $\sigma_x\in \Aut(\gg_x)$ be the involution introduced in Section~\ref{involutions section}. Note that $\sigma_x=\Id$ except
		for the cases $D(2,1,a), F(4)$ and $\osp(2m|2n)$ with $\rank x<m$.
		We also denote by $\sigma_x$ the induced involution of the ring $\Gr_-(\gg_x)$.
		
		\begin{theorem}\label{thm2}
			Take $x\not=0$. 
			For  non-exceptional $\gg$  from the list (\ref{list1}),
			one has 
			$$\ds_x\bigl(\Gr_-(\cF(G)) \bigr)=\Gr_-(\cF(G_x))^{\sigma_x}.$$
			For  $\gg=D(2,1,a)$ with $a\in\mathbb{Q}$,
			$G(3)$ and $F(4)$ one has
			$$\ds_x\bigl(\Gr_-(\cF(G)) \bigr)=\Gr_-(\cF(G'_x))^{\sigma_x}$$
			where $G'_x=\mathbb{C}^*,PSL(2), PSL(3)$ for respectively. For
			$\gg=D(2,1,a)$ with $a\not\in\mathbb{Q}$ one has $\ds_x\bigl(\Gr_-(\cF(G)) \bigr)=
			\mathbb{Z}$.
			
			Recall that
			$\cF(\gg)$ is equivalent to $\cF(G)$ for $\gg=G(3)$ and $D(2,1,a)$.
			We have
			$$\ds_x(\Gr_-(\gg))=\left\{\begin{array}{ll}
				\Gr_-(\gg_x)^{\sigma_x} & \text{ for }  \mathfrak{gl}(m|n),\osp(m|n), \fp(n), F(4)\\
				\Gr_-(\cF(G_x))^{\sigma_x} & \text{ for } \fq(n).
			\end{array}\right.$$
		\end{theorem}
		Note that for $\gg=F(4)$ the categories $\cF(G'_x)^{\sigma_x}$ and $\cF(G''_x)$ are equivalent, see (\ref{F4}). 
		For $\gg\not=\fq(n),\fp(n)$ the statement was established in~\cite{HR}
		and for $\fp(n)$ with
		$x\in\gg^{-1}$ the assertion was proven in~\cite{IRS};  the proofs are
		based on the evaluation~(\ref{xaxaxa}) (in these case
		$x$ can be embedded in an $\fsl(1|1)$-triple);
		for $\fq(n)$ the assertion is proven in \cite{GSS}.
		For $\fp(n)$ with $x$ of rank $1$ we prove the assertion in Corollary~\ref{p case thm 8.1};
		since $\ds^r=\ds^1\circ\ldots\circ\ds^1$, this implies the assertion
		for general $x$.
		
		\begin{remark}
			For a precise description of $ds_x(\Gr_-(\cF(\gg)))$ when $\gg=D(2,1;a)$ with $a\in\mathbb{Q}$, see Section~\ref{D21a}.
		\end{remark}
		
		\begin{remark}
			For $\qq(n)$-case $\ds_x(\Gr_-(\cF(\gg)))=\Gr_-(\cF(G))^{\sigma_x}$, but it can be  easily seen  that if $M\in\cF(\gg)$ then $\DS_xM$ need not have the structure of a $G$-module.
		\end{remark}
		
		\subsection{The kernel of $\ds_x$}
		\subsubsection{Notation}
		Take $\gg\not=\fp(n),\fq(n)$.
		For a fixed choice of negative roots $\Delta^{-}=\Delta_{0}^{-}\sqcup\Delta_{1}^{-}$,
		we denote the super Weyl denominator by $R=\frac{R_{0}}{R_{1}}$,
		where $R_{0}=\prod_{\alpha\in\Delta_{0}^{-}}\left(1-e^{\alpha}\right)$
		and $R_{1}=\prod_{\alpha\in\Delta_{1}^{-}}\left(1-e^{\alpha}\right)$.
		For $\lambda\in\ft^*$ set
		\[
		k\left(\lambda\right):=R^{-1}\cdot\sum_{w\in W}\left(-1\right)^{l\left(w\right)+p\left(w\left(\rho\right)-\rho\right)}
		e^{w\left(\lambda+\rho\right)-\rho},
		\]
		where $l\left(w\right)$
		denotes the length of $w$ as a product of simple reflections with
		respect to a set of simple roots for $\gg_{0}$. Let
		$P^+(\gg_0)$ be the set of dominant weights of $\gg_0$.  For $\gg=\mathfrak{gl}(m|n),\osp(2|2n)$ with the distinguished choice of simple
		roots, $k\left(\lambda\right)$ is the supercharacter of a Kac module 
		$K\left(\lambda\right)={\operatorname{Ind}_{\gg_{0}\oplus\gg^{1}}^{\gg}
			L_{\gg_0}\left(\lambda\right)}$ when $\lambda\in P^{+}(\gg_0)$ , and we have:
		\[
		k(\lambda)=\sch L_{\gg_0}(\lambda)\prod_{\alpha\in\Delta_{1}^{-}}\left(1-e^{\alpha}\right).
		\]
		However in Type II, $k\left(\lambda\right)$
		is a virtual supercharacter.
		
		For $\fp(n)$ we set
		$K(\lambda)={\operatorname{Ind}_{\gg_{0}\oplus\gg^{1}}^{\gg}
			L_{\gg_0}\left(\lambda\right)}$; then $K(\lambda)$ is a ``thin'' Kac module. For $\lambda\in P^+(\gg_0)$ the 
		expression
		$k(\lambda):=\sch K(\lambda)$ is given by the above formulas.
		Finally we need one more virtual supercharacter for $\pp(n)$ given by
		\[
		k'(\lambda):=k(\lambda)\sch S^{\bullet}\Pi(\mathbb{C}^n)^*
		\]
		Here $\mathbb{C}^n$ denotes the standard representation of $\gg\ll(n)$.  In coordinates $\epsilon_1,\dots,\epsilon_n$, this is given explicitly by
		\[
		k'(\lambda)=\sch L_{\gg_{0}}(\lambda)\prod\limits_{i=1}^n(1-e^{-\epsilon_i})\prod\limits_{\alpha\in\Delta_{-1}}(1-e^{\alpha}).
		\]
		
		Let $\rho_{iso}:=\frac{1}{2}\sum_{\alpha\in\Delta_{iso}^{+}}\alpha$.

		\begin{theorem}\label{thm3}
			Let $\ker_1\subset\ker_2\subset\ldots$ be the kernels for $\ds^r:\Gr_-(\gg)\to \Gr_-(\gg_x)$.
			\begin{enumerate}
				\item For $\gg\not=\fp(n),\fq(n),\mathfrak{gl}(1|1)$ the set
				$\{k(\lambda)|\ \lambda\in P^+(\gg_0)+\rho_{iso}\}$ forms a basis of $\ker_1$.
				\item For $\fp(n)$ with $n>1$ the set $\{k'(\lambda)|\lambda\in P^+(\gg_0)\}$ forms a basis of
				$\ker_1$.
				\item  For $\fp(n)$ with $n>2$ the supercharacters of the thin Kac modules form a basis of $\ker_2$. 
			\end{enumerate}
		\end{theorem}
		
		Hoyt and Reif proved (1) in~\cite{HR}, and (3) was proven in~\cite{IRS}. We will give a proof of (2) in Section~\ref{proofii}
		below.

		\begin{remark}\begin{itemize}
				\item The ring $\Gr_-(\mathfrak{p}(1))$ is
				the group ring of $\mathbb{C}$ and $\ds^1$ acts as $\ds^1(e^c)=1$.
				\item The ring $\Gr_-(\mathfrak{gl}(1|1))$ lies in the group ring of
				$\mathbb{C}\times\mathbb{C}$: this ring is spanned by the elements 
				$e^{(0,a)}$ and $e^{(b,a)}-e^{(b,a-1)}$ for $a,b\in\mathbb{C}$ with $b\not=0$.
				The map  $\ds^1$ is the restriction of the algebra map given by 
				$\ds^1(e^c)=1$ for any $c\in \mathbb{C}\times\mathbb{C}$.
				\item The kernel of $ds^1$ for $\gg=\qq(n)$ is computed in \cite{GSS}.  The result is in terms of an explicit basis of $\Gr_-(Q(n))$.
			\end{itemize}
		\end{remark}

		\begin{proof}[Proof of (ii)]\label{proofii}
			Retain the notation of Section~\ref{subcats and DS}.  Clearly, 
			$K_-(\lambda+c\str)=K_-(\lambda)\otimes \chi_c$, where $\chi_c$ is a one-dimensional
			$\fp_n$-module corresponding to $c\str$. Using Section~\ref{subcats and DS}, we can reduce the statement to
			$\cF^{\Lambda_G}(\gg)$.
			
			We utilize methods of~\cite{HR}.  Let $h_1,\ldots,h_n$ be the standard basis of $\ft$ (which is dual to
			$\vareps_1,\ldots,\vareps_n\in\ft^*$). 
			Take $x\in  \gg_{2\vareps_{n}}$
			and identify $\gg_x$ with the ``natural copy'' of $\fp(n-1)$ in $\fp(n)$;
			in this case $\ft_x$ is spanned by $h_1,\ldots,h_{n-1}$.  
			Take $f\in \Gr_-(P(n))$ such that 
			$\ds^1(f)=0$. 
			
			Identify $\ft_x^*$  with the span of
			$\vareps_1,\ldots,\vareps_{n-1}$ and write
			\[
			f=\sum m_{\nu} e^{\nu}=\sum_{\mu'\in \ft_x^*} f_{\mu'},\ \ \text{ where }
			f_{\mu'}:=e^{\mu'} \sum_{a\in\mathbb{C}} m_{\mu'+a\vareps_n} e^{a\vareps_n}.
			\]
			By~(\ref{xaxa}) $\ds^1(f)=0$ is equivalent to
			$\sum_{a\in\mathbb{C}} m_{\mu'+a\vareps_n}=0$ for each $\mu'\in\ft_x^*$.
			
			Therefore  $\sum_{a\in\mathbb{C}} m_{\mu'+a\vareps_n}=0$
			means that $f_{\mu'}$ is divisible by $1-e^{-\vareps_n}$. Hence
			$f$ is divisible by $1-e^{-\vareps_n}$.
			
			Recall that $\ker \ds^1\subset\ker \ds^2$, so  $\ds^2(f)=0$.
			Using the above argument for $x\in \gg_{\vareps_{n-1}+\vareps_n}$.
			we obtain that $f$ is divisible by $1-e^{-\vareps_{n-1}-\vareps_n}$.
			
			The restriction $\Res^{\gg}_{\gg_0}$ gives an embedding of the
			supercharacter ring of $\fp(n)$ to the
			supercharacter ring of $\fp(n)_0=\mathfrak{gl}_n$. In particular, 
			$f$ is $W$-invariant and thus $f$ is divisible by the element 
			\[
			R':=\prod\limits_{i=1}^n(1-e^{-\epsilon_i})\prod\limits_{\alpha\in\Delta_{-1}}(1-e^{\alpha})
			\]
			Since
			$R'$ is $W$-invariant one has
			$f=R' f'$ where $f'$ is a $W$-invariant element in 
			$\mathbb{Z}[e^{\nu},\nu\in\ft^*]$. 
			The ring $\mathbb{Z}[e^{\nu},\nu\in\ft^*]^W$ is the character ring
			of $\mathfrak{gl}_n$, so $f'$ can be written as
			$f'=\sum_j m_j \sch L_{\mathfrak{gl}_n}(\nu_i)$. This gives
			$f=\sum_j m_j  R'\sch L_{\mathfrak{gl}_n}(\nu_i)=\sum_j m_j k(\nu_i)$ as required.
			
			Finally, one can use a standard highest weight argument to show that the $k(\lambda)$ are linearly independent.
		\end{proof}
		
		\begin{remark} Take $n>0$.
			One has $\Lambda_G/\mathbb{Z}\Delta\cong\mathbb{Z}_2$. Writing
			$\Lambda_G=\mathbb{Z}\Delta+ (\vareps_1+\mathbb{Z}\Delta)$
			we have
			\[
			\cF(P(n))=\cF_0(\gg)\oplus \cF_1(\gg),\ \text{ where }
			\cF_0(\gg):=\cF^{\mathbb{Z}\Delta}(\gg),\ \ \cF_1(\gg):=\cF^{\mathbb{Z}\Delta+\vareps_1}(\gg).
			\]
			Thus we have 
			\[
			\Gr_-(\cF(P(n))=\Gr_-(\cF_0(\gg))\oplus\Gr_-(\cF_1(\gg))
			\]
			If $f\in \Gr_-(\cF(P(n))=\Gr_-(\cF_0(\gg))$  has $ds^1(f)=0$, then the above argument will imply that $f$ is divisible by $(1-e^{2\epsilon_n})$ and $(1-e^{\epsilon_{n-1}+\epsilon_n})$.  Applying $W$-invariance of $f$, we learn that it is the subspace of $\Gr_-(\cF(P(n))$ spanned by the supercharacters  $k_+(\lambda)$, i.e., the supercharacters of thick Kac modules $\Ind_{\gg\oplus\gg^{-1}}^{\gg}L_{\gg_0}(\lambda)$.
			
			However thick Kac modules do not span the kernel of $ds^1$; for instance when $n=2$, we have 
			\[
			\ds^1([L(\epsilon_1)]-[\mathbb{C}_{\str}]+[\mathbb{C}_{-\str}])=0.
			\]
			However one can show (using evaluation arguments) that $[L(\epsilon_1)]-[\mathbb{C}_{\str}]+[\mathbb{C}_{-\str}]$ is not in the span of supercharacters of thick Kac modules.
			
		\end{remark}



		\section{The $\DS$ functor and $\mathfrak{sl}(\infty)$-modules}\label{sec DS sl}
		
		In this section, we discuss a connection between the $\DS$ functor  and  $\mathfrak{sl}(\infty)$-modules arising from  $\mathfrak{gl}(m|n)$-representation theory, which was discovered and studied by Hoyt,  Penkov and Serganova in \cite{HPS}. We will recall some basic facts for $\mathfrak{sl}(\infty)$, and refer the reader to the book ``Classical Lie algebras at infinity''   by Penkov and Hoyt for an in-depth treatment of the Lie algebra $\mathfrak{sl}(\infty)$ and other locally finite Lie algebras \cite{PH}.

		In the pioneering paper  \cite{B}, Brundan showed that the complexification of the Grothendieck group for the categories $\cF (GL(m|n))$ and the integral BGG category $\mathcal{O}^\ZZ_{m|n}$ inherit a natural $\mathfrak{sl}(\infty)$-module structure from the action of translation functors $E_i,F_i$. This action and general categorification methods were used by Brundan, Losev and Webster in \cite{BLW} to develop Kazhdan-Lusztig theory for $\mathfrak{gl}(m|n)$.
		
		Now since the $\DS$ functor commutes with translation functors,  the induced homomorphism $\ds$  of  reduced Grothendieck groups is, in fact, a homomorphism of $\mathfrak{sl}(\infty)$-modules  \cite{HPS}. This homomorphism $\ds$ was used in  \cite{HPS} to help obtain a description of the $\mathfrak{sl}(\infty)$-module structure of the reduced Grothendieck groups for both of the categories $\cF^{\ZZ}_{m|n}$ and $\mathcal{O}^\ZZ_{m|n}$ of integral $\mathfrak{gl}(m|n)$-modules.

		\subsection{The Lie algebra $\mathfrak{sl}(\infty)$}
		
		The Lie algebra $\mathfrak{gl}\left(\infty\right)$ can be defined by taking countable-dimensional vector spaces $\bV,\bV_{*}$  with bases  $\{ v_{i}\} _{i\in\ZZ},\{ v_{j}^{*}\} _{j\in\ZZ}$, and letting $\mathfrak{gl}\left(\infty\right)=\bV\otimes \bV_{*}$ with bracket  (extended linearly) given by
		$$
		[v_{i}\otimes v_{j}^*,v_{k}\otimes v_{l}^*]=\langle v_{k},v_{j}^*\rangle v_{i}\otimes v_{l}^*-\langle v_{i}, v_{l}^*\rangle v_{k}\otimes v_{j}^*,
		$$
		where  $\langle\cdot,\cdot\rangle:\bV\otimes \bV_{*}\rightarrow\CC$ is the nondegenerate pairing defined by
		$\langle v_i,v_j^*\rangle=\delta_{ij}$.

		We can identify $\mathfrak{gl}(\infty)$ with the
		space of infinite matrices $\left(a_{ij}\right)_{i,j\in\ZZ}$
		which have only finitely many nonzero entries, using the correspondence $v_{i}\otimes v_{j}^{*}\mapsto E_{ij}$, where $E_{ij}$ is the matrix with $1$ in the $i,j$-position and zeros elsewhere.
		Under this identification, $\langle\cdot,\cdot\rangle$ is the trace map on $\mathfrak{gl}(\infty)$, and the kernel of $\langle\cdot,\cdot\rangle$ is the Lie algebra $\mathfrak{sl}\left(\infty\right)$.
		The center of $\mathfrak{gl}(\infty)$ is trivial, and the following exact sequence does not split:
		$$
		0\to\mathfrak{sl}(\infty)\to\mathfrak{gl}(\infty)\to\CC\to 0.
		$$
		The Lie algebra
		$\mathfrak{sl}\left(\infty\right)$ is generated by the elements $e_i: = E_{i, i+1}$, $f_i:=E_{i+1,i}$ for $i\in\ZZ$. 
		
		We can realize $\mathfrak{sl}(\infty)$ as a direct limit of finite-dimensional Lie algebras $\underrightarrow{\lim}\ \mathfrak{sl}\left(n\right)$, that is,
		$\mathfrak{sl}(\infty)$ is  isomorphic to a union $\bigcup_{n\in\ZZ_{\geq 2}} \mathfrak{sl}(n)$, of nested Lie algebras
		$$
		\mathfrak{sl}(2)\subset\mathfrak{sl}(3)\subset\cdots\subset\mathfrak{sl}(n)\subset\mathfrak{sl}(n+1)\subset\cdots.
		$$
		The Lie algebra obtained from this union is independent, up to isomorphism, of the choice of the  inclusions $\mathfrak{sl}(n)\hookrightarrow\mathfrak{sl}(n+1)$.

		\subsection{Modules over $\mathfrak{sl}(\infty)$ }
		
		The modules  $\bV,\bV_{*}$ are the defining representations of $\mathfrak{sl}(\infty)$, and for $p,q>0$, the tensor modules $\bV^{\otimes p}\otimes \bV_{*}^{\otimes q}$, $p,q\in\ZZ_{\geq 0}$ are not semisimple. 
		Schur-Weyl duality for $\mathfrak{sl}(\infty)$ implies that the module $\bV^{\otimes p}\otimes \bV_{*}^{\otimes q}$ decomposes as
		$$
		\bV^{\otimes p}\otimes \bV_{*}^{\otimes q}=\bigoplus_{|\boldsymbol{\lambda}|=p,|{\boldsymbol{\mu}}|=q}
		(\mathbb{S}_{\boldsymbol{\lambda}}(\bV)\otimes \mathbb{S}_{\boldsymbol{{\boldsymbol{\mu}}}}(\bV_*))\otimes (Y_{\boldsymbol{\lambda}}\otimes Y_{{\boldsymbol{\mu}}}),
		$$
		where $Y_{\boldsymbol{\lambda}}$ and $Y_{\boldsymbol{\mu}}$ are irreducible $S_p$- and $S_q$-modules,  $\mathbb{S}_{\boldsymbol{{\boldsymbol{\lambda}}}}$ denotes the Schur functor corresponding to the Young diagram  $\boldsymbol{\boldsymbol{\lambda}}$, and $|\boldsymbol{\boldsymbol{\lambda}}|$ is the size of $\boldsymbol{\boldsymbol{\lambda}}$.
		
		The $\mathfrak{sl}(\infty)$-modules  $\mathbb{S}_{\boldsymbol{\lambda}}(\bV)\otimes \mathbb{S}_{{\boldsymbol{\mu}}}(\bV_*)$ are indecomposable, and their socle filtration was described by Penkov and Styrkas in \cite{PStyr}. We recall that the {\em socle} of a module $\bM$, denoted $\soc\bM$,
		is the largest semisimple submodule of $\bM$, and that the {\em socle filtration} of $\bM$ is defined inductively by $\soc^0 \bM:=\soc\bM$ and $\soc^i \bM:=p_i^{-1}(\soc (\bM/(\soc^{i-1}\bM)))$, where $p_i:\bM\to
		\bM/(\soc^{i-1} \bM)$ is the natural projection. We denote the layers of the socle filtration by $\overline{\soc}^i \bM :=\soc^i \bM /\soc^{i-1} \bM $.
		From Thm. 2.3 of \cite{PStyr}, we have that the layers of $\mathbb{S}_{\boldsymbol{\lambda}}(\bV)\otimes \mathbb{S}_{{\boldsymbol{\mu}}}(\bV_*)$ are
		$$
		\overline{\soc}^k(\mathbb{S}_{\boldsymbol{\lambda}}(\bV)\otimes \mathbb{S}_{{\boldsymbol{\mu}}}(\bV_*))\cong\bigoplus_{\boldsymbol{\lambda'},\boldsymbol{\mu'}, |{\boldsymbol{\gamma}}|=k} N^{\boldsymbol{\lambda}}_{\boldsymbol{\lambda}',{\boldsymbol{\gamma}}}N^{\boldsymbol{\mu}}_{{\boldsymbol{\mu}}',{\boldsymbol{\gamma}}}\bV^{\boldsymbol{\boldsymbol{\lambda}}',{\boldsymbol{\mu}}'}
		$$
		where
		$N^{\boldsymbol{\lambda}}_{\boldsymbol{\lambda}',{\boldsymbol{\gamma}}}$ are the standard Littlewood-Richardson coefficients. In particular, the indecomposable module $\mathbb{S}_{\boldsymbol{\lambda}}(\bV)\otimes \mathbb{S}_{{\boldsymbol{\mu}}}(\bV_*)$ has a simple socle, denoted by $\bV^{\boldsymbol{\lambda},\boldsymbol{\mu}}$.
		For example, the layers of  $\Lambda^{m}\bV\otimes\Lambda^{n}\bV_*$ are given by $\overline{\soc}^i (\Lambda^{m}\bV\otimes\Lambda^{n}\bV_*)\cong \bV^{\left(m-i\right)^{\perp},\left(n-i\right)^{\perp}}$,  where $\perp$ indicates the conjugate Young diagram.
		
		An $\mathfrak{sl}(\infty)$-module is called a {\em tensor module} if it is isomorphic to a submodule of a finite direct sum of modules of the form $\bV^{\otimes p_i}\otimes \bV_{*}^{\otimes q_i}$ for $p_i,q_i \in\ZZ_{\geq 0}$.
		The category of tensor modules  $\mathbb{T}_{\mathfrak{sl}(\infty)}$  is by definition the full subcategory of $\mathfrak{sl}(\infty)\textrm{-mod}$  consisting of tensor modules \cite{DPS}.
		The modules $\bV^{\otimes p}\otimes \bV_{*}^{\otimes q}$, $p,q\in\ZZ_{\geq 0}$ are injective in  $\mathbb{T}_{\mathfrak{sl}(\infty)}$. Moreover, every indecomposable injective object of $\mathbb{T}_{\mathfrak{sl}(\infty)}$ is isomorphic to an indecomposable direct summand of  $\bV^{\otimes p}\otimes \bV_{*}^{\otimes q}$  for some $p,q\in\ZZ_{\geq 0}$, which means, it is isomorphic to
		$\mathbb{S}_{\boldsymbol{\lambda}}(\bV)\otimes \mathbb{S}_{\boldsymbol{\mu}}(\bV_*)$ for some $\boldsymbol{\lambda},\boldsymbol{\mu}$
		\cite{DPS}.
		
		\subsection{Representation theory of $\mathfrak{gl}(m|n)$}

		Let $\mathcal{O}_{m|n}$ denote the category of $\ZZ_{2}$-graded
		modules over $\mathfrak{gl}(m|n)$ which when restricted to $\mathfrak{gl}(m|n)_{0}$,
		belong to the BGG category $\mathcal{O}_{\mathfrak{gl}(m|n)_{0}}$  (see \cite{Mu}, Sec. 8.2.3).
		This category depends only on a choice of Borel subalgebra for  $\mathfrak{gl}(m|n)_{0}$, and not for $\mathfrak{gl}(m|n)$. We denote by $\mathcal{O}_{m|n}^{\ZZ}$ the Serre subcategory of $\mathcal{O}_{m|n}$
		consisting of modules with integral weights. Any simple object in $\mathcal{O}_{m|n}^{\ZZ}$ is isomorphic
		to $L(\lambda)$ for some integral weight  $\lambda$ (for a fixed Borel subalgebra of $\mathfrak{gl}(m|n)$). The objects of the category $\mathcal{O}_{m|n}^{\ZZ}$ have finite length. 
		We denote by  $\cF^{\ZZ}_{m|n}$ the Serre subcategory of $\mathcal{O}_{m|n}^{\ZZ}$
		consisting of finite-dimensional modules. (Note that $\cF^{\ZZ}_{m|n}=\cF^{\Lambda_G}(\mathfrak{gl}(m|n))$ as defined in Section \ref{subcats and DS}.)   Each simple object of   $\cF^{\ZZ}_{m|n}$ is isomorphic to $L(\lambda)$ for some dominant integral weight $\lambda$.

		We define the translation functors $\mathrm{E}_i,\mathrm{F}_i$  on the category $\mathcal{O}_{m|n}^{\ZZ}$ as follows.
		Let $X_j,Y_j$ be a pair of $\ZZ_2$-homogeneous dual bases of $\mathfrak{gl}(m|n)$ with respect to the $\mathfrak{gl}(m|n)$-invariant form $\operatorname{str}(XY)$. 
		For a pair of $\mathfrak{gl}\left(m|n\right)$-modules $V,W$, we define the Casimir operator $\Omega\in\operatorname{End}_{\mathfrak{gl}\left(m|n\right)}(V\otimes W)$ on homogeneous vectors by setting
		$$
		\Omega(v\otimes w):=\sum_j(-1)^{p(X_j)(p(v)+1)}X_j v\otimes Y_j w,
		$$
		where $p(\cdot)$ denotes the parity function.
		Let $U,U^*$ be the defining  $\mathfrak{gl}\left(m|n\right)$-modules. Then for every $M\in\mathcal{O}_{m|n}^{\ZZ}$, we let $\mathrm{E}_i(M)$
		(respectively, $\mathrm{F}_i(M)$) be the
		generalized eigenspace of $\Omega$ in $M\otimes U^*$ (respectively, $M\otimes U$) with eigenvalue~$i$.
		Then, as it follows from \cite{BLW}, the functor $\cdot\otimes U^*$ (respectively, $\cdot\otimes U$) decomposes into the direct sum of functors $\oplus_{i\in\ZZ}\mathrm{E}_i(\cdot)$ (respectively,
		$\oplus_{i\in\ZZ}\mathrm{F}_i(\cdot)$). Moreover, the functors $\mathrm{E}_i$ and $\mathrm{F}_i$ are adjoint functors on $\mathcal{O}_{m|n}^{\ZZ}$.

		\subsection{Grothendieck groups and the $\mathfrak{sl}(\infty)$-modules $\bK_{m|n}$, $\bJ_{m|n}$}
		
		We let $\bK_{m|n}$ (respectively, $\bJ_{m|n}$) denote the complexification of the reduced Grothendieck group of  $\mathcal{O}_{m|n}^{\ZZ}$ (respectively, of $\cF^{\ZZ}_{m|n}$), that is, 
		$$\bK_{m|n}:= \Gr_-(\mathcal{O}_{m|n}^{\ZZ})\otimes_{\ZZ}\CC,\hspace{1cm} \bJ_{m|n}:= \Gr_-(\cF^{\ZZ}_{m|n})\otimes_{\ZZ}\CC.$$
		We will denote by $e_i,f_i$ the linear operators that the translation functors $\mathrm{E}_i,\mathrm{F}_i$ induce on $\bK_{m|n}$ and  $\bJ_{m|n}$.
		Brundan showed in \cite{B} that if we identify  $e_i,f_i$ with the Chevalley generators  $E_{i,i+1}$,  $E_{i+1,i}$ of $\mathfrak{sl}(\infty)$, then we obtain  an $\mathfrak{sl}(\infty)$-module structure on $\bJ_{m|n}$ and $\bK_{m|n}$.
		
		Let $\bT_{m|n}\subset \bK_{m|n}$ denote the subspace generated by the classes $[M(\lambda)]$ of all Verma modules $M(\lambda)$ for $\lambda\in\Phi$. Let furthermore $\bW_{m|n}\subset\bJ_{m|n}$ denote the subspace generated by the classes $[K(\lambda)]$ of all Kac modules $K(\lambda)$ for $\lambda\in\Phi^+$.
		Then $\bW_{m|n}$ and $\bT_{m|n}$    are   $\mathfrak{sl}\left(\infty\right)$-modules under the action defined above, and $\bW_{m|n}\cong \Lambda^{m}\bV\otimes\Lambda^{n}\bV_*$ and $\bT_{m|n}\cong \bV^{\otimes m}\otimes \bV^{\otimes n}_*$  \cite{B}. The modules $\bT_{m,n}$ and $\bW_{m|n}$ are injective in the category $\mathbb{T}_{\mathfrak{sl}(\infty)}$, and $\bW_{m|n}$ is an indecomposable summand of  $\bT_{m|n}$.  Now let $\bP_{m,n}:= \Gr_-(\mathcal{P}_{m|n})\otimes_{\ZZ}\CC$ (respectively, $\bQ_{m|n}:= \Gr_-(\mathcal{Q}_{m|n})\otimes_{\ZZ}\CC$), where $\mathcal P_{m|n}$ (respectively, $\mathcal{Q}_{m|n}$) is the semisimple subcategory of $\mathcal{O}_{m|n}^{\ZZ}$  (respectively, of $\cF^{\ZZ}_{m|n}$) consisting of projective modules.  Then we have
		$\soc\bK_{m,n}=\soc\bT_{m,n}=\bP_{m,n}$ and $\soc\bJ_{m,n}=\soc\bW_{m,n}=\bQ_{m,n}$ \cite{HPS,CS}. Consequently, 
		$\bT_{m|n}$ (respectively, $\bW_{m|n}$) is the maximal submodule of $\bK_{m|n}$ (respectively, of $\bJ_{m|n}$) lying in the category $\mathbb{T}_{\mathfrak{sl}(\infty)}$ and in particular, $\bK_{m,n}$ and $\bJ_{m|n}$ are not objects of  $\mathbb{T}_{\mathfrak{sl}(\infty)}$. A new category $\mathbb{T}_{\mathfrak{sl}(\infty),2}$ of $\mathfrak{sl}(\infty)$-modules was introduced in \cite{HPS}  (wherein it is denoted $\mathbb{T}_{\gg,\kk}$) for which $\bK_{m,n}$ and $\bJ_{m|n}$  are injective objects.

		\subsection{The $\DS$ functor on $\mathcal{O}^\ZZ_{m|n}$}
		Let $X$ be the associated variety for  $\mathfrak{gl}(m|n)$, and let $x\in X_{k}=\Phi\left(\cS_{k}\right) $. By Prop. 33 of \cite{HPS}, the restriction of the functor $\DS_x$  to $\mathcal{O}_{m|n}$  is a well-defined functor to
		$\mathcal{O}_{m-k|n-k}$,
		and it follows that the further restriction  to $\mathcal O^{\ZZ}_{m|n}$  gives a well-defined functor
		$$\DS_x:\mathcal O^{\ZZ}_{m|n}\to \mathcal O^{\ZZ}_{m-k|n-k}.$$
		Using the naturality of the Casimir (see \cite{HPS} and \cite{GS}) and that $\DS_x$ is a tensor functor, we obtain that it commutes with translation functors. 
		
		The following propositions are proven in \cite{HPS}.
		
		\begin{proposition} The map $\ds_x:\bK_{m|n}\rightarrow\bK_{m-k|n-k}$ is a homomorphism of $\mathfrak{sl}(\infty)$-modules, and so is its restriction $\ds_x:\bJ_{m|n}\rightarrow\bJ_{m-k|n-k}$.
		\end{proposition}
		
		The map $\ds_x:\bJ_{m|n}\rightarrow\bJ_{m-k|n-k}$ depends only on $k=|S|$ and not on $x$, so we will simply denote it by $\ds$ when $k=1$. (Note that this does not hold for $\bK_{m|n}$.)
		
		The next proposition follows from Theorem~\ref{thm3} (1).
		
		\begin{proposition}
			The kernel of  $\ds:\bJ_{m|n}\rightarrow\bJ_{m-1|n-1}$ is $$\Ker \ds =\bW_{m|n}.$$
		\end{proposition}

		The following result is from
		\cite{HPS}, Prop.~43.
		
		\begin{proposition}
			Fix a nonzero $x\in\gg_{\delta_j-\varepsilon_i}$, and denote by $ds_{i,j}:\bK_{m|n}\to \bK_{m-1|n-1}$
			the $\mathfrak{sl}(\infty)$-module homomorphism $ds_x$.
			We have $$\bigcap_{i,j}\Ker ds_{i,j}=\bT_{m|n}.$$
		\end{proposition}

		\subsection{The socle filtration}\
		
		Here is a description of the $\mathfrak{sl}(\infty)$-module $\bJ_{m|n}$  (see \cite{HPS}, Cor.~29).
		
		\begin{theorem} The module $\bJ_{m|n}$ is an injective hull of the simple module $\bQ_{m|n}$,  and the socle filtration of $\bJ_{m|n}$ has layers
			\[
			\overline{\soc}^i \bJ_{m|n}\cong\left(\bV^{\left(m-i\right)^{\perp}\left(n-i\right)^{\perp}}\right)^{\oplus (i+1)}.
			\]
		\end{theorem}
		
		For a proof of the following theorem, see
		\cite{HPS}, Thm.~24.
		
		\begin{theorem} The $\mathfrak{sl}(\infty)$-module $\bK_{m|n}$ is an injective hull in the category  $\mathbb{T}_{\mathfrak{sl}(\infty),2}$ of the semisimple module $\bP_{m|n}$. Furthermore, there is an isomorphism
			$$
			\bK_{m|n}\cong\bigoplus_{|\boldsymbol{\lambda}|=m,|\boldsymbol{\mu}|=n}\bI^{\boldsymbol{\lambda},\boldsymbol{\mu}}\otimes(Y_{\boldsymbol{\lambda}}\otimes Y_{{\boldsymbol{\mu}}})
			$$
			where $Y_{\boldsymbol{\lambda}}$ and $Y_{\boldsymbol{\mu}}$ are irreducible modules over $S_m$ and $S_n$, respectively, and $\bI^{\boldsymbol{\lambda},\boldsymbol{\mu}}$ is an injective hull of the simple module $\bV^{\boldsymbol{\lambda},\boldsymbol{\mu}}$ in $\mathbb{T}_{\mathfrak{sl}(\infty),2}$. The layers of the socle filtration of $\bK_{m|n}$  are given by
			$$
			\overline{\soc}^k  \bK_{m|n}\cong\bigoplus_{|\boldsymbol{\lambda}|=m,|\boldsymbol{\mu}|=n}(\overline{\soc}^k \bI^{\boldsymbol{\lambda},\boldsymbol{\mu}})^{\oplus (\dim Y_{\boldsymbol{\lambda}} \dim Y_{{\boldsymbol{\mu}}})}
			$$
			where 
			$$\overline{\soc}^k \bI^{\boldsymbol{\lambda},{\boldsymbol{\mu}}}\cong
			\bigoplus_{\boldsymbol{\lambda'},\boldsymbol{\mu'}}
			\bigoplus_{|{\boldsymbol{\gamma}}_1|+|{\boldsymbol{\gamma}}_2|=k}
			N^{\boldsymbol{\lambda}}_{{\boldsymbol{\gamma}}_1,{\boldsymbol{\gamma}}_2,\boldsymbol{\lambda}'}N^{{\boldsymbol{\mu}}}_{{\boldsymbol{\gamma}}_1,{\boldsymbol{\gamma}}_2,{\boldsymbol{\mu}}'}\bV^{\boldsymbol{\lambda}',{\boldsymbol{\mu}}'}.$$
		\end{theorem}\

		
		\section{Projectivity criteria for quasireductive Lie superalgebras}\label{sec proj} 
		In this section, we assume that $\gg$ is a quasireductive Lie superalgebra, that is, $\gg_{0}$ is reductive and acts semisimply on $\gg_1$.
		We discuss to what extent the associated varieties for Lie superalgebras can be used to detect projectivity in the category of finite-dimensional $\gg$-modules.

		\subsection{Projectivity and the associated variety}
		
		Let $ \cF(\gg) $ be the category of finite-dimensional $ \gg $-modules which are semisimple over
		$ \gg_{0} $. The latter condition is automatic if $\gg_{0} $ is semisimple. The
		category $\cF(\gg)$ has enough projective modules and injective modules.
		By duality, every injective object is projective and vice
		versa. Moreover, every indecomposable projective module is a direct
		summand of
		$\operatorname{Ind}^\gg_{\gg_0}L$ for some simple $\gg_0$-module $L$.
		We say that a subalgebra $\kk\subset\gg$ is a quasireductive  subalgebra if $\kk_0$ is reductive
		and $\gg$ is a semisimple $\kk_0$-module.
		
		The following fact is useful.
		\begin{proposition}\label{resproj} Let $\kk$ be a quasireductive subalgebra of $\gg$.
			If $P$ is projective in $\cF(\gg)$ then
			$\operatorname{Res}_{\kk} P$ is projective in $\cF(\kk)$.
		\end{proposition}
		\begin{proof} If $P$ is projective then it is a direct summand of the
			induced module $\operatorname{Ind}^\gg_{\gg_0}N$
			for some semisimple 
			$\gg_0$-module $N$.
			Furthermore, we have an isomorphism of $\gg_0$-modules:
			$$\operatorname{Ind}^\gg_{\gg_0}N\simeq N\otimes S^\bullet(\gg_1)$$
			and an isomorphism of $\kk_0$-modules
			$$\gg_1=\kk_1\oplus(\gg_1/\kk_1).$$
			By Frobenius reciprocity the homomorphism of $\kk_0$-modules $$N\otimes S^\bullet(\gg_1/\kk_1)\to\operatorname{Ind}^\gg_{\gg_0}N$$
			induces an isomorphism
			$$\operatorname{Ind}^\kk_{\kk_0}(N\otimes S^\bullet(\gg_1/\kk_1))\simeq \operatorname{Res}_{\kk}\operatorname{Ind}^\gg_{\gg_0}N.$$
			We obtain that $\operatorname{Res}_{\kk}P$ is a direct summand of some module induced from a
			semisimple $\kk_0$-module. Therefore $P$ is projective
			in $\cF(\kk)$.
		\end{proof}
		
		We can now give another proof of Lemma~\ref{proj_lemma}.
		\begin{theorem} \label{th142}\myLabel{th142}\relax  Suppose $\gg$ is quasireductive.
			If $ M\in\cF(\gg) $ is projective, then $ X_M=\left\{0\right\} $.
		\end{theorem}
		\begin{proof}
			Let $x\in X$ be nonzero, and consider the subalgebra $\kk=k\langle x\rangle$ generated by $x$.  Since $\kk$ is quasireductive, $\Res_{\kk}M$ is projective over $\kk$, which implies that $M_x=0$.
		\end{proof}
		\subsection{Criteria for type I Lie superalgebras}
		
		In this section we prove that for certain quasireductive Lie superalgebras the converse of Theorem~\ref{th142} holds.
		We start with the following.
		\begin{lemma}\label{lmcom} Let $\gg$ be quasireductive and $[\gg_1,\gg_1]=0$. If $X_M=\{0\}$ then $M$ is projective in $\cF(\gg)$.
		\end{lemma}
		\begin{proof} Since $\mathcal{U}(\gg_1)$ is isomorphic to the exterior algebra $\Lambda(\gg_1)$ we have that $X_M=\{0\}$ implies that $M$ is free over $\mathcal{U}(\gg_1)$, see \cite{AB}. Then an embedding of $\gg_0$-modules $M/\gg_1M\hookrightarrow M$ induces an isomorphism $\Ind^\gg_{\gg_0}(M/{\gg_1}M)\simeq M$. Therefore $M$ is projective.
		\end{proof}
		\begin{theorem}\label{thtype1} Assume that $\gg_0$ is reductive and there exists an element $h$ in the center of $\gg_0$ such that $\ad_h$ acts diagonally on
			$\gg_1$ with nonzero real eigenvalues. If $X_M=\{0\}$ then $M$ is projective in $\cF(\gg)$.
		\end{theorem}
		\begin{proof} Write down $\gg=\gg_+\oplus\gg_0\oplus\gg_-$, where
			$\gg_{+}$ (respectively, $\gg_{-}$)
			denote the span of $\ad_h$-eigenvectors with positive (respectively, negative)
			eigenvalues. Since $\gg_{\pm}$ are purely odd subalgebras, they are commutative, hence $\pp_+:=\gg_0\oplus\gg_+$ and $\pp_-:=\gg_0\oplus\gg_-$ are
			subalgebras satisfying the condition of Lemma~\ref{lmcom}. In particular, if $X_M=\{0\}$, then $M$ is projective in  $\cF(\pp_\pm)$.
			For a  $\gg_0$-module $L$ set
			$K^{\pm}(L):=\Ind^\gg_{\pp_\pm}L$. We claim that there exists a finite filtration
			$$0=M_0\subset M_1\subset\dots\subset M_k$$
			such that $M_i/M_{i-1}\simeq K^-(L_i)$. Indeed, let $L_1$ be $h$-eigenspace with maximal eigenvalue. Then $\gg_+L_1=0$ and we have an embedding $K^-(L_1)\subset M$.
			The quotient $M/K^-(L_1)$ is again free over $\mathcal{U}(\gg_+)$ and projective in $\cF(\pp_+)$. Hence we can finish the proof by induction on dimension of $M$.
			Similarly $M^*$ has a finite filtration with quotients isomorphic to $K^+(N_j)$. Therefore $M\otimes M^*$ has a filtration with quotients isomorphic to
			$K^+(L_i)\otimes K^-(N_j)\simeq\Ind^\gg_{\gg_0}(L_i\otimes N_j)$. In other words  $M\otimes M^*$ has a filtration by projective modules. Therefore  $M\otimes M^*$  is projective
			in $\cF(\gg)$.
			Then $M\otimes M^*\otimes M$ is also projective. In any symmetric monoidal rigid category $M$ is a direct summand of $M\otimes M^*\otimes M$.
			Therefore $M$ is projective.
		\end{proof}
		\begin{corollary}\label{type1} Theorem~\ref{thtype1} holds for $\mathfrak{gl}(m|n)$, $\mathfrak{sl}(m|n)$, $m\neq n$, $\mathfrak{osp}(2|2n)$ and $\pp(n)$.
		\end{corollary}
		\begin{remark}\label{counterexample} Let
			$\gg=\mathfrak{sl}(1|1)$. Then it is easy to construct a
			$\gg$-module $M$ such that
			$X_M=0$ and $M$ is not projective. Recall that $\gg_1$ has a basis $\{x,y\}$ and $\gg_0=\CC z$ 
			with $[x,y]=z$, $[z,x]=[z,y]=0$. Then $X=\CC x\cup\CC y$.
			Let $M=\CC^{1|1}$, $z$ acts trivially on $M$, and both $x$ and $y$ act via the same matrix
			$\left(\begin{matrix}0&1\\0&0\end{matrix}\right)$. Clearly $M$ is
			not projective. Note that the $\mathfrak{sl}(1|1)$-module $M$ is not the restriction of a $\mathfrak{gl}(1|1)$-module.
		\end{remark}

		\subsection{Rank varieties}\label{section ss support}  
		From reductivity, we have a well-defined notion of semisimple elements in $\gg_0$, so the following definition makes sense.
		\begin{definition}
			We set
			\[
			\gg_1^{hom}:=\{x\in\gg_{1}:[x,x]\text{ is semisimple}\}.
			\]
			We refer to elements of $\gg_1^{hom}$ as homological elements.
		\end{definition}
		\begin{remark}\label{odd ss elem comment}
			Clearly we have $X\subseteq\gg_1^{hom}$.  Further $\gg_1^{hom}$ is $G_0$-stable, just like $X$.  However $\gg_1^{hom}$ is no longer closed in $\gg_1$, and its geometric structure is much more complicated.
		\end{remark}
		
		Let $x\in\gg_1^{hom}$ and write $h=[x,x]$.  Then for $M$ in $\cF(\gg)$, if we consider $M^h$, the fixed points of $h$ on $M$, it is $x$-stable and further $x$ defines a square-zero endomorphism on it.  Thus we may define
		\[
		M_x:=(\ker x|_{M^h})/(\Im x|_{M^h}).
		\]
		This defines a functor which we continue to call $DS_x$, the Duflo--Serganova functor for the element $x$.  Note that the Duflo--Serganova functor as we defined it in Section~\ref{sec:definitions} 
		comes from the case when $h=0$. 
		
		\begin{remark}
			It is easy to check that Lemma~\ref{gx Mx} and Lemma~\ref{tensor} hold for this generalization of the $DS$ functor.  
		\end{remark}
		
		The following space was considered in \cite{ES4}.
		\begin{definition}
			Let $M$ be in $\cF(\gg)$ and define the rank variety of $M$ to be
			\[
			X_M^{rk}:=\{x\in\gg_1^{hom}:M_{x}\neq0\}
			\]
		\end{definition}
		Again we have that $X_M^{rk}\subseteq\gg_1^{hom}$ is $G_0$-stable; however as is hinted in Remark~\ref{odd ss elem comment}, the geometric structure of $X_M^{rk}$ can be quite complicated.  
		
		We note that rank varieties share many of the same properties as associated varieties; in particular, all properties from Lemma~\ref{lm2} continue to hold.  In particular, we can use an analogous proof as in Theorem~\ref{th142} to show that:
		
		\begin{proposition}
			Let $P$ be projective in $\cF(\gg)$.  Then $X_{P}^{rk}=\{0\}$.
		\end{proposition}
		
		We make the following conjecture:
		
		\begin{conjecture}\label{ss support conj}
			Let $\gg$ be quasireductive and suppose that $M$ is in $\cF(\gg)$ with $X_{M}^{rk}=\{0\}$.  Then $M$ is projective.
		\end{conjecture}
		A proof of this conjecture is currently forthcoming, and is being worked on by the third and fourth authors with other collaborators. 
		
		\begin{example}
			Consider the example given in Remark~\ref{counterexample}.  For $\gg=\mathfrak{sl}(1|1)$ we have $\gg_1^{hom}=\gg_1$.  Clearly for the module $M$ considered there, $X_M^{rk}=\{c(x-y):c\in\CC\}$.  
		\end{example}

		
		\section{Localization of the $\DS$ functor}\label{app cohom} 
		
		In this section, we associate to every finite-dimensional $\gg$-module a vector bundle on $X$ with a square-zero $\mathcal{O}_X$-module endomorphism, which interpolates the actions of the elements of $X$.  We relate the cohomology of this operator to the associated variety of $M$, and apply it to a cohomology computation for $\mathfrak{gl}(m|n)$.

		\subsection{Localization}
		Let $\gg$ be a finite-dimensional Lie superalgebra, and let $M$ be a $\gg$-module. Let $ {\mathcal O}_{X} $ denote the structure sheaf of $ X $. Then $ {\mathcal O}_{X}\otimes M $ is the sheaf of
		sections of the trivial vector bundle with fiber isomorphic to  $ M $. Let $ \partial:{\mathcal O}_{X}\otimes M \to
		{\mathcal O}_{X}\otimes M $ be the map defined by
		\begin{equation}
			\partial\varphi\left(x\right)=x\varphi\left(x\right)
			\notag\end{equation}
		for any $ x\in X $, $ \varphi\in{\mathcal O}_{X}\otimes M $. Clearly
		$ \partial^{2}=0 $ and the cohomology
		$ {\mathcal M} $ of $ \partial $ is a
		quasi-coherent sheaf on $ X $. If $ M $ is finite-dimensional, then $ {\mathcal M} $ is coherent.
		
		For any $ x\in X $ denote by $ {\mathcal O}_{x} $ the local ring at $ x $, by $ {\mathcal I}_{x} $ the maximal ideal.
		Let $ \tilde{\mathcal M}_{x} $ be the cohomology of
		$ \partial:{\mathcal O}_{x}\otimes M \to {\mathcal O}_{x}\otimes M $ and $\mathcal M_x:= \tilde{\mathcal M}_{x} /{\mathcal I}_{x}\tilde{\mathcal M}_{x} $.
		The evaluation
		map $ j_{x}:{\mathcal O}_{x}\otimes M \to M $ satisfies $ j_{x}\circ\partial=x\circ j_{x} $. Hence we have the maps
		\begin{equation}
			j_{x}:\operatorname{Ker} \partial \to\operatorname{Ker} x\text{, }j_{x}:\operatorname{Im} \partial \to xM.
			\notag\end{equation}
		The  embedding $M\hookrightarrow {\mathcal O}_{x}\otimes M$ ensures the surjectivity of the latter map. Thus, $ j_{x} $ induces
		the map $ \bar{j}_{x}\colon \tilde{\mathcal M}_{x} \to M_{x} $, and
		$\operatorname{Im} \bar{j}_{x}\cong{\mathcal M}_{x} $.
		
		\begin{remark} It is easy to see that ${\mathcal M}_{x}$ is a  $(\gg_x)_0$-module and $\bar j_x$ is a homomorphism of $(\gg_x)_0$-modules.
		\end{remark}
		\begin{lemma} Let $ M $ be a finite-dimensional $ \gg $-module.
			\begin{enumerate}
				\item The support of $ {\mathcal M} $ is
				contained in
				$ X_M $.
				\item The map $ \bar{j}_{x} $ is surjective for a generic point $ x\in X $. In particular,
				if $ X_M=X $, then $ \operatorname{supp} {\mathcal M}=X $.
			\end{enumerate}
		\end{lemma}
		
		\begin{proof} First, we will show that for any $ x\in X\backslash X_M $ there exists a
			neighborhood $ U $ of $ x $ such that $ {\mathcal M}\left(U\right)=0 $. Indeed, there exists 
			$ h_{x}\in\End_{\CC}(M)$ such that $ x\circ h_{x}+h_{x}\circ x=id_M. $ Therefore in some neighborhood $ U $ of $ x $ there
			exists an $\mathcal O(U)$-morphism 
			$ h:{\mathcal O}\left(U\right)\otimes M \to {\mathcal O}\left(U\right)\otimes M $ such that
			$\partial\circ h+h\circ\partial$ is invertible and $h(x)=h_x$. Hence
			the cohomology of $\partial:{\mathcal O}\left(U\right)\otimes M \to
			{\mathcal O}\left(U\right)\otimes M $ are trivial.
			In other words, $ {\mathcal M}\left(U\right)=0 $.
			Thus, $ x $ does not belong to the support of $ {\mathcal M} $ and we have obtained that $\operatorname{supp}{\mathcal M}\subset X_M$.
			
			To prove (2) let $ x\in X $ be a non-singular point such
			that $ \dim M_x $ is minimal. Let $ m\in\operatorname{Ker} x_M$. Then there exists some neighborhood $ U $ of $ x $ and
			$ \varphi\in{\mathcal O}\left(U\right)\otimes M $ such that
			$ \partial\varphi=0 $ and $ \varphi\left(x\right)=m $. By definition
			$ \varphi\in{\mathcal M}_{x} $ and $ \bar{j}_{x}\left(\varphi\right)=m $.\end{proof}
		
		\begin{corollary}  Let $ x\in X $ be a generic point, then in some neighborhood $ U $
			of $ x $, the sheaf $ {\mathcal M}_{U} $ coincides with the sheaf of sections of a vector bundle
			with fiber $ M_{x} $.
		\end{corollary}
		
		Let $ X_M\not=X $. Then $ {\mathcal M} $ is the extension by zero of the sheaf $ {\mathcal M}_{X_M} $ and $ {\mathcal M}_{X_M} $ locally is the sheaf of
		sections of a
		vector bundle with fiber $ \bar{j}_{x} ({\mathcal M}_x)$ for a generic $ x\in X_M $. Note that
		$\bar{j}_{x} ({\mathcal M}_x)\subset M_{x} $, but usually this is a
		strict embedding, as one can see from the following example.
		
		\begin{example}  Let $ \gg=\mathfrak{gl}\left(1|1\right)$ and $ M $ be the standard
			$\gg $-module. Then $$X=\left\{\left (\begin{matrix}0&u\\v&0\end{matrix}\right )\mid uv=0\right\}.$$
			Therefore $ X_M=\left\{0\right\} $, but a simple computation shows that $\mathcal M=0$, and in particular, the support of $\mathcal M$ is empty.
		\end{example}
		
		For $x\in X$, let $G_0^x$ denote the stabilizer of $x$ in $G_0$.
		The following statement illustrates a geometric meaning of $(\gg_x)_1$.
		\begin{proposition} Let $x\in X$. Then the $G_0$-vector bundle
			$G_0\times_{G_0^x}(\gg_x)_1$ is canonically isomorphic to the normal bundle to $G_0x$ in $X$.
		\end{proposition}
		\begin{proof} First, we compute the tangent space $T_xX$. The condition $$[x+\epsilon y,x+\epsilon y]=0\mod\epsilon^2$$ is equivalent to $y\in\Ker\,\ad_x$.
			Therefore $T_xX\cong (\Ker\,\ad_x)_1$. On the other hand, the tangent space $T_x(G_0x)$ to the orbit is canonically isomorphic to $[\gg_0,x]=(\Im \,\ad_x)_1$.
			Hence the normal space to $G_0 x$ in $X$ at the point $x$ is isomorphic to $(\gg_x)_1$.  Using $G_0$-action we obtain
			$$\mathcal N_{G_0 x}X\cong G_0\times_{G_0^x}(\gg_x)_1.$$
		\end{proof}
		
		\subsection{A special $G_0$-invariant subset for basic classical Lie superalgebras}
		
		For this subsection, let $\gg$  be a basic classical Lie superalgebra.

		Let $ x\in\overline{X}_{k} $ and
		\begin{equation}
			Y_{x}:=\left\{y\in\left(\gg_{x}\right)_{1} \mid \left[y,y\right]=0\right\}.
			\notag\end{equation}
		Then
		\begin{equation}
			x+Y_{x}\subset X
			\label{equ83}\end{equation}\myLabel{equ83,}\relax 
		
		The following is a consequence of Theorem~\ref{th1} and Lemma~\ref{lm110}.
		
		\begin{corollary}\label{cor:trans}   Let $x\in X$  and denote by $X'\subset X$ the union of all $G_0$-orbits $O$ such that $x\in\overline O$.
			Then $X'=G_0(x+Y_{x})$,
		\end{corollary}

		\begin{lemma}\label{compstab} Let $x\in X$. There exists a subgroup $Q\subset G_0$ satisfying the following properties
			\begin{enumerate}
				\item $G_0^{x+y}\cap Q=\{1\}$ for any $y\in Y_x$,
				\item $Q (G_0^x)$ is Zariski dense in $G_0$. 
			\end{enumerate}
		\end{lemma} 
		\begin{proof} First, we check the statement for classical $\gg$. We denote by $V$ the defining representation of $\gg$. 
			Then for some subspace $V'\subset V$  we have a decomposition $V=xV'\oplus V_x\oplus V'$, and we may assume that $V'$, $xV'$ are isotropic subspaces and orthogonal to $V_x$ in the orthosymplectic case.
			We set 
			$$Q:=\{g\in G_0\mid g|_{W}=\operatorname{id}_{V'},\, g(V_x)\subset W\oplus V_x,\, g|_{V_x}\equiv \operatorname{id}_{V_x}\mod V'\}.$$ 
			
			Now let $\gg$ be exceptional and $x\neq 0$. Then $G_0^x$ is a subgroup of codimension $1$ in some parabolic $P$ with maximal normal unipotent subgroup $U$. We set 
			$Q=\CC^*\rtimes U^{-}$ where $U^-$ is the opposite (complementary) to $U$ and $\CC^*$ be a one-parameter subgroup in the maximal torus of $G_0$
			which acts freely on $\CC^*x$.
		\end{proof}

		\begin{lemma}\label{fiber} Let $N:=M_x$ and $\mathcal N_0$ denote the fiber at $0$ of the sheaf $\mathcal N$ on $Y_x$. Then there exists an injective morphism
			$\mathcal N_0\hookrightarrow\mathcal M_x$. 
		\end{lemma}
		\begin{proof} The action map $Q\times (x+Y_x)\to X$ defines an isomorphism $a:\mathcal U\to Q\times Y_x$ for some Zariski dense open $\mathcal U\subset G'$. 
			Denote by $p$ the composition of $a$ with the projection. Then $p^*:\mathcal N\to \mathcal M(\mathcal U)$ is injective and hence induces an injection of fibers.
		\end{proof}
		
		\begin{lemma}\label{lem:closure} Let $x\in X$ and $K$ is the algebraic subgroup of $G_0$ with the Lie algebra $(\gg_x)_0$. For every $x'\in X'$ we have 
			$G_0x'\cap (x+Y_x)=K(x+y)$ for some $y\in Y_x$. Thus we have a
			bijection between $G_0$-orbits in $X'$ and $K$-orbits in $Y_x$.\end{lemma}
		\begin{proof} Let $x=\sum_{\alpha\in A}x_\alpha$ and $\cS'=\{B\in \cS\mid A\subset B\}$. Then $G_0$-orbits in $X$ are in bijection with $\cS'/W$ and $K$-orbits in $Y_x$
			are in bijection with $\cS_{\gg_x}/W_{\gg_x}$, where $\cS_{\gg_x}$ and $W_{\gg_x}$ are analogues of $\cS$ and $W$ for $\gg_x$. The map $\cS'\to \cS_{\gg_x}$ defined
			by $B\to B\backslash A$ induces the bijection $\cS'/W\to \cS_{\gg_x}/W_{\gg_x}$. Hence the statement.
		\end{proof}

		\subsection{Application to cohomology of finite-dimensional $ \mathfrak{gl}\left(m|n\right) $-modules} For the rest of this section $ \gg=\mathfrak{gl}\left(m|n\right) $. Recall the grading $\gg=\gg^1\oplus\gg^0\oplus\gg^{-1}$ and observe that the
		abelian subalgebra $\gg^1$ is an irreducible component of $X$. We can identify $\gg^1$ with $\Hom_{\CC}(\CC^n,\CC^m)$.
		Then
		$$\gg^1_k:=X_k\cap\gg^1=\{\varphi\in \Hom_{\CC}(\CC^n,\CC^m) \mid \rank\varphi=k\},$$
		$\gg^1_k$ is a single $G_0$-orbit.
		
		Let $M$ be a $\gg$-module. The restriction $\mathcal M^h$ of  $\mathcal M$ to $\gg^1$
		is given by the cohomology
		\begin{equation}
			\partial:{\mathcal O}_{\gg^1}\otimes M \to {\mathcal O}_{\gg^1}\otimes M,
			\notag\end{equation}
		where $ \partial $ is the same as for the sheaf $ {\mathcal M} $. The complex of global section equipped with the standard grading 
		$$\dots\to S^p(\gg^1)^*\otimes M\to\dots\to (\gg^1)^*\otimes M\to M\to 0$$
		is nothing else but the Koszul complex computing the cohomology $H^\bullet(\gg^1,M)$. These cohomology groups are important since they are used in the Kazhdan--Lusztig 
		theory for $\cF(GL(m|n))$,  \cite{S1}, \cite{B}. The sheaf $\mathcal M^h$ can be considered as the localization of $H^\bullet(\gg^1,M)$ in the sense
		of Beilinson--Bernstein. It is clear that 
		\begin{equation}\label{eqsup}
			\operatorname{supp}{\mathcal M}^h\subset X_M\cap\gg^1.
		\end{equation}
		
		\begin{lemma}\label{typcoh} If $M$ admits a typical central character, then $\operatorname{supp}{\mathcal M}^h=\{0\}$ and the fiber of $\mathcal M^h$ at 
			$0$ equals $H^0(\gg^1,M)$.
		\end{lemma}
		\begin{proof} Follows from the fact that $M$ is a free $\mathcal{U}(\gg^1)$-module and Koszul duality between $\mathcal{U}(\gg^1)$ and $S(\gg^1)$.
		\end{proof}

		\begin{theorem} \label{th99}\myLabel{th99}\relax  Let $ M $ be an irreducible finite-dimensional $\gg$-module with
			atypicality degree $ k $. Then $ \operatorname{supp}
			{\mathcal M}^h=\overline{X}_{k}\cap\gg^1$.
		\end{theorem}
		\begin{proof} The inclusion $ \operatorname{supp}{\mathcal M}^h\subset\overline{X}_{k}\cap\gg^1$ follows from Theorem~\ref{th3} and (\ref{eqsup}).
			To prove the equality consider $x\in \gg_k^1$. The fiber $\mathcal M^h_x\neq 0$ by Lemma~\ref{typcoh} and Lemma~\ref{fiber}.
		\end{proof} 
		
		Consider the Hilbert--Poincare series
		$$H_M(t):=\sum_{i=1}^\infty \dim H^i(\gg^1,M)t^i.$$
		The Hilbert--Serre Theorem and Theorem~\ref{th99} imply
		
		\begin{corollary}  Let $ M $ be an irreducible finite-dimensional $\gg$-module with
			atypicality degree $ k $. Then
			$$H_M(t)=\frac{q(t)}{(1-t)^{k(m+n-k)}}$$
			for some polynomial $q(t)$.
		\end{corollary}
		\begin{proof} The degree in the denominator equals $\dim\gg^1_k=k(m+n-k)$.
		\end{proof}


		\section{The $\DS$ functor on simple modules}\label{action of DS}
		
		In this section we discuss what is known about the action of $\DS_x$ on simple modules for classical Lie superalgebras.  Serganova originally conjectured that these functors are semisimple when $\gg$ is basic classical, meaning that they takes semisimple modules to semisimple modules.  Following the work of \cite{HsW} and \cite{GH1} this is now a theorem.  For $\pp(n)$ these functors are known not to be semisimple; however by the work of \cite{ES2}, the composition factors of $\DS_x(L)$ for $x$ of rank one when $L$ is a simple module.  
		
		The case of $\qq(n)$ was studied in \cite{GSh}; in this case, it is interesting to study $DS_x$ on simple modules for all $x\in\gg_{1}^{rk}$ (see Section \ref{section ss support}), not only for those $x$ with $[x,x]=0$. Note however that $\gg_{1}^{rk}$ has infinitely many $G_0$-orbits.  In \cite{GSh}, $DS_x$ is computed on simple modules and shown to be semisimple for $x$ of rank 1.  In general it is expected that $DS_x$ will be always be semisimple.   We will not discuss this case here any further.

		\subsection{General results} We begin with a general statement. Recall that if $N$ is a $\gg$-module and $L$ is a simple $\gg$-module, we write $[N:L]_{non}$ for the ungraded Jordan-H\"older multiplicity of $L$ in $N$, meaning for the number of times both $L$ and $\Pi L$ appear as Jordan-H\"older factors of $N$. 
		
		The following result is a compilation of results from \cite{HsW}, \cite{GH1}, \cite{M}, and \cite{ES2}.  
		\begin{theorem}\label{gen thm DS simples}
			Let $\gg$ be one of the Lie superalgebras $\gg\ll(m|n),\mathfrak{osp}(m|2n),\pp(n)$, or a simple exceptional Lie superalgebra.  Let $L$ be a simple $\gg$-module, $x$ a rank one odd root vector (see Definition~\ref{rank_defn}), and $L'$ a simple $\gg_x$-module.  
			\begin{enumerate}
				\item If $\gg\neq \pp(n)$, then $\DS_x(L)$ is a semisimple $\gg_x$-module.
				\item We have $[\DS_x(L):L']_{non}\leq 2$; if $\gg=\gg\ll(m|n),\pp(n)$ then $\DS_x(L)$ is multiplicity free.
				\item (Purity) If $\gg$ is basic classical ($\gg\neq\pp(n)$) then we have
				\[
				[\DS_x(L):L'][\DS_x(L):\Pi L']=0.
				\]
				\item For $\gg=\gg\ll(m|n),\mathfrak{osp}(m|2n)$, and $\pp(n)$, the composition factors of $\DS_x(L)$ are determined by removing maximal arcs from the arc diagram associated to $L$ (see the subsections below for explanations on the arc diagrams for each case).
			\end{enumerate}
		\end{theorem}
		
		\begin{remark}
			In (2) of Theorem~\ref{gen thm DS simples} the multiplicity bound states that $[\DS_xL:L']_{non}\leq1$ when $\gg$ is of type I, i.e.,~when it has a $\ZZ$-grading (i.e.,~$\gg\ll(m|n),\mathfrak{osp}(2|2n)$, or $\pp(n)$).  
			
			The proof of these bounds (and the rest of the results above) are still case-dependent; general proofs are unknown but would be of great interest.
		\end{remark}
		
		\begin{remark}\label{purity_rmk}
			There is an elegant explanation of the purity property, i.e., part (3) of Theorem~\ref{gen thm DS simples}, which is explained by Gorelik in \cite{G4}.  For $\gg$ basic classical there exists a semisimple subcategory $\mathcal{C}(\gg)$ of the category of finite-dimensional modules such that for any simple module $L$ of $\gg$, exactly one of $L$ or $\Pi L$ lies in $\mathcal{C}(\gg)$.  These semisimple subcategories can be chosen so that if $L$ lies in $\mathcal{C}(\gg)$ then $\DS_xL$ lies in $\mathcal{C}(\gg_x)$, which of course implies part (3) of Theorem~\ref{gen thm DS simples}.     
			
			Further, it is possible to choose these semisimple subcategories so that when $\gg$ is reductive (i.e.,~$\gg=\gg_{0}$ or $\gg=\mathfrak{osp}(1|2n)$), $\mathcal{C}(\gg)$ contains all simple modules of positive superdimension.  Using this and the fact that the $\DS$ functor preserves superdimension, one can obtain combinatorial formulas given by a sum of non-negative numbers for the superdimension of any simple module.  This was done in \cite{HsW} for $\gg\ll(m|n)$.  For $\pp(n)$ the superdimension was computed in \cite{ES2}.
		\end{remark}
		
		Before explaining part (4) of Theorem~\ref{gen thm DS simples} and beginning the discussion of arc diagrams, we state a result from which we compute the value of any $\DS$ functor for $\gg=\gg\ll(m|n),\mathfrak{osp}(m|2n)$ on any simple module.
		
		Let $x_r\in X$ be a rank $r$ element of the associated variety of $\gg$.  Write $\DS^1$ for the functor obtained by applying $\DS_x$ for a rank one vector $x$.  
		\begin{theorem}[See \cite{HsW}, \cite{GH1}]
			Let $\gg=\gg\ll(m|n)$ or $\mathfrak{osp}(m|2n)$.  For a simple $\gg$-module $L$ we have an isomorphism of $\DS_{x_r}(\gg)\cong  \DS^1(\DS^1(\cdots \DS^1(\gg)\cdots ))$-modules:
			$$
			\DS_{x_r}(L)\cong \DS^1(\DS^1(\cdots \DS^1(L)\cdots )).
			$$
		\end{theorem}

		

		


		We will now discuss arc diagrams and part (4) of Theorem~\ref{gen thm DS simples}, after which we will explain the case of the exceptional superalgebras.
		
		\begin{remark}
			As we will see, the calculus of arc diagrams below will explain how to compute $DS^1$ on simple modules with integral weights.  By Section~\ref{subcats and DS}, this is enough to compute $DS^1$ on all of $\cF(\gg)$.
		\end{remark}
		
		\subsection{An overview of arc diagrams} 
		
		We now begin the explanation of part (4) of Theorem~\ref{gen thm DS simples}, which will consume the rest of this section.  For the classical series $\gg\ll(m|n),\mathfrak{osp}(m|2n)$, and $\pp(n)$, there is a remarkable thread which links the computations of the composition factors of $\DS_xL$ for a simple module $L$, namely arc diagrams. (For $\mathfrak{gl}(m|n)$ and $\pp(n)$ these diagrams have also been called cap/cup diagrams, in \cite{HsW} for $\mathfrak{gl}(m|n)$ and in \cite{ES2} for $\pp(n)$; we have changed the name for the sake of consistency.)   These arc diagrams are defined individually for each superalgebra and provide a combinatorial tool to study this question.
		
		We summarize the situation as follows.  For each of the four Lie superalgebras listed above, we explain a procedure which associates to each simple module $L$ an arc diagram, which consists of symbols lying on (half)-integer points on the real line, along with arcs which connect them.  These arcs sometimes are nested within one another, giving rise to the notion of maximal arcs, those which do not lie beneath another arc.  Then, as is stated in Theorem~\ref{gen thm DS simples}, the composition factors of $\DS_xL$ are given by the simple modules whose associated arc diagram is obtained by removing one maximal arc from the arc diagram of $L$.  As will be seen, the procedure for defining arc diagrams is different for each superalgebra.
		
		The idea of using arc diagrams to study the representations of Lie superalgebras goes back to the work of Brundan and Stroppel, where they realized the category $\operatorname{Rep} GL(m|n)$ as a certain diagram algebra of Khovanov type (see \cite{BS}).  Their arc diagrams are, notation aside, in essence the same as what will define for $\gg\ll(m|n)$ below. 
		
		For the orthosymplectic supergroup, Gruson and Serganova used arc diagrams for `tailess' dominant weights in \cite{GrS2}.  More recently, Ehrig and Stroppel have done similar work on realizing $\operatorname{Rep} OSp(m|2n)$ as a certain diagram algebra, (see \cite{EhSt}).  Their diagram algebra is related to type D Khovanov algebras; however, their arc diagrams differ from those used in \cite{GH1} to study the action of $\DS_x$ on simple modules.  A dictionary to go between them is described in Appendix A of \cite{GH1}.

		\begin{remark}
			There is an interesting link between arc diagrams and the computations of character formulas for $\gg\ll(m|n)$ and $\mathfrak{osp}(m|2n)$ (see \cite{GH2}) as well as for $\qq(n)$ (see \cite{SuZh}).  A similar connection is expected for $\pp(n)$ as well.  
		\end{remark}

		We now begin our case by case explanations of arc diagrams. We will write $\Lambda_{m|n}$ for the free $\ZZ$-module with basis $\epsilon_1,\dots,\epsilon_m,\delta_1,\dots,\delta_n$ which will be used for $\gg=\mathfrak{gl}(m|n)$, $\mathfrak{osp}(2m|2n)$, and $\mathfrak{osp}(2m+1|2n)$.  For these superalgebras we define a parity homomorphism $p:\Lambda_{\gg}\to\mathbb{Z}_2=\{0,1\}$ by $p(\epsilon_i)=0$ and $p(\delta_j)=1$ for all $i,j$, and extending linearly. For $\gg=\pp(n)$ we will use that lattice $\Lambda_{n}$, which is the free $\mathbb{Z}$-module with basis $\epsilon_1,\dots,\epsilon_n$.   
		
		\subsection{$\gg\ll(m|n)$ case} 
		
		The $\gg=\gg\ll(m|n)$ case is due to \cite{HsW}, and we refer the reader there for full details and more in-depth results.
		
		We take the Borel subalgebra corresponding to the simple roots
		\[
		\epsilon_1-\epsilon_2,\dots,\epsilon_m-\delta_1,\dots,\delta_{n-1}-\delta_n.
		\]
		Let 
		\[
		\rho=-\epsilon_2-2\epsilon_3-\dots-(m-1)\epsilon_m+
		(m-1)\delta_1+(m-2)\delta_2+\dots+(m-n)\delta_n.
		\]
		We identify $\lambda\in\Lambda_{m|n}$ with the $(m|n)$-tuple of integers
		\[
		(a_1,\dots,a_m|b_1,\dots,b_n)
		\]
		where 
		\[
		\lambda+\rho=a_1\epsilon_1+\dots+a_m\epsilon_m-b_1\delta_1-\dots-b_n\delta_n.
		\]
		
		We write $\Lambda^+(\gg\ll(m|n))$ for the set of dominant weights in $\Lambda_{m|n}$ with respect to this Borel.  Then $\lambda$ is dominant if and only if $a_1>\dots>a_m$  and $b_1>\dots>b_n$.
		
		\subsubsection{Weight and arc diagrams} Write $I_{<}(\lambda)=\{a_1,\dots,a_m\}$ and $I_{>}(\lambda)=\{b_1,\dots,b_n\}$.  Then define the weight diagram associated to $\lambda$ to be the following labelling $f_{\lambda}:\ZZ\to\{\times,\circ,<,>\}$:  
		\[
		f_{\lambda}(k)=\left\{ \begin{array}{@{}rl@{}}
			\times & k\in I_{<}(\lambda)\cap I_{>}(\lambda);\\
			\circ & k\notin I_{<}(\lambda)\cup I_{>}(\lambda);\\
			< & k\in I_{<}(\lambda)\setminus (I_{>}(\lambda)\cap I_{<}(\lambda)) ;\\
			> & k\in I_{>}(\lambda)\setminus (I_{<}(\lambda)\cap I_{>}(\lambda));\\
		\end{array}\right.
		\]
		The correspondence $\lambda\mapsto f_{\lambda}$ defines a bijection between $\Lambda^+(\gg\ll(m|n))$ and the labelings of $\ZZ$ by the appropriate number of the symbols $\times,\circ,<,>$.
		
		\begin{remark}
			It is not hard to check that the atypicality of a dominant weight $\lambda$ is equal to the number of symbols $\times$ in its weight diagram.
		\end{remark}
		Given a weight diagram $f_{\lambda}$ we associate an arc diagram via the following inductive construction.  Connect an arc between $i<j$ if:
		\begin{enumerate}
			\item $f(i)=\times$, $f(j)=\circ$; and
			\item for all $k$ satisfying $i<k<j$ and $f_{\lambda}(k)=\circ$, $k$ already lies on a previously drawn arc.
		\end{enumerate}
		In other words, if $f(i)=\times$, $f(j)=\circ$, and for all $i<k<j$ we have $f(k)\notin\{\times,\circ\}$, then we connect $i$ and $j$ by an arc.  Then we continue drawing arcs inductively according to the above procedure.

		\begin{example}
			For $\gg=\gg\ll(n|n)$, the trivial weight $\lambda=0$ has the following weight diagram:
			$$
			
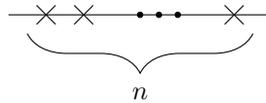
\begin{tikzpicture}
				\draw (0,0) -- (3.5,0);
				\draw (0.5,0) node[cross=4pt] {};
				\draw (1,0) node[cross=4pt] {};
				\filldraw[black] (1.75,0) circle (1pt); 
				\filldraw[black] (2,0) circle (1pt); 
				\filldraw[black] (2.25,0) circle (1pt); 
				\draw (3,0) node[cross=4pt] {};
				\draw [decorate,decoration={brace,amplitude=15pt, mirror}] (0.25,-0.25) -- (3.25,-0.25) node [black,midway,yshift=-0.8cm] {\footnotesize
					$n$}; 
			\end{tikzpicture}
			$$
			The arc diagram is given by
			$$
			
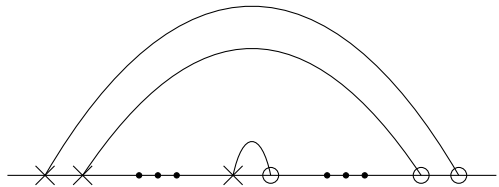
\begin{tikzpicture}
				\draw (0,0) -- (6.5,0);
				\draw (0.5,0) node[cross=4pt] {};
				\draw (1,0) node[cross=4pt] {};
				\filldraw[black] (1.75,0) circle (1pt); 
				\filldraw[black] (2,0) circle (1pt); 
				\filldraw[black] (2.25,0) circle (1pt); 
				\draw (3,0) node[cross=4pt] {};
				\draw (3.5,0) circle (3pt);
				\draw (3,0) .. controls (3.125,0.6) and (3.375,0.6) .. (3.5,0);
				\filldraw[black] (4.25,0) circle (1pt); 
				\filldraw[black] (4.5,0) circle (1pt); 
				\filldraw[black] (4.75,0) circle (1pt); 
				\draw (5.5,0) circle (3pt);
				\draw (1,0) .. controls (2.5,2.25) and (4,2.25) .. (5.5,0);
				\draw (6,0) circle (3pt);
				\draw (0.5,0) .. controls (2.25,3) and (4.25,3) .. (6,0);
			\end{tikzpicture}
			$$
			Clearly $\DS_x\CC=\CC$, and from the diagram we also see that when removing the only maximal arc we obtain the arc diagram of the trivial module for $\gg\ll(n-1|n-1)$.
		\end{example}
		
		\begin{example}\label{gl ds example}
			For $\gg\ll(6|7)$ consider the weight 
			\[
			\lambda=3\epsilon_1+3\epsilon_2+2\epsilon_3+\epsilon_4+\epsilon_5-2\delta_2-2\delta_3-2\delta_4-3\delta_5-3\delta_6-6\delta_7
			\]
			Its weight diagram looks as follows:
			$$
			\begin{tikzpicture}
				\draw (-3,0) -- (3,0);
				\draw (-2.5,0) node[cross=4pt] {};
				\draw (-2,0) circle (3pt);
				\draw (-1.5,0) node[label=center:{\large $<$}] {};
				\draw (-1,0) node[cross=4pt] {};
				\draw (-0.5,0) node[label=center:{\large $>$}] {};
				\draw (0,0) node[cross=4pt] {};
				\draw (0.5,0) circle (3pt);
				\draw (1,0) node[cross=4pt] {};
				\draw (1.5,0) node[cross=4pt] {};
				\draw (2,0) circle (3pt);
				\draw (2.5,0) node[label=center:{\large $>$}] {};
			\end{tikzpicture}
			$$
			Its arc diagram is given by:
			$$
			
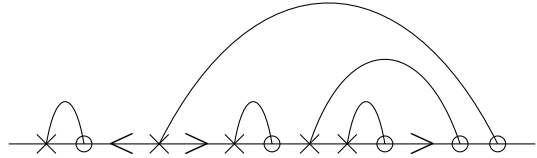
\begin{tikzpicture}
				\draw (-3,0) -- (4,0);
				\draw (-2.5,0) node[cross=4pt] {};
				\draw (-2,0) circle (3pt);
				\draw (-2.5,0) .. controls (-2.375,0.75) and (-2.125,0.75) .. (-2,0);
				\draw (-1.5,0) node[label=center:{\large $<$}] {};
				\draw (-1,0) node[cross=4pt] {};
				\draw (-0.5,0) node[label=center:{\large $>$}] {};
				\draw (0,0) node[cross=4pt] {};
				\draw (0.5,0) circle (3pt);
				\draw (0,0) .. controls (0.125,0.75) and (0.375,0.75) .. (0.5,0);
				\draw (1,0) node[cross=4pt] {};
				\draw (1.5,0) node[cross=4pt] {};
				\draw (2,0) circle (3pt);
				\draw (1.5,0) .. controls (1.625,0.75) and (1.875,0.75) .. (2,0);
				\draw (2.5,0) node[label=center:{\large $>$}] {};
				\draw (3,0) circle (3pt);
				\draw (1,0) .. controls (1.5,1.5) and (2.5,1.5) .. (3,0);
				\draw (3.5,0) circle (3pt);
				\draw (-1,0) .. controls (0.25,2.5) and (2.25,2.5) .. (3.5,0);
			\end{tikzpicture}
			$$
		\end{example}
		
		\subsubsection{$\operatorname{dex}$ and simple modules} For $\lambda\in\Lambda^+(\gg\ll(m|n))$, in order to properly specify the parity of $L(\lambda)$ we need to briefly explain the equivalences of blocks for $\gg\ll(m|n)$.  Namely, every block of atypicality $k$ for $\gg\ll(m|n)$ is equivalent to the principal block of $\gg\ll(k|k)$.  
		
		This equivalence defines a correspondence on simple modules, and thus on dominant weights, which we denote by $\lambda\mapsto \overline{\lambda}$, and it works as follows.
		
		In the weight diagram of $\lambda$, move all core symbols (i$.$e$.$ $>,<$) to the right of the symbols $\times$ by simply swapping adjacent symbols one at a time.  This pictorial procedure corresponds to applying translation functors between different blocks of the same atypicality.  After moving all core symbols to the right, we simply remove them from the diagram, leaving us with a diagram only with the symbols $\times$, and thus it will correspond to a dominant weight $\overline{\lambda}$ in the principal block of $\gg\ll(k|k)$.
		
		For example, for the simple module of Example~\ref{gl ds example}, the atypicality is 5 and the weight diagram of $\overline{\lambda}$ is given by
		$$
		\begin{tikzpicture}
			\draw (-3,0) -- (1,0);
			\draw (-2.5,0) node[cross=4pt] {};
			\draw (-2,0) circle (3pt);
			\draw (-1.5,0) node[cross=4pt] {};
			\draw (-1,0) node[cross=4pt] {};
			\draw (-0.5,0) circle (3pt);
			\draw (0,0) node[cross=4pt] {};
			\draw (0.5,0) node[cross=4pt] {};
		\end{tikzpicture}
		$$

		\begin{remark}\label{remark_eq_cat_ds}
			The equivalence of categories here described commutes with the application of $\DS$, and thus it in fact suffices to understand how $\DS$ acts on the principal block of $\gg\ll(k|k)$, although we will explain the general case for $\gg\ll(m|n)$.  However for $\mathfrak{osp}(m|2n)$ we will use this principal and thus only explain in full how $\DS$ acts on the principal blocks of certain superalgebras. \end{remark}
		
		\begin{definition}
			For $\lambda\in\Lambda^+(\gg\ll(m|n))$, we define 
			\[
			\operatorname{dex}\lambda:=p(\overline{\lambda}).
			\]
			Then we set $L(\lambda)$ to be the simple module of highest weight $\lambda$ such that the parity of the highest weight vector is $\operatorname{dex}\lambda$.
		\end{definition}
		
		\begin{example}
			Consider an integer multiple of the Berezinian weight of $\gg\ll(n|n)$, that is, for $k\in\mathbb{Z}$, 
			\[
			\lambda=k(\epsilon_1+\dots+\epsilon_n-\delta_1-\dots-\delta_n).
			\]
			Its weight diagram is a translation of the weight diagram of the trivial module.  
			We have $\operatorname{dex}(\lambda)=kn$ mod $2$.
		\end{example}
		
		\begin{example}
			For $\lambda$ as in Example~\ref{gl ds example}, we have $\operatorname{dex}\lambda=1$.
		\end{example}

		\begin{theorem}[\cite{HsW}] 
			For $\lambda\in\Lambda^+(\gg\ll(m|n))$, we have
			\[
			\DS_xL(\lambda)=\bigoplus_i\Pi^{n_i}L(\lambda_i)
			\]
			where $\lambda_i$ are the weights which correspond to the arc diagrams obtained by removing a single maximal arc from the arc diagram of $\lambda$, and $n_i=\operatorname{dex}\lambda-\operatorname{dex}\lambda_i$.  
		\end{theorem}

		\begin{remark}
			For $\gg\ll(m|n)$ there are two conjugacy classes of rank one odd root vectors, but as is explained in \cite{HsW} the action of the corresponding Duflo--Serganova functors on simple modules is the same up to isomorphism.
		\end{remark}
		
		\begin{example}
			
			For $\gg\ll(6|7)$ consider the weight introduced in Example~\ref{gl ds example}.   We recall its arc diagram is given by:
			$$
			
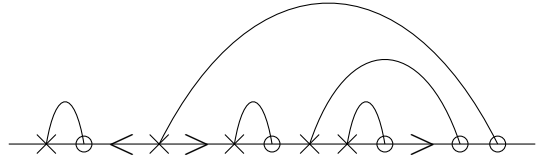
\begin{tikzpicture}
				\draw (-3,0) -- (4,0);
				\draw (-2.5,0) node[cross=4pt] {};
				\draw (-2,0) circle (3pt);
				\draw (-2.5,0) .. controls (-2.375,0.75) and (-2.125,0.75) .. (-2,0);
				\draw (-1.5,0) node[label=center:{\large $<$}] {};
				\draw (-1,0) node[cross=4pt] {};
				\draw (-0.5,0) node[label=center:{\large $>$}] {};
				\draw (0,0) node[cross=4pt] {};
				\draw (0.5,0) circle (3pt);
				\draw (0,0) .. controls (0.125,0.75) and (0.375,0.75) .. (0.5,0);
				\draw (1,0) node[cross=4pt] {};
				\draw (1.5,0) node[cross=4pt] {};
				\draw (2,0) circle (3pt);
				\draw (1.5,0) .. controls (1.625,0.75) and (1.875,0.75) .. (2,0);
				\draw (2.5,0) node[label=center:{\large $>$}] {};
				\draw (3,0) circle (3pt);
				\draw (1,0) .. controls (1.5,1.5) and (2.5,1.5) .. (3,0);
				\draw (3.5,0) circle (3pt);
				\draw (-1,0) .. controls (0.25,2.5) and (2.25,2.5) .. (3.5,0);
			\end{tikzpicture}
			$$
			To apply $\DS_x$ to $L(\lambda)$ we remove the maximal arcs from the diagram to obtain two new arc diagrams:
			
			$$
			
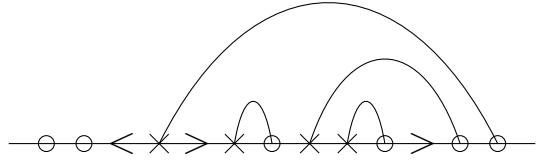
\begin{tikzpicture}
				\draw (-3,0) -- (4,0);
				\draw (-2.5,0) circle (3pt);
				\draw (-2,0) circle (3pt);
				\draw (-1.5,0) node[label=center:{\large $<$}] {};
				\draw (-1,0) node[cross=4pt] {};
				\draw (-0.5,0) node[label=center:{\large $>$}] {};
				\draw (0,0) node[cross=4pt] {};
				\draw (0.5,0) circle (3pt);
				\draw (0,0) .. controls (0.125,0.75) and (0.375,0.75) .. (0.5,0);
				\draw (1,0) node[cross=4pt] {};
				\draw (1.5,0) node[cross=4pt] {};
				\draw (2,0) circle (3pt);
				\draw (1.5,0) .. controls (1.625,0.75) and (1.875,0.75) .. (2,0);
				\draw (2.5,0) node[label=center:{\large $>$}] {};
				\draw (3,0) circle (3pt);
				\draw (1,0) .. controls (1.5,1.5) and (2.5,1.5) .. (3,0);
				\draw (3.5,0) circle (3pt);
				\draw (-1,0) .. controls (0.25,2.5) and (2.25,2.5) .. (3.5,0);
			\end{tikzpicture}
			$$
			which corresponds to the weight $\lambda_1=3\epsilon_1+3\epsilon_2+2\epsilon_3+\epsilon_4+\epsilon_5-2\delta_1-2\delta_2-2\delta_3-3\delta_4-3\delta_5-4\delta_6$, which has $\operatorname{dex}\lambda_1=0$.
			$$
			
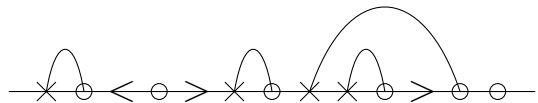
\begin{tikzpicture}
				\draw (-3,0) -- (4,0);
				\draw (-2.5,0) node[cross=4pt] {};
				\draw (-2,0) circle (3pt);
				\draw (-2.5,0) .. controls (-2.375,0.75) and (-2.125,0.75) .. (-2,0);
				\draw (-1.5,0) node[label=center:{\large $<$}] {};
				\draw (-1,0) circle (3pt);
				\draw (-0.5,0) node[label=center:{\large $>$}] {};
				\draw (0,0) node[cross=4pt] {};
				\draw (0.5,0) circle (3pt);
				\draw (0,0) .. controls (0.125,0.75) and (0.375,0.75) .. (0.5,0);
				\draw (1,0) node[cross=4pt] {};
				\draw (1.5,0) node[cross=4pt] {};
				\draw (2,0) circle (3pt);
				\draw (1.5,0) .. controls (1.625,0.75) and (1.875,0.75) .. (2,0);
				\draw (2.5,0) node[label=center:{\large $>$}] {};
				\draw (3,0) circle (3pt);
				\draw (1,0) .. controls (1.5,1.5) and (2.5,1.5) .. (3,0);
				\draw (3.5,0) circle (3pt);
			\end{tikzpicture}
			$$
			which corresponds to the weight $\lambda_2=3\epsilon_1+3\epsilon_2+2\epsilon_3-\epsilon_5+\delta_1-2\delta_2-2\delta_3-3\delta_4-3\delta_5-6\delta_6$, which has $\operatorname{dex}\lambda_2=0$  Thus we have 
			\[
			\DS_xL(\lambda)=L(\lambda_1)\oplus L(\lambda_2).
			\]
		\end{example}

		\subsection{$\mathfrak{osp}(m|2n)$ case}

		
		For a full explanation of the $\mathfrak{osp}$ case with many examples, see \cite{GH1}.  Below we closely follow the treatment given in \cite{G4}. 
		
		We have the following equivalences of categories that are obtained via a functor which respects the action of $\DS$.
		
		\begin{itemize}
			\item A block of atypicality $k$ for $\mathfrak{osp}(2m+1|2n)$ is equivalent to the principal block for $\mathfrak{osp}(2k+1|2k)$. 
			\item For $\mathfrak{osp}(2m|2n)$ with $m,n>0$, a block of atypicality $k$ is equivalent to the principal block of either $\mathfrak{osp}(2k|2k)$ or $\mathfrak{osp}(2k+2|2k)$. 
		\end{itemize}  
		In this way we obtain that every block for $\mathfrak{osp}(m|2n)$ is equivalent to the principal block of $\mathfrak{osp}(2k+t|2k)$ for some $k$ and some $t=0,1,$ or $2$, and so it suffices to understand how $\DS_x$ acts on modules in these blocks.  We will deal with these three cases individually, and refer to them according to the value of $t$.
		
		\begin{remark}
			There are a number of parallels between the principal blocks of $\mathfrak{osp}(2k+1|2k)$ and $\mathfrak{osp}(2k+2|2k)$.  In particular in \cite{GH1} they find an explicit bijection $\tau$ between simple modules such that it respects the action of the $\DS$ functor, meaning we have an equality of multiplicity numbers $[\DS_x(\tau(L)):\tau(L)]=[\DS_xL:L]$. 
			
			It is important open question whether there is an equivalence of categories between these principal blocks, and in particular if there is one which commutes with the $\DS$ functor.
		\end{remark}
		
		\subsubsection{}\label{mix}  The weight lattice of $\mathfrak{osp}(2k+t|2k)$ is given by $\Lambda_{k+\ell|k}$, where we set $\ell=0$ for $t=0,1$ and $\ell=1$ for $t=2$. We fix triangular decompositions
		corresponding to the ``mixed'' bases:
		$$\Sigma:=\left\{\begin{array}{ll}
			\varepsilon_1-\delta_{1},\delta_1-\varepsilon_2,\ldots,\varepsilon_{k}-\delta_{k},
			\delta_k\ & \text{ for }\mathfrak{osp}(2k+1|2k)\\
			\delta_{1}-\varepsilon_1,\varepsilon_1-\delta_{2},\ldots,\varepsilon_{k-1}-\delta_{k},
			\delta_k\pm\varepsilon_k & \text{ for } \mathfrak{osp}(2k|2k)\\
			\varepsilon_{1}-\delta_1,\delta_1-\varepsilon_{2},\ldots,\varepsilon_{k}-\delta_k,\delta_k\pm\varepsilon_{k+1} & \text{ for }
			\mathfrak{osp}(2k+2|2k).\\
		\end{array}\right.$$
		We have $\rho=0$ for $t=0,2$ and $\rho=\displaystyle\frac{1}{2}\sum_{i=1}^k(\delta_i-\varepsilon_i)$ for $t=1$. 
		
		\subsubsection{Highest weights in the principal block}\label{lambda+rho}
		
		For $\lambda\in\Lambda_{k+\ell|k}$ we set
		$$a_i:=-(\lambda|\delta_i)$$.
		Write $\Lambda^0(\mathfrak{osp}(2k+t|2k))$ for the dominant weights of $\mathfrak{osp}(2k+t|2k)$ which lie in the principal block.  By~\cite{GrS1},  $\lambda\in\Lambda^0(\mathfrak{osp}(2k+t|2k))$ if and only if $a_1,\ldots,a_k$
		are non-negative integers with $a_{i+1}>a_i$ or $a_i=a_{i+1}=0$, and
		
		$$\lambda+\rho=\left\{
		\begin{array}{lll}
			\sum_{i=1}^{k-1} a_i(\varepsilon_i+\delta_i)+a_k(\delta_k+\xi \varepsilon_k) &\text{ for } & t=0\\
			\sum_{i=1}^{k} a_i(\varepsilon_i+\delta_i)&\text{ for } & t=2\\
			\sum_{i=1}^{s-1} (a_i+\frac{1}{2})
			(\varepsilon_i+\delta_i)+\frac{1}{2}(\delta_s+\xi  \varepsilon_s)+\sum_{i=s+1}^{k} \frac{1}{2}
			(\delta_i-\varepsilon_i)&\text{ for } & t=1\end{array}
		\right.$$
		
		for $\xi\in\{\pm 1\}$. For $t=1$ we have $1\leq s\leq k+1$
		and $a_s=a_{s+1}=\ldots=a_k=0$ if  $s\leq k$ (for $s=k+1$ we have
		$\lambda+\rho=\sum_{i=1}^{k} (a_i+\frac{1}{2})
		(\varepsilon_i+\delta_i)$).

		\subsubsection{Weight diagrams}\label{wtdiag}
		Take $\lambda\in\Lambda^0(\mathfrak{osp}(2k+t|2k))$ and define $a_i$ for $i=1,\ldots,k$ as above.
		We assign to $\lambda$ a weight diagram $f_{\lambda}$, which is
		a number line with one or several symbols drawn at each position with non-negative integral coordinate:
		
		\begin{itemize}
			\item we put the sign $\times$ at each  position with the coordinate $a_i$;
			\item for $t=2$ we add $>$ at the zero position;
			\item we add the ``empty symbol'' $\circ$ to all empty positions.
		\end{itemize}
		
		For $t\neq2$ a weight $\lambda\in\Lambda^0(\mathfrak{osp}(2k+t|2k))$ is not uniquely determined by the weight diagram constructed by the above procedure. Therefore, for  $t=0$ with $a_k\neq0$ 
		and for $t=1$ with $s\leq k$, we write the 
		sign of $\xi$ before the diagram  ($+$ if $\xi=1$ and $-$ if $\xi=-1$).
		
		Notice that each position with a nonzero coordinate contains either $\times$ or $\circ$. For $t=0,1$ 
		the zero position is occupied either by $\circ$ or by
		several symbols $\times$; we
		write this as $\times^i$ for $i\geq 0$. Similarly, for
		$t=2$ the zero position is occupied by $\overset{\times^i}{>}$ with $i\geq 0$.
		
		\begin{remark}
			The weight diagrams we have defined are essentially the same as those defined in \cite{GrS1}, except that when $t=1$ we shift by $-1/2$.
		\end{remark}
		
		\subsubsection{Examples}
		The weight diagram of $0$ is 
		$$
		\begin{tikzpicture}
			\draw (0,0) -- (2.5,0);
			\draw (0.5,0.07) node[label=center:{\large $\times^k$}] {};
			\draw (1,0)  circle (3pt);
			\filldraw[black] (1.5,0) circle (1pt); 
			\filldraw[black] (1.75,0) circle (1pt); 
			\filldraw[black] (2,0) circle (1pt); \end{tikzpicture}
		$$
		where the three small $\bullet$s together are an ellipsis, indicating the diagram continues with $\circ$s.  We omit these in future diagrams.  For $t=0$, 
		$$
		\begin{tikzpicture}
			\draw (-0.5,0) node[label=center:{\large $-$}] {};
			\draw (0,0) -- (2,0);
			\draw (0.5,0.07) node[label=center:{\large $\times^k$}] {};
			\draw (1,0)  circle (3pt);
			\draw (1.5,0)  circle (3pt);
		\end{tikzpicture}
		$$
		for $t=1$, and 
		$$
		\begin{tikzpicture}
			\draw (0,0) -- (2,0);
			\draw (0.5,0.24) node[label=center:{\large $\overset{\times^k}{>}$}] {};
			\draw (1,0)  circle (3pt);
			\draw (1.5,0)  circle (3pt);
		\end{tikzpicture}
		$$
		for $t=2$.
		
		The diagram 
		$$
		\begin{tikzpicture}
			\draw (-0.5,0) node[label=center:{\large +}] {};
			\draw (0,0) -- (2,0);
			\draw (0.5,0.07) node[label=center:{\large $\times^k$}] {};
			\draw (1,0)  circle (3pt);
			\draw (1.5,0)  circle(3pt);
		\end{tikzpicture}
		$$
		corresponds to $\mathfrak{osp}(2k+1|2k)$-weight $\lambda=\varepsilon_1$.
		
		The diagram 
		$$
		\begin{tikzpicture}
			\draw (-0.5,0) node[label=center:{\large +}] {};
			\draw (0,0) -- (2,0);
			\draw (0.5,0) circle (3pt);
			\draw (1,0)  node[cross=4pt] {};
			\draw (1.5,0)  node[cross=4pt] {};
		\end{tikzpicture}
		$$
		corresponds to the $\mathfrak{osp}(4|4)$-weight $\lambda=\lambda+\rho=(\varepsilon_2+\delta_2)+2(\varepsilon_1+\delta_1)$.

		The empty diagram correspond to $\mathfrak{osp}(0|0)=\mathfrak{osp}(1|0)=0$; the diagram
		$>$ 
		corresponds to the weight $0$  for $\mathfrak{osp}(2|0)=\CC$.
		
		\subsubsection{}\label{diag}
		The definition of weight diagrams defines a one to one correspondence between dominant weights for $\mathfrak{osp}(2k+t|2k)$ and (sometime signed) weight diagrams with $n$ $\times$ symbols and one symbol $>$ if $t=2$, satisfying certain conditions.
		
		For  $t=0$ (respectively, $t=1$) a  diagram $f_{\lambda}$ in  has a sign
		if and only if $f_{\lambda}(0)=\circ$ (respectively,  $f_{\lambda}(0)\not=\circ$).
		
		
		



		\subsubsection{Arc diagrams}\label{subsec_arc_diag_osp}
		
		We associate an arc diagram to each weight diagram constructed according to the following steps:
		\begin{enumerate}
			\item For $0<i<j$ with $f_{\lambda}(i)=\times$ and $f_{\lambda}(j)=\circ$, connect an arc from $i$ to $j$ if for all $k$ with $i<k<j$ and $f_{\lambda}(k)$, $k$ already lies on an arc.  
			\item If there is at least one $\times$ at 0, order them from top to bottom.  If $t=0,1$ then draw a single arc from the bottom $\times$ to the nearest position with $\circ$.  If $t=2$, draw two arcs emanating from the bottom $\times$ to the two nearest positions with $\circ$ not already lying on an arc.  Then for any $t$, for each $\times$ at 0 above the bottom one (and working from bottom to top), draw two arcs from the $\times$ to the two nearest positions $\circ$ not already on an arc.
		\end{enumerate}
		
		In what follows we refer to the arcs (either one or two) which lie on a single $\times$ as just an arc.  For example, consider the arc diagram associated to the weight $\lambda=9\epsilon_1+8\epsilon_2+4\epsilon_3+\epsilon_4+8\delta_1+7\delta_2+3\delta_3$ for $\mathfrak{osp}(11|10)$; in this case $\lambda+\rho=(8+1/2)(\epsilon_1+\delta_1)+(7+1/2)(\epsilon_2+\delta_2)+(3+1/2)(\epsilon_3+\delta_3)+1/2(\epsilon_4+\delta_4)+1/2(\delta_5-\epsilon_5)$:
		$$
		
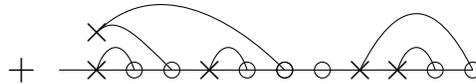
\begin{tikzpicture}
			\draw (-0.5,0) node[label=center:{\large +}] {};
			\draw (0,0) -- (5.5,0);
			\draw (0.5,0) node[label=center:{\large $\times$}] {};
			\draw (0.5,0.5) node[label=center:{\large $\times$}] {};
			\draw (1,0)  circle (3pt);
			\draw (0.5,0) .. controls (0.625,0.4) and (0.875,0.4) .. (1,0);
			\draw (1.5,0)  circle(3pt);
			\draw (0.5,0.5) .. controls (0.75, 0.75) and (1,0.5) .. (1.5,0);
			\draw (2,0) node[label=center:{\large $\times$}] {};
			\draw (2.5,0) circle (3pt);
			\draw (2,0) .. controls (2.125,0.4) and (2.375,0.4) .. (2.5,0);
			\draw (3,0) circle (3pt);
			\draw (0.5,0.5) .. controls (1.25,1.25) and (2.25,0.75) .. (3,0);
			\draw (3,0) circle (3pt);
			\draw (3.5,0) circle (3pt);
			\draw (4,0) node[label=center:{\large $\times$}] {};
			\draw (4.5,0) node[label=center:{\large $\times$}] {};
			\draw (5,0) circle (3pt);
			\draw (4.5,0) .. controls (4.625,0.4) and (4.875,0.4) .. (5,0);
			\draw (5.5,0) circle (3pt);
			\draw (4,0) .. controls (4.5,1) and (5,1) .. (5.5,0);
		\end{tikzpicture}
		$$
		There are two maximal arcs in the above diagram: one which consists of the two arcs emanating from the top $\times$ at 0.  If we remove this top arc we obtain the diagram corresponding to the weight $\lambda_1=9\epsilon_1+8\epsilon_2+4\epsilon_3+\epsilon_4+8\delta_1+7\delta_2+3\delta_3$ for $\mathfrak{osp}(9|8)$:
		$$
		
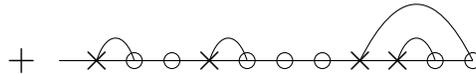
\begin{tikzpicture}
			\draw (-0.5,0) node[label=center:{\large +}] {};
			\draw (0,0) -- (5.5,0);
			\draw (0.5,0) node[label=center:{\large $\times$}] {};
			\draw (1,0)  circle (3pt);
			\draw (0.5,0) .. controls (0.625,0.4) and (0.875,0.4) .. (1,0);
			\draw (1.5,0)  circle(3pt);
			\draw (2,0) node[label=center:{\large $\times$}] {};
			\draw (2.5,0) circle (3pt);
			\draw (2,0) .. controls (2.125,0.4) and (2.375,0.4) .. (2.5,0);
			\draw (3,0) circle (3pt);
			\draw (3.5,0) circle (3pt);
			\draw (4,0) node[label=center:{\large $\times$}] {};
			\draw (4.5,0) node[label=center:{\large $\times$}] {};
			\draw (5,0) circle (3pt);
			\draw (4.5,0) .. controls (4.625,0.4) and (4.875,0.4) .. (5,0);
			\draw (5.5,0) circle (3pt);
			\draw (4,0) .. controls (4.5,1) and (5,1) .. (5.5,0);
		\end{tikzpicture}
		$$
		The other maximal arc is the one emanating from the $\times$ lying at 7.  If we remove it, we obtain the diagram corresponding to the weigh $\lambda_2=9\epsilon_1+4\epsilon_2+\epsilon_3+8\delta_1+3\delta_2$ of $\mathfrak{osp}(9|8)$:
		$$
		
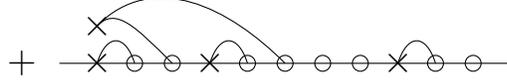
\begin{tikzpicture}
			\draw (-0.5,0) node[label=center:{\large +}] {};
			\draw (0,0) -- (6,0);
			\draw (0.5,0) node[label=center:{\large $\times$}] {};
			\draw (0.5,0.5) node[label=center:{\large $\times$}] {};
			\draw (1,0)  circle (3pt);
			\draw (0.5,0) .. controls (0.625,0.4) and (0.875,0.4) .. (1,0);
			\draw (1.5,0)  circle(3pt);
			\draw (0.5,0.5) .. controls (0.75, 0.75) and (1,0.5) .. (1.5,0);
			\draw (2,0) node[label=center:{\large $\times$}] {};
			\draw (2.5,0) circle (3pt);
			\draw (2,0) .. controls (2.125,0.4) and (2.375,0.4) .. (2.5,0);
			\draw (3,0) circle (3pt);
			\draw (0.5,0.5) .. controls (1.25,1.25) and (2.25,0.75) .. (3,0); 
			\draw (3.5,0) circle (3pt);
			\draw (4,0) circle (3pt);
			\draw (4.5,0) node[label=center:{\large $\times$}] {};
			\draw (5,0) circle (3pt);
			\draw (4.5,0) .. controls (4.625,0.4) and (4.875,0.4) .. (5,0);
			\draw (5.5,0) circle (3pt);
		\end{tikzpicture}
		$$
		
		\subsubsection{$\operatorname{dex}$ and simple modules} Given a dominant weight $\lambda$, we define
		\[
		\|\lambda\|=\sum\limits_{i=1}^{k}a_i-\ell(k-\text{tail}(\lambda))
		\]
		where $\text{tail}(\lambda)$ denotes the number of symbols $\times$ at 0 when $\ell=1$ (we omit the definition for other cases since we do not need it).  The we define
		\[
		\operatorname{dex}\lambda:=\|\lambda\| \text{ mod }2.
		\]

		We will write $L(\lambda)$ for the simple module of highest weight $\lambda$, where the parity of the highest weight vector is given by $\operatorname{dex}(\lambda)$.
		
		\begin{theorem}[\cite{GH1}, Thm. 8.2] \label{GH theorem}\
			
			\begin{itemize}
				\item[(i)] Let $\lambda\in\Lambda^0(\mathfrak{osp}(2k+t|2k)))$ and $\nu\in\Lambda^0(\mathfrak{osp}(2k+t-2|2k-2)$.  Then $[\DS_xL(\lambda):L(\nu)]_{non}\neq0$ if and only if the arc diagram of $\nu$ can be obtained from the arc diagram of $\lambda$ by removing a maximal arc.  If $t\neq 1$, then the sign of $\nu$ and $\lambda$ (if relevant) need not agree, while if $t=1$ then if $\nu$ has sign it must agree with the sign of $\lambda$.
				\item[(ii)] Let $e$ denote the number of free positions (i.e.,~those with $\circ$ and not attached to any arc) to the left of the maximal arc removed to obtain $\nu$.  For $t=1,2$ we have:
				\[
				[\DS_x(L(\lambda)):L(\nu)]=\left\{ \begin{array}{ll} (1|0) & e=0; \\  (2|0)  & e>0 \text{ and even}; \\ (0|2) & e\text{ odd} . \end{array}\right.
				\]
				For $t=0$ we have
				\[
				[\DS_x(L(\lambda)):L(\nu)]=\left\{ \begin{array}{ll} (1|0) & e\text{ even}; \\  (0|1)  & e\text{ odd}. \end{array}\right.
				\]
			\end{itemize}
		\end{theorem}
		
		\begin{remark}
			For every $m,n$ with $m>0$, the Lie superalgebra $\mathfrak{osp}(2m|2n)$ admits an involution $\sigma$ which comes from a reflection of its Kac-Dynkin diagram.  This involution is the same as the involution $\sigma_x$ defined in Section~\ref{involutions section}.
			
			In \cite{GH1} it is proven via a general argument that for a simple $\mathfrak{osp}(2m|2n)$-module $L$, we have $\DS_x(L^\sigma)\cong \DS_xL$, and $(\DS_xL)^{\sigma_x}\cong \DS_xL$, where $\sigma_x$ is the corresponding involution for $\mathfrak{osp}(2m-2|2n-2)$.
			
			It would be interesting to know if these isomorphisms hold for all finite-dimensional $\mathfrak{osp}(2m|2n)$-modules.
		\end{remark}
		
		We now give an example of the above theorem for each case $t=0,1,2$:
		\begin{example}
			For $t=1$, consider the weight $\lambda=9\epsilon_1+8\epsilon_2+4\epsilon_3+\epsilon_4+8\delta_1+7\delta_2+2\delta_3$.  We have $\operatorname{dex}\lambda=0$, and we looked at the weight diagram already in Section~\ref{subsec_arc_diag_osp}.  By  Theorem~\ref{GH theorem} we have
			\[
			\DS_xL(\lambda)=L(\lambda_{1})\oplus \Pi L(\lambda_2)^{\oplus 2}.
			\]
			Note for $\lambda_1$ we have $e=0$ and $\lambda_2$, $e=1$, hence the parities are as shown.
		\end{example}
		
		\begin{example}
			For $t=0$ consider the weight $\lambda=6(\delta_1+\epsilon_1)+2(\delta_2+\epsilon_2)+(\delta_3-\epsilon_3)$ for $\mathfrak{osp}(6|6)$, which has $\operatorname{dex}\lambda=1$ and arc diagram:
			$$
			
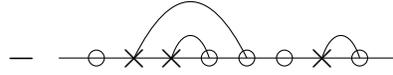
\begin{tikzpicture}
				\draw (-0.5,0) node[label=center:{\large $-$}] {};
				\draw (0,0) -- (4.5,0);
				\draw (0.5,0)  circle (3pt);
				\draw (1,0)  node[label=center:{\large $\times$}] {};
				\draw (1.5,0) node[label=center:{\large $\times$}] {};
				\draw (2,0) circle (3pt);
				\draw (1.5,0) .. controls (1.625,0.4) and (1.875,0.4) .. (2,0);
				\draw (2.5,0) circle (3pt);
				\draw (1,0) .. controls (1.5,1) and (2,1) .. (2.5,0);
				\draw (3,0) circle (3pt);
				\draw (3.5,0) node[label=center:{\large $\times$}] {};
				\draw (4,0) circle (3pt);
				\draw (3.5,0) .. controls (3.625,0.4) and (3.875,0.4) .. (4,0);
			\end{tikzpicture}
			$$
			There are two maximal arcs.  Removing the arc starting at position one gives the arc diagram associated to $\lambda_1^{\pm}=6(\delta_1+\epsilon_1)+2(\delta_2\pm\epsilon_2)$ with $e=1$:
			$$
			\begin{tikzpicture}
				\draw (-0.5,0) node[label=center:{\large $\pm$}] {};
				\draw (0,0) -- (4.5,0);
				\draw (0.5,0)  circle (3pt);
				\draw (1,0)  circle (3pt);
				\draw (1.5,0) node[label=center:{\large $\times$}] {};
				\draw (2,0) circle (3pt);
				\draw (1.5,0) .. controls (1.625,0.4) and (1.875,0.4) .. (2,0);
				\draw (2.5,0) circle (3pt);
				\draw (3,0) circle (3pt);
				\draw (3.5,0) node[label=center:{\large $\times$}] {};
				\draw (4,0) circle (3pt);
				\draw (3.5,0) .. controls (3.625,0.4) and (3.875,0.4) .. (4,0);
			\end{tikzpicture}
			$$
			Removing the maximal arc starting at position 6 gives the arc diagrams associated to $\lambda_2^{\pm}=2(\delta_1+\epsilon_1)+(\delta_2\pm\epsilon_2)$ with $e=2$:
			$$
			
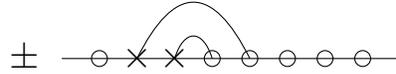
\begin{tikzpicture}
				\draw (-0.5,0) node[label=center:{\large $\pm$}] {};
				\draw (0,0) -- (4.5,0);
				\draw (0.5,0)  circle (3pt);
				\draw (1,0)  node[label=center:{\large $\times$}] {};
				\draw (1.5,0) node[label=center:{\large $\times$}] {};
				\draw (2,0) circle (3pt);
				\draw (1.5,0) .. controls (1.625,0.4) and (1.875,0.4) .. (2,0);
				\draw (2.5,0) circle (3pt);
				\draw (1,0) .. controls (1.5,1) and (2,1) .. (2.5,0);
				\draw (3,0) circle (3pt);
				\draw (3.5,0) circle (3pt);
				\draw (4,0) circle (3pt);
			\end{tikzpicture}
			$$
			Thus we have that
			\[
			\DS_xL(\lambda)=\Pi L(\lambda_1^+)\oplus \Pi L(\lambda_1^-)\oplus L(\lambda_2^+)\oplus L(\lambda_2^-).
			\]
		\end{example}
		
		\begin{example}
			In the $t=2$ case consider the weight $\lambda=8(\epsilon_1+\delta_1)+5(\epsilon_2+\delta_2)+8(\epsilon_3+\delta_3)$ for $\mathfrak{osp}(6|4)$ with $\operatorname{dex}\lambda=0$ and arc diagram:
			$$
			
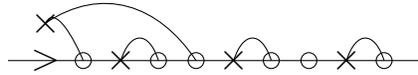
\begin{tikzpicture}
				\draw (0,0) -- (5.5,0);
				\draw (0.5,0) node[label=center:{\large $>$}] {};
				\draw (0.5,0.5) node[label=center:{\large $\times$}] {};
				\draw (1,0)  circle (3pt);
				\draw (0.5,0.5) .. controls (0.625,0.75) and (0.875,0.3) .. (1,0);
				\draw (1.5,0) node[label=center:{\large $\times$}] {};
				\draw (2,0) circle (3pt);
				\draw (1.5,0) .. controls (1.625,0.4) and (1.875,0.4) .. (2,0);
				\draw (2.5,0) circle (3pt);
				\draw (0.5,0.5) .. controls (1.25,1) and (2,0.75)  .. (2.5,0);
				\draw (3,0) node[label=center:{\large $\times$}] {};
				\draw (3.5,0) circle (3pt);
				\draw (3,0) .. controls (3.125,0.4) and (3.375,0.4) .. (3.5,0);
				\draw (4,0) circle (3pt);
				\draw (4.5,0) node[label=center:{\large $\times$}] {};
				\draw (5,0) circle (3pt);
				\draw (4.5,0) .. controls (4.625,0.4) and (4.875,0.4) .. (5,0);
			\end{tikzpicture}
			$$
			There are three maximal arcs.  If we remove the one emanating from 0, we obtain the arc diagram associated to the weight $\lambda_1=8(\epsilon_1+\delta_1)+5(\epsilon_2+\delta_2)+2(\epsilon_3+\delta_3)$ with $e=0$:
			$$
			\begin{tikzpicture}
				\draw (0,0) -- (5.5,0);
				\draw (0.5,0) node[label=center:{\large $>$}] {};
				\draw (1,0)  circle (3pt);
				\draw (1.5,0) node[label=center:{\large $\times$}] {};
				\draw (2,0) circle (3pt);
				\draw (1.5,0) .. controls (1.625,0.4) and (1.875,0.4) .. (2,0);
				\draw (2.5,0) circle (3pt);
				\draw (3,0) node[label=center:{\large $\times$}] {};
				\draw (3.5,0) circle (3pt);
				\draw (3,0) .. controls (3.125,0.4) and (3.375,0.4) .. (3.5,0);
				\draw (4,0) circle (3pt);
				\draw (4.5,0) node[label=center:{\large $\times$}] {};
				\draw (5,0) circle (3pt);
				\draw (4.5,0) .. controls (4.625,0.4) and (4.875,0.4) .. (5,0);
			\end{tikzpicture}
			$$
			If we remove the maximal arc emanating from 5 we obtain the arc diagram associated to the weight $\lambda_2=8(\epsilon_1+\delta_1)+2(\epsilon_2+\delta_2)$ with $e=0$:
			$$
			\begin{tikzpicture}
				\draw (0,0) -- (5.5,0);
				\draw (0.5,0) node[label=center:{\large $>$}] {};
				\draw (0.5,0.5) node[label=center:{\large $\times$}] {};
				\draw (1,0)  circle (3pt);
				\draw (0.5,0.5) .. controls (0.625,0.75) and (0.875,0.3) .. (1,0);
				\draw (1.5,0) node[label=center:{\large $\times$}] {};
				\draw (2,0) circle (3pt);
				\draw (1.5,0) .. controls (1.625,0.4) and (1.875,0.4) .. (2,0);
				\draw (2.5,0) circle (3pt);
				\draw (0.5,0.5) .. controls (1.25,1) and (2,0.75)  .. (2.5,0);
				\draw (3,0) circle (3pt);
				\draw (3.5,0) circle (3pt);
				\draw (4,0) circle (3pt);
				\draw (4.5,0) node[label=center:{\large $\times$}] {};
				\draw (5,0) circle (3pt);
				\draw (4.5,0) .. controls (4.625,0.4) and (4.875,0.4) .. (5,0);
			\end{tikzpicture}
			$$
			Finally if we remove the arc emanating from 8 we obtain the arc diagram associated to the weight $\lambda_2=5(\epsilon_1+\delta_1)+2(\epsilon_2+\delta_2)$ with $e=1$:
			$$
			
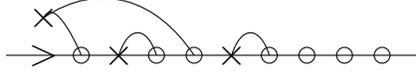
\begin{tikzpicture}
				\draw (0,0) -- (5.5,0);
				\draw (0.5,0) node[label=center:{\large $>$}] {};
				\draw (0.5,0.5) node[label=center:{\large $\times$}] {};
				\draw (1,0)  circle (3pt);
				\draw (0.5,0.5) .. controls (0.625,0.75) and (0.875,0.3) .. (1,0);
				\draw (1.5,0) node[label=center:{\large $\times$}] {};
				\draw (2,0) circle (3pt);
				\draw (1.5,0) .. controls (1.625,0.4) and (1.875,0.4) .. (2,0);
				\draw (2.5,0) circle (3pt);
				\draw (0.5,0.5) .. controls (1.25,1) and (2,0.75)  .. (2.5,0);
				\draw (3,0) node[label=center:{\large $\times$}] {};
				\draw (3.5,0) circle (3pt);
				\draw (3,0) .. controls (3.125,0.4) and (3.375,0.4) .. (3.5,0);
				\draw (4,0) circle (3pt);
				\draw (4.5,0) circle (3pt);
				\draw (5,0) circle (3pt);
			\end{tikzpicture}
			$$
			It follows that we have
			\[
			\DS_xL(\lambda)=L(\lambda_1)\oplus L(\lambda_2)\oplus \Pi L(\lambda_3)^{\oplus 2}.
			\]
		\end{example}

		\subsection{$\pp(n)$ case}
		
		We refer to \cite{ES2} for a full explanation of the $\pp(n)$ case with examples.
		
		We write $\Lambda_n$ for the $\ZZ$-module spanned by $\epsilon_1,\dots,\epsilon_n$.  We fix the following simple roots for $\pp(n)$:
		\[
		\pm\epsilon_n-\epsilon_{n-1},\epsilon_{n-1}-\epsilon_{n-2},\dots,\epsilon_2-\epsilon_1.
		\]
		Write $\Lambda^+(\pp(n))$ for the set of dominant integral weights with respect to the corresponding Borel subalgebra. Let 
		\[
		\rho=\epsilon_2+2\epsilon_3+\dots+(n-1)\epsilon_n.
		\]
		For $\lambda\in\Lambda^{\pp(n)}$ we write
		\[
		\lambda+\rho=a_1\epsilon+\dots+a_n\epsilon_n.
		\]
		Then the dominance condition is precisely that $a_1<\cdots<a_n$.  
		
		Given $\lambda\in\Lambda^+(\pp(n))$ we write $L(\lambda)$ for the irreducible representation corresponding to $\lambda$ such that the highest weight vector is even.

		\subsubsection{Weight and arc diagrams} To $\lambda\in\Lambda^+(\pp(n))$ we define the associated weight diagram $f_{\lambda}:\ZZ\to\{\circ,\bullet\}$ by $f_{\lambda}(a_i)=\bullet$ and $f_{\lambda}(n)=\circ$ if $n\neq a_i$ for all $i$.
		
		Now we define an arc diagram associated to $\lambda$ according to the same approach as for $\gg\ll(m|n)$, except we work from left to right now, i.e.,~from negative to positive integers. More explicitly, we draw an arc connecting $i$ and $j$ if $f_{\lambda}(i)=\circ$, $f_{\lambda}(j)=\bullet$ and all $k$ with $i<k<j$ already lie on an arc.
		
		\begin{theorem}[\cite{ES2}]
			Let $\lambda\in\Lambda^+(\pp(n))$ and $\mu\in\Lambda^+(\pp(n-1))$.  Then $\Pi^zL(\mu)$ appears as a factor  of $\DS_xL(\lambda)$ for some $z$ if and only if the arc diagram of $\mu$ can be obtained from the arc diagram of $\lambda$ by removing a maximal arc.   In this case, $z$ is equal to the number of arcs to the right of the one removed. Further, $\DS_xL(\lambda)$ is multiplicity-free.
		\end{theorem}
		
		We can now give the proof of Theorem~\ref{thm2} for $\pp(n)$. 
		\begin{corollary}\label{p case thm 8.1}
			The map $ds^1:\Gr_-(\pp(n))\to\Gr_-(\pp(n-1))$ is surjective.
		\end{corollary}
		
		\begin{proof}
			In particular $DS_xL$ is multiplicity-free for a simple $\gg$-module $L$, and the factors are obtained by removing maximal arcs.  So let $L'$ be a simple $\gg_x$-module, with arc diagram $f'$.  Let $f_1$ be the arc diagram obtained by adding a symbol $\bullet$ to the first free space to the right of all symbols $\bullet$ of $f'$.  Then this new symbol will give a maximal arc in $f_1$. If we write $L_1$ for the irreducible representation corresponding to $f_1$, then in $\Gr_-(\pp(n-1))$ we will have a multiplicity-free sum
			\[
			ds_x[L_1]=\pm[L']+\sum_{L''} \pm[L''].
			\] 
			By switching the parity of $L_1$, we can assume the sign in front of $[L']$ is positive.  Now we can induct on the length of the furthest most right string of symbols $\bullet$ in the arc diagram, for which all $L''$s will have a longer length than $L$, giving surjectivity.
		\end{proof}
		
		\begin{example}[The following example is taken from \cite{ES2}]
			Let $n=9$, and consider the dominant weight $\lambda=\epsilon_3+3\epsilon_4+3\epsilon_5+6\epsilon_6+8\epsilon_7+8\epsilon_8+8\epsilon_9$.  We draw the associated arc diagram below:
			$$
			
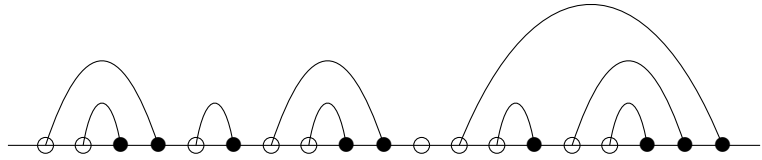
\begin{tikzpicture}
				\draw (0,0) -- (10,0);
				\draw (0.5,0) circle (3pt);
				\draw (1,0) circle (3pt);
				\draw (1.5,0) node[label=center:{\large $\bullet$}] {};
				\draw (1,0) .. controls (1.125,0.75) and (1.375,0.75) .. (1.5,0);
				\draw (2,0) node[label=center:{\large $\bullet$}] {};
				\draw (0.5,0) .. controls (1,1.5) and (1.5,1.5) .. (2,0);
				\draw (2.5,0) circle (3pt);
				\draw (3,0) node[label=center:{\large $\bullet$}] {};
				\draw (2.5,0) .. controls (2.625,0.75) and (2.875,0.75) .. (3,0);
				\draw (3.5,0) circle (3pt);
				\draw (4,0) circle (3pt);
				\draw (4.5,0) node[label=center:{\large $\bullet$}] {};
				\draw (4,0) .. controls (4.125,0.75) and (4.375,0.75) .. (4.5,0);
				\draw (5,0) node[label=center:{\large $\bullet$}] {};
				\draw (3.5,0) .. controls (4,1.5) and (4.5,1.5) .. (5,0);
				\draw (5.5,0) circle (3pt);
				\draw (6,0) circle (3pt);
				\draw (6.5,0) circle (3pt);
				\draw (7,0) node[label=center:{\large $\bullet$}] {};
				\draw (6.5,0) .. controls (6.625,0.75) and (6.875,0.75) .. (7,0);
				\draw (7.5,0) circle (3pt);
				\draw (8,0) circle (3pt);
				\draw (8.5,0) node[label=center:{\large $\bullet$}] {};
				\draw (8,0) .. controls (8.125,0.75) and (8.375,0.75) .. (8.5,0);
				\draw (9,0) node[label=center:{\large $\bullet$}] {};
				\draw (7.5,0) .. controls (8,1.5) and (8.5,1.5) .. (9,0);
				\draw (9.5,0) node[label=center:{\large $\bullet$}] {};
				\draw (6,0) .. controls (7,2.5) and (8.5,2.5) .. (9.5,0);
			\end{tikzpicture}
			$$
			We see that there are 4 maximal arcs in the diagram, thus $\DS_xL(\lambda)$ has 4 simple factors $\Pi L(\lambda_1)$, $L(\lambda_2)$, $L(\lambda_3)$, and $L(\lambda_4)$.  They are listed with their arc diagrams below, along with the corresponding value of $z$: $\lambda_1=2\epsilon_2+4\epsilon_3+4\epsilon_4+7\epsilon_5+9\epsilon_6+9\epsilon_7+9\epsilon_8$, $z=7$:
			$$
			
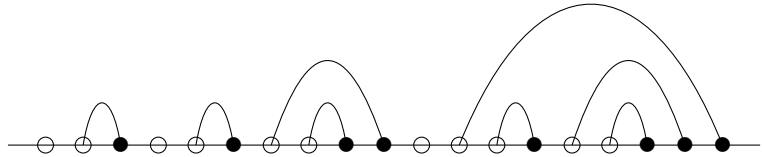
\begin{tikzpicture}
				\draw (0,0) -- (10,0);
				\draw (0.5,0) circle (3pt);
				\draw (1,0) circle (3pt);
				\draw (1.5,0) node[label=center:{\large $\bullet$}] {};
				\draw (1,0) .. controls (1.125,0.75) and (1.375,0.75) .. (1.5,0);
				\draw (2,0) circle (3pt);
				\draw (2.5,0) circle (3pt);
				\draw (3,0) node[label=center:{\large $\bullet$}] {};
				\draw (2.5,0) .. controls (2.625,0.75) and (2.875,0.75) .. (3,0);
				\draw (3.5,0) circle (3pt);
				\draw (4,0) circle (3pt);
				\draw (4.5,0) node[label=center:{\large $\bullet$}] {};
				\draw (4,0) .. controls (4.125,0.75) and (4.375,0.75) .. (4.5,0);
				\draw (5,0) node[label=center:{\large $\bullet$}] {};
				\draw (3.5,0) .. controls (4,1.5) and (4.5,1.5) .. (5,0);
				\draw (5.5,0) circle (3pt);
				\draw (6,0) circle (3pt);
				\draw (6.5,0) circle (3pt);
				\draw (7,0) node[label=center:{\large $\bullet$}] {};
				\draw (6.5,0) .. controls (6.625,0.75) and (6.875,0.75) .. (7,0);
				\draw (7.5,0) circle (3pt);
				\draw (8,0) circle (3pt);
				\draw (8.5,0) node[label=center:{\large $\bullet$}] {};
				\draw (8,0) .. controls (8.125,0.75) and (8.375,0.75) .. (8.5,0);
				\draw (9,0) node[label=center:{\large $\bullet$}] {};
				\draw (7.5,0) .. controls (8,1.5) and (8.5,1.5) .. (9,0);
				\draw (9.5,0) node[label=center:{\large $\bullet$}] {};
				\draw (6,0) .. controls (7,2.5) and (8.5,2.5) .. (9.5,0);
			\end{tikzpicture}
			$$
			$\lambda_2=4\epsilon_3+4\epsilon_4+7\epsilon_5+9\epsilon_6+9\epsilon_7+9\epsilon_8$, $z=6$:
			$$
			
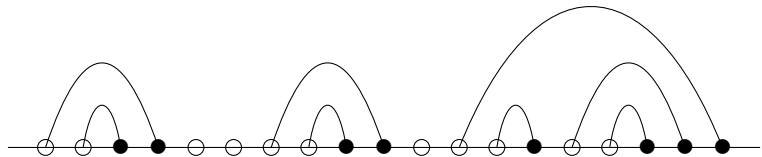
\begin{tikzpicture}
				\draw (0,0) -- (10,0);
				\draw (0.5,0) circle (3pt);
				\draw (1,0) circle (3pt);
				\draw (1.5,0) node[label=center:{\large $\bullet$}] {};
				\draw (1,0) .. controls (1.125,0.75) and (1.375,0.75) .. (1.5,0);
				\draw (2,0) node[label=center:{\large $\bullet$}] {};
				\draw (0.5,0) .. controls (1,1.5) and (1.5,1.5) .. (2,0);
				\draw (2.5,0) circle (3pt);
				\draw (3,0) circle (3pt);
				\draw (3.5,0) circle (3pt);
				\draw (4,0) circle (3pt);
				\draw (4.5,0) node[label=center:{\large $\bullet$}] {};
				\draw (4,0) .. controls (4.125,0.75) and (4.375,0.75) .. (4.5,0);
				\draw (5,0) node[label=center:{\large $\bullet$}] {};
				\draw (3.5,0) .. controls (4,1.5) and (4.5,1.5) .. (5,0);
				\draw (5.5,0) circle (3pt);
				\draw (6,0) circle (3pt);
				\draw (6.5,0) circle (3pt);
				\draw (7,0) node[label=center:{\large $\bullet$}] {};
				\draw (6.5,0) .. controls (6.625,0.75) and (6.875,0.75) .. (7,0);
				\draw (7.5,0) circle (3pt);
				\draw (8,0) circle (3pt);
				\draw (8.5,0) node[label=center:{\large $\bullet$}] {};
				\draw (8,0) .. controls (8.125,0.75) and (8.375,0.75) .. (8.5,0);
				\draw (9,0) node[label=center:{\large $\bullet$}] {};
				\draw (7.5,0) .. controls (8,1.5) and (8.5,1.5) .. (9,0);
				\draw (9.5,0) node[label=center:{\large $\bullet$}] {};
				\draw (6,0) .. controls (7,2.5) and (8.5,2.5) .. (9.5,0);
			\end{tikzpicture}
			$$
			$\lambda_3=\epsilon_3+3\epsilon_4+7\epsilon_5+9\epsilon_6+9\epsilon_7+9\epsilon_8$, $z=4$:
			$$
			\begin{tikzpicture}
				\draw (0,0) -- (10,0);
				\draw (0.5,0) circle (3pt);
				\draw (1,0) circle (3pt);
				\draw (1.5,0) node[label=center:{\large $\bullet$}] {};
				\draw (1,0) .. controls (1.125,0.75) and (1.375,0.75) .. (1.5,0);
				\draw (2,0) node[label=center:{\large $\bullet$}] {};
				\draw (0.5,0) .. controls (1,1.5) and (1.5,1.5) .. (2,0);
				\draw (2.5,0) circle (3pt);
				\draw (3,0) node[label=center:{\large $\bullet$}] {};
				\draw (2.5,0) .. controls (2.625,0.75) and (2.875,0.75) .. (3,0);
				\draw (3.5,0) circle (3pt);
				\draw (4,0) circle (3pt);
				\draw (4.5,0) node[label=center:{\large $\bullet$}] {};
				\draw (4,0) .. controls (4.125,0.75) and (4.375,0.75) .. (4.5,0);
				\draw (5,0) circle (3pt);
				\draw (5.5,0) circle (3pt);
				\draw (6,0) circle (3pt);
				\draw (6.5,0) circle (3pt);
				\draw (7,0) node[label=center:{\large $\bullet$}] {};
				\draw (6.5,0) .. controls (6.625,0.75) and (6.875,0.75) .. (7,0);
				\draw (7.5,0) circle (3pt);
				\draw (8,0) circle (3pt);
				\draw (8.5,0) node[label=center:{\large $\bullet$}] {};
				\draw (8,0) .. controls (8.125,0.75) and (8.375,0.75) .. (8.5,0);
				\draw (9,0) node[label=center:{\large $\bullet$}] {};
				\draw (7.5,0) .. controls (8,1.5) and (8.5,1.5) .. (9,0);
				\draw (9.5,0) node[label=center:{\large $\bullet$}] {};
				\draw (6,0) .. controls (7,2.5) and (8.5,2.5) .. (9.5,0);
			\end{tikzpicture}
			$$
			$\lambda_4=\epsilon_3+3\epsilon_4+3\epsilon_5+6\epsilon_6+8\epsilon_7+8\epsilon_8$, $z=0$:
			$$
			\begin{tikzpicture}
				\draw (0,0) -- (10,0);
				\draw (0.5,0) circle (3pt);
				\draw (1,0) circle (3pt);
				\draw (1.5,0) node[label=center:{\large $\bullet$}] {};
				\draw (1,0) .. controls (1.125,0.75) and (1.375,0.75) .. (1.5,0);
				\draw (2,0) node[label=center:{\large $\bullet$}] {};
				\draw (0.5,0) .. controls (1,1.5) and (1.5,1.5) .. (2,0);
				\draw (2.5,0) circle (3pt);
				\draw (3,0) node[label=center:{\large $\bullet$}] {};
				\draw (2.5,0) .. controls (2.625,0.75) and (2.875,0.75) .. (3,0);
				\draw (3.5,0) circle (3pt);
				\draw (4,0) circle (3pt);
				\draw (4.5,0) node[label=center:{\large $\bullet$}] {};
				\draw (4,0) .. controls (4.125,0.75) and (4.375,0.75) .. (4.5,0);
				\draw (5,0) node[label=center:{\large $\bullet$}] {};
				\draw (3.5,0) .. controls (4,1.5) and (4.5,1.5) .. (5,0);
				\draw (5.5,0) circle (3pt);
				\draw (6,0) circle (3pt);
				\draw (6.5,0) circle (3pt);
				\draw (7,0) node[label=center:{\large $\bullet$}] {};
				\draw (6.5,0) .. controls (6.625,0.75) and (6.875,0.75) .. (7,0);
				\draw (7.5,0) circle (3pt);
				\draw (8,0) circle (3pt);
				\draw (8.5,0) node[label=center:{\large $\bullet$}] {};
				\draw (8,0) .. controls (8.125,0.75) and (8.375,0.75) .. (8.5,0);
				\draw (9,0) node[label=center:{\large $\bullet$}] {};
				\draw (7.5,0) .. controls (8,1.5) and (8.5,1.5) .. (9,0);
				\draw (9.5,0) circle (3pt);
			\end{tikzpicture}
			$$
		\end{example}

		\subsection{Exceptional cases} We now explain the case of the exceptional Lie superalgebras $G(3),F(4),$ or $D(2|1;a)$.  These all have defect one, and their atypical blocks have one of the following extension graphs: 
		\[
		\xymatrix{
			A_\infty^\infty: & \cdots \ar@{-}[r] & L^{-1} \ar@{-}[r] & L^0 \ar@{-}[r] & L^{1} \ar@{-}[r] & \cdots 
		}
		\]
		\[
		\xymatrix{D_\infty: & L^{1} \ar@{-}[r] & L^2 \ar@{-}[r] & L^{3} \ar@{-}[r] & L^4 \ar@{-}[r] & \cdots \\
			&  & L^0\ar@{-}[u] & & & }
		\]
		
		Specifically, $A_{\infty}^{\infty}$ will be the extension graph for certain blocks of $F(4)$ and for blocks of $D(2|1;a)$ when $a\in\mathbb{Q}$.  On the other hand $D_{\infty}$ will be the extension graph for all blocks of $G(3)$, along with certain blocks of $F(4)$ and $D(2|1;a)$.
		
		\begin{remark}
			Note that it is clear that the above extension graphs are bipartite.  Extension graphs are in fact always bipartite for basic classical Lie superalgebras, as was hinted in Remark~\ref{purity_rmk} and is shown in \cite{G4}.
		\end{remark}
		
		The following lemma is determined from the full relations on the extension graphs for each block, which are described in \cite{Ger} and \cite{M}.
		\begin{lemma}\label{atyp 1 projs}
			Let $P(L^i)$ be the projective indecomposable cover of a simple, non-projective module $L^i$ over an exceptional Lie superalgebra. Then the radical and socle filtrations of $P(L^i)$ coincide, with socle and cosocle isomorphic to $L$, and middle layer isomorphic to 
			\[
			\bigoplus\limits_{j\in Adj(i)}L^j
			\]
			where $Adj(j)$ denotes the vertices adjacent to $i$ in the extension graph containing $L^i$.  
		\end{lemma}
		
		Using the above lemma, we get the following.
		
		\begin{proposition}\label{atyp 1 blocks ds}
			Let $\mathcal{B}$ be an atypical block for an exceptional Lie superalgebra $\gg$, and let $x\in\gg_{1}$ be non-zero such that $[x,x]=0$.  Suppose that for some simple module $L$ in $\mathcal{B}$, $L_x$ is pure.  Then $L_x$ is pure for all simples $L$ in $\mathcal{B}$.  Further we have the following isomorphisms of $\gg_x$-modules:
			\[
			\hspace{-12.25em} \Ext(\mathcal{B})=A_{\infty}^\infty: \ \ \  L^i_x\cong\Pi^iL^0_x
			\]
			\[
			\Ext(\mathcal{B})=D_{\infty}: \ \ \  L^0_x\cong L^1_x, \ \ \ L^i_x\cong \Pi^{i-1} (L_x^0)^{\oplus 2} \text{ for }i\geq 2.
			\]
			By $\Ext(\mathcal{B})$ we denote the extension graph of $\mathcal{B}$.
		\end{proposition}
		\begin{proof}
			
			For any simple module $L^i$ in $\mathcal{B}$ we have a short exact sequence
			\[
			0\to M\to P(L^i) \to L^i \to 0,
			\]
			where $P(L^i)$ denotes the projective cover of $L^i$ and $M$ is its radical.  By Lemma~\ref{reduced} and that $P(L^i)_x=0$, we find that $M_x\cong \Pi L_x^i$.  On the other hand by Lemma~\ref{atyp 1 projs}, we have the short exact sequence
			\[
			0\to L^i\to M\to\ \bigoplus\limits_{j\in Adj(i)}L^j\to 0
			\]
			Suppose that $L^i_x$ is pure, so that $\Hom(L^i_x,M_x)=0$.  Then we obtain by Lemma~\ref{reduced} the short exact sequence (using that $M_x\cong \Pi L_x^i$):
			\[
			0\to \Pi L_x^i\to\bigoplus\limits_{j\in Adj(i)}L^j_x\to\Pi L_x^i\to0
			\]
			Using connectedness of the extension graph, purity of $L_x^j$ for any $j$ easily follows, along with the formulas in the case of each type of block.
		\end{proof}

		In order to complete the description of the DS functor for exceptional  Lie superalgebras, we need to compute the value of the $DS$ functor on one module in every atypical block, and in particular check that it is pure so that Proposition~\ref{atyp 1 blocks ds} will apply.
		
		Thus let $\gg$ be one of the Lie superalgebras $D(2|1;a), G(3), F(4).$
		Let $\hh$ be a Cartan subalgebra of $\gg_0$.
		We denote by  $W$  the Weyl group of $\gg_0$ and  by $(-|-)$ the symmetric non-degenerate form on $\hh^*$ which is induced by a non-degenerate invariant
		form on $\gg$.
		
		Let $\Sigma$ be a base of $\gg$ which contains an isotropic root $\beta$.
		Fix a non-zero $x\in\gg_{\beta}$.  Set
		${\Delta}_x:=(\beta^{\perp}\cap{\Delta})\setminus\{\beta,-\beta\}$.
		By Proposition~\ref{lm202},  ${\gg}_x$ can be identified
		with a subalgebra of ${\gg}$ generated by the root spaces
		$\gg_{\alpha}$ with $\alpha\in{\Delta}_x$ and a Cartan subalgebra
		${\hh}_x\subset\hh$. If ${\Delta}_x$ is not empty, then
		${\Delta}_x$ is the root system of the Lie superalgebra ${\gg}_x$ and one can choose
		${\Sigma}_x$ in $\Delta_x$ such that 
		$\Delta^+({\Sigma}_x)=\Delta^+\cap {\Delta}_x$.
		For 
		$\gg=D(2|1;a),G(3),F(4)$  one has   ${\gg}_x=\mathbb{C},\mathfrak{sl}_2,\mathfrak{sl}_3$ respectively.
		
		\begin{lemma}\label{tame0}
			Let $L:=L(\lambda)$ be a finite-dimensional module and $(\lambda|\beta)=0$.
			Set $L':=L_{\gg_x}(\lambda|_{\hh_x})$. One has
			$$\DS_x(L)\cong \left\{\begin{array}{lll}
				L' &\text{ for } & G(3)\\
				L' &\text{ for } & D(2|1;a),\ F(4) \text{ if }\  L'\cong (L')^*\\
				L'\oplus (L')^*&\text{ for } & D(2|1;a),\ F(4) \text{ if }\  L'\not\cong (L')^*.
			\end{array}\right.$$
		\end{lemma} 
		\begin{proof}
			It is easy to see that $[\DS_x(L):L']=1$.     
			Set $\lambda':=\lambda|_{\hh_x}$. 
			By Section~\ref{subsec_preimage_pullback_map},
			$\DS_x(L)$ is a typical module and each simple subquotient of $\DS_x(L)$
			is of the form
			$L_{\gg_x}(\nu)$ with $\nu\in\{\lambda', \sigma(\lambda')\}$, where $\sigma=\text{Id}$
			for $\gg=G(3)$, $\sigma=-\text{Id}$ for $D(2|1;a)$, and
			$\sigma$ is the Dynkin diagram automorphism of $\gg_x=\mathfrak{sl}_3$ in $F(4)$-case. 
			This gives the first formula.
			For $D(2|1;a)$, $F(4)$ one has $L_{\gg_x}(\nu)^*\cong L_{\gg_x}(\sigma(\nu))$;
			giving the second formula. Finally in $D(2|1;a), F(4)$
			the Weyl group contains $-\text{Id}$, so $L\cong L^*$ and thus
			$\DS_x(L)\cong \DS_x(L^*)$ by Lemma~\ref{tensor}, implying the third formula.
		\end{proof}

		We fix a triangular decomposition of $\gg_0$ and denote by $\Delta^+_0$ the corresponding set of positive roots. We consider 
		all bases $\Sigma$ for $\Delta$ which satisfy
		$\Delta^+_0\subset \Delta^+(\Sigma)$.
		We say that an isotropic root $\beta$ is of the {\em first type}
		if $\beta$ lies in  a base $\Sigma$ with $\Delta^+_0\subset \Delta^+(\Sigma)$.

		Take any base $\Sigma$ as above and denote by $\rho$ the corresponding Weyl vector.
		It is easy to see that a simple atypical module $L=L(\nu)$  satisfies the assumptions of Lemma~\ref{tame0} for some $\Sigma'$ and $\beta\in\Sigma'$ if and only if $\nu+\rho$ is orthogonal
		to an isotropic root of the first type.

		Let $\mathcal{B}$ be  an atypical block of $\gg$.  We call the block containing the trivial module  $L(0)$
		a {\em principal block}. Clearly, $\DS_x(L(0))$ is the trivial $\gg_x$-module,
		so Proposition~\ref{atyp 1 blocks ds} gives  $\DS_x(L)$ for each simple module $L$ in $\mathcal{B}_0)$.

		Combining Proposition~\ref{atyp 1 blocks ds} and Lemma~\ref{tame0}, we see that in order to compute $\DS_x(L)$ for each  simple $L$ in $\mathcal{B}$, it is enough to find
		$L(\nu)\in\text{Irr}(\mathcal{B})$ such that $\nu+\rho$ is orthogonal to  an isotropic root of the first type.
		Below we will list such $\nu$ for each non-principal atypical block for $D(2|1;a)$, $F(4)$ and $G(3)$.

		%
		%

		\subsubsection{Case $D(2|1;a)$}\label{D21a}
		For $\gg:=D(2|1;a)$ one has $\gg_x=\mathbb{C}$.
		The atypical blocks were described in \cite{Ger}, Thm. 3.1.1.

		The extension graph of the principal block $\mathcal{B}_0$ is $D_{\infty}$, so for a simple $L^i$ in $\mathcal{B}_0$ we have
		$\DS_x(L^i)=\mathbb{C}$
		for $i=0,1$ and $\DS_x(L^i)=\Pi^{i-1}(\mathbb{C})^{\oplus 2}$
		for $i>1$ (where $\mathbb{C}$ stands for the trivial even $\gg_x$-module).

		If $a$ is irrational, the principal block is
		the only atypical block in $\cF(\gg)$. Consider the case when $a$ is rational.
		Recall that  $\hh^*$ has an orthogonal basis
		$\{\varepsilon_1,\varepsilon_2,\varepsilon_3\}$ with 
		$$(\varepsilon_1,\varepsilon_1)=-\frac{1+a}{2},\ \ \ 
		(\varepsilon_2,\varepsilon_2)=\frac{1}{2},\ \ \ (\varepsilon_3,\varepsilon_3)=\frac{a}{2}.$$
		One has
		$$D(2|1;1)=\mathfrak{osp}(4|2),\ \ D(2|1;a)\cong D(2|1; -1-a)\cong D(2|1; a^{-1})$$
		so we can assume that  $0<a<1$ and write  $a=\frac{p}{q}$, where $p,q$ are relatively prime positive integers. 
		
		The atypical blocks  are $\mathcal{B}_k$ for  $k\in\mathbb{N}$ 
		(the principal block is $\mathcal{B}_0$). Consider the block $\mathcal{B}_k$ with 
		$k>0$. The extension graph of $\mathcal{B}_k$ is
		$A^{\infty}_{\infty}$. By \cite{Ger}, Thm. 3.1.1, the block 
		$\mathcal{B}_k$ contains a simple module $L$ with the highest weight $\lambda_{k;0}$
		satisfying $(\lambda_{k;0}+\rho|\beta)=0$ for 
		$$\beta:=\varepsilon_1+\varepsilon_2-\varepsilon_3.$$
		Taking $x\in\gg_{\beta}$ we can
		identify $\gg_x$ with $\mathbb{C}h$ for
		$h:=\varepsilon_1^*-\varepsilon_2^*$ (where $\varepsilon_1^*,\varepsilon_2^*,\varepsilon_3^*$ is the dual basis in $\hh$). 
		By Lemma~\ref{tame0} we get
		$$\DS_x(L)=L_{\gg_x}(k)\oplus L_{\gg_x}(-k),$$
		where $L_{\gg_x}(u)$ stands for the even one-dimensional $\gg_x$-module
		with $h$ acting by $k(p+q)$.
		By Proposition~\ref{atyp 1 blocks ds}, $\DS_x(L^i)\cong \Pi^i(\DS_x(L))$ 
		for each $L^i\in\text{Irr}(\mathcal{B}_k)$ (for $k>0$).

		\subsubsection{Case $G(3)$}
		For $\gg:=G(3)$ the atypical blocks were described in Thm. 4.1.1 of \cite{Ger}.
		The atypical blocks in $\cF(\gg)$ are $\mathcal{B}_k$ for  $k\in\mathbb{N}$;
		the extension graphs are  $D_{\infty}$. The block 
		$\mathcal{B}_k$ contains a simple module with the highest weight $\lambda_{k;0}$
		satisfying $(\lambda_{k;0}+\rho|\beta)=0$ for 
		$$\beta:=-\varepsilon_1+\delta.$$
		Taking $\Sigma:=\{\delta-\varepsilon_1,\varepsilon_2-\delta,\delta\}$ and 
		$x\in\gg_{\beta}$ we can
		identify $\gg_x$ with
		$\mathfrak{sl}_2$-triple corresponding to the root $\alpha=\varepsilon_1+2\varepsilon_2$.
		One has $\lambda_{k;0}=k\alpha$.
		Combining Lemma~\ref{tame0} and Proposition~\ref{atyp 1 blocks ds} we get
		$$\DS_x(L^0)\cong \DS_x(L^1)\cong L_{\mathfrak{sl}_2}(2k),\ \ \ 
		\DS_x(L^i)=\Pi^{i-1}(L_{\mathfrak{sl}_2}(2k))^{\oplus 2}\ 
		\text{ for }i>1.$$

		\subsubsection{Case $F(4)$}
		For $\gg:=F(4)$ we have $\gg_x\cong \mathfrak{sl}_3$. 
		The integral weight lattice is spanned by $\varepsilon_1$, $\varepsilon_2$,
		$\frac{1}{2}(\varepsilon_1+\varepsilon_2+\varepsilon_3)$ and $\frac{1}{2}\delta$;
		the parity is given by 
		$p(\frac{\varepsilon_i}{2})=0$ and $p(\frac{\delta}{2})=1$.

		The atypical blocks are described
		in Thm. 2.1 of \cite{M}. These blocks are
		parametrized by the pairs $(m_1,m_2)$, where
		$m_1,m_2\in \mathbb{N}$, $m_1\geq m_2$, and $m_1-m_2$ is divisible by 3.
		We denote the corresponding block by $\mathcal{B}_{(m_1;m_2)}$.

		The extension
		graph of $\mathcal{B}_{(i;i)}$ is $D_{\infty}$; the block
		$\mathcal{B}_{(0;0)}$ is principal.
		For $i>0$ the block 
		$\mathcal{B}_{(i;i)}$  contains a simple module $L(\lambda)$ with
		$$\lambda+\rho=(i+1)(\varepsilon_1+\varepsilon_2)-\beta_1, \ \text{ where }
		\beta_1:=
		\frac{1}{2}(-\varepsilon_1+\varepsilon_2-\varepsilon_3+\delta).$$
		One has
		$(\lambda+\rho|\beta_1)=0$. Take $x\in \gg_{\beta_1}$ and consider the base
		$$\Sigma_1:=\{\beta_1; \frac{1}{2}(\varepsilon_1+\varepsilon_2-\varepsilon_3-\delta); 
		\varepsilon_3; \varepsilon_1-\varepsilon_2\}.
		$$
		Then
		$\gg_x$ can be identified with $\mathfrak{sl}_3$ corresponding to the set of simple roots
		$\ \{\varepsilon_2+\varepsilon_3;\varepsilon_1-\varepsilon_3\}$ and Lemma~\ref{tame0} gives
		$$\DS_x(L(\lambda))=L_{\mathfrak{sl}_3}(i\omega_1+i\omega_2),$$
		where $\omega_1,\omega_2$  are the fundamental weights of $\mathfrak{sl}_3$.
		By Proposition~\ref{atyp 1 blocks ds} we get for the simple module $L^j$ in $\mathcal{B}_{(i;i)}$:
		$$
		\DS_x(L^0)\cong \DS_x(L^1)\cong L_{\mathfrak{sl}_3}(i\omega_1+i\omega_2),\ \ 
		\DS_x(L^j)\cong \Pi^{j-1}(L_{\mathfrak{sl}_3}(i\omega_1+i\omega_2))^{\oplus 2}\
		\text{ for } j>1.$$ 
		
		Consider a   block  $\mathcal{B}_{(i_1;i_2)}$  for $i_1\not=i_2$.
		The extension
		graph of this block is $A_{\infty}^{\infty}$  and this
		block  contains a simple module $L:=L(\lambda')$ with
		$$\lambda'+\rho=i_1\varepsilon_1+i_2\varepsilon_2+(i_1-i_2)\varepsilon_3.$$
		In particular, $(\lambda'+\rho|\beta_2)=0$ for
		$\beta_2:=\frac{1}{2}(-\varepsilon_1+\varepsilon_2+\varepsilon_3+\delta)$.
		Taking  $x\in\gg_{\beta_1}$ and
		$$\Sigma_2:=\{\beta_2; \varepsilon_2-\varepsilon_3;  -\beta_1; \frac{1}{2}(\varepsilon_1-\varepsilon_2-\varepsilon_3+\delta)\}$$
		we identify
		$\gg_x$ with $\mathfrak{sl}_3$ corresponding to the set of simple roots
		$\{\varepsilon_2-\varepsilon_3;\varepsilon_1+\varepsilon_3\}$. Combining Lemma~\ref{tame0} 
		and Proposition~\ref{atyp 1 blocks ds} we get 
		$$\DS_x(L)=L_{\mathfrak{sl}_3}(i_1\omega_1+i_2\omega_2)\oplus 
		L_{\mathfrak{sl}_3}(i_2\omega_1+i_1\omega_2),\ \ \DS_x(L^i)\cong \Pi^i(\DS_x(L))$$
		for each $L^i$ in the block $\mathcal{B}_{(i_1;i_2)}$.
		
		\section{Declarations}
		
		\subsection{Funding}
		M.G. was supported by  ISF Grant 1957/21. C.H. was supported by ISF Grant 1221/17. V.S. was supported by NSF Grant 2001191. A.S. was   supported by ISF Grant 711/18 and NSF-BSF Grant 2019694. 
		
		\subsection{Ethical statement:} The authors have no conflicts of interest.

	\end{document}